\documentclass[12pt,american]{article}
\usepackage[T1]{fontenc}
\usepackage[utf8]{inputenc}
\usepackage{babel}
\usepackage{mathrsfs}
\usepackage{bm}
\usepackage{amsmath}
\usepackage{amsthm}
\usepackage{amssymb}
\usepackage{stackrel}
\usepackage{geometry}
\usepackage{rotfloat}
\usepackage{setspace}
\usepackage[authoryear]{natbib}
\usepackage[bookmarks=false,
 breaklinks=false,pdfborder={0 0 1},backref=section,colorlinks=false]
 {hyperref}
\hypersetup{
 colorlinks,linkcolor=magenta,citecolor=blue,pagebackref=true}

\makeatletter

\providecommand{\tabularnewline}{\\}

\newcommand{\lyxaddress}[1]{
	\par {\raggedright #1
	\vspace{1.4em}
	\noindent\par}
}
\theoremstyle{plain}
\newtheorem{thm}{\protect\theoremname}
\theoremstyle{plain}
\newtheorem{lem}[thm]{\protect\lemmaname}
\theoremstyle{remark}
\newtheorem{rem}[thm]{\protect\remarkname}
\theoremstyle{plain}
\newtheorem{prop}[thm]{\protect\propositionname}
\theoremstyle{plain}
\newtheorem{cor}[thm]{\protect\corollaryname}

\usepackage{babel}

\makeatother

\providecommand{\corollaryname}{Corollary}
\providecommand{\lemmaname}{Lemma}
\providecommand{\propositionname}{Proposition}
\providecommand{\remarkname}{Remark}
\providecommand{\theoremname}{Theorem}

\begin{document}
\title{Modifications of the BIC for order selection in finite mixture models}
\author{Hien Duy Nguyen$^{1,2}$ and TrungTin Nguyen$^{3,4}$}
\maketitle

\lyxaddress{}

\lyxaddress{\begin{center}
$^{1}$Department of Mathematics and Physical Science, La Trobe University,
Melbourne, Australia.\\
 $^{2}$Institute of Mathematics for Industry, Kyushu University,
Fukuoka, Japan.\\
 $^{3}$ARC Centre of Excellence for the Mathematical Analysis of
Cellular Systems.\\
 $^{4}$School of Mathematical Sciences, Queensland University of
Technology, Brisbane, Australia. 
\par\end{center}}
\begin{abstract}
Finite mixture models are ubiquitous in modern statistical modeling,
and a recurring practical issue is choosing the model order. In \citet[Sankhy\=a Series A, \textbf62, pp. 49--66]{keribin2000consistent},
the Bayesian information criterion (BIC) was proved consistent in
mixtures, but under strong regularity, including high moments and
high-order derivatives of the component density. We introduce the
$\nu$-BIC and $\epsilon$-BIC, which weight the BIC penalty by negligibly
small logarithmic factors immaterial in practice. This minor modification
yields consistency under substantially weaker conditions, without
differentiability and with mild moment assumptions, and we also give
a misspecification result: when the truth lies outside the candidate
family, any vanishing-penalty IC eventually selects a Kullback--Leibler
optimal order among candidates. Finally, we clarify two limitations
of consistent IC-based selection in mixtures: there is no universally
minimal BIC-scale penalty within our sufficient conditions, and order
consistency can conflict with minimax optimality in Hellinger risk.
We illustrate the theory for Gaussian mixtures, non-differentiable
Laplace mixtures, heavy-tailed $t$-mixtures, and mixtures of regression
models. 
\end{abstract}

\section{Introduction}

\label{sec:Introduction}

Finite mixtures are a natural and ubiquitous class of models for modeling
heterogeneous populations consisting of distinct subpopulations. Multiple
volumes have been written regarding the application, analysis, and
implementation of such models, including the volumes of \citet{titterington1985statistical},
\citet{lindsay1995mixture}, \citet{peel2000finite}, \citet{fruhwirth2019handbook},
\citet{ng2019mixture}, \citet{chen2023statistical}, and \citet{yao2024mixture},
among many others. When conducting mixture modeling, a frequently
required task is that of order selection. In this text, we will take
an information criteria (IC) approach following the tradition of \citet{akaike1974new}
as espoused in the texts of \citet{anderson2004model}, \citet{claeskens2008model},
and \citet{konishi2008information}. We characterise the task of order
selection and the IC method as follows.

Let $\left(\Omega,{\cal A},\mathrm{P}\right)$ be a probability space
with expectation operator $\mathrm{E}$ and suppose that $X:\Omega\rightarrow\mathbb{X}$
is a random map with density function $f_{0}:\mathbb{X}\to\mathbb{R}_{>0}$
with respect to dominating measure $\mathfrak{m}$, in the sense that
$\mathrm{P}\left(X\in\mathbb{A}\right)=\int_{\mathbb{A}}f_{0}\mathrm{d}\mathfrak{m}$,
for each $\mathbb{A}\in{\cal B}\left(\mathbb{X}\right)$ (the Borel
$\sigma$-algebra of $\mathbb{X}$). Let $\phi:\mathbb{X}\times\mathbb{T}\to\mathbb{R}_{>0}$
be a probability density function (PDF) dependent on parameter $\theta=\left(\vartheta_{1},\dots,\vartheta_{m}\right)\in\mathbb{T}\subset\mathbb{R}^{m}$
in the sense that $\int_{\mathbb{X}}\phi\left(x;\theta\right)\mathfrak{m}\left(\mathrm{d}x\right)=1$,
for each $\theta$, and define the $k$-component mixture of $\phi$-densities
as 
\[
{\cal M}_{k}^{\phi}=\left\{ \mathbb{X}\ni x\mapsto f_{k}\left(x;\psi_{k}\right)=\sum_{z=1}^{k}\pi_{z}\phi\left(x;\theta_{z}\right):\theta_{z}\in\mathbb{T},\pi_{z}\in\left[0,1\right],\sum_{z=1}^{k}\pi_{z}=1,z\in\left[k\right]\right\} ,
\]
where $\left[k\right]=\left\{ 1,\dots,k\right\} $ and $\psi_{k}=\left(\pi_{1},\dots,\pi_{k},\theta_{1},\dots,\theta_{k}\right)\in\mathbb{S}_{k}\subset\mathbb{R}^{\left(1+m\right)k}$
(the parameter space of densities in ${\cal M}_{k}^{\phi}$). Here,
$\phi$ is referred to as the component PDF.

We suppose that $f_{0}\in{\cal M}_{k_{0}}^{\phi}$ but $f_{0}\notin{\cal M}_{k_{0}-1}^{\phi}$
for some $k_{0}\in\left[\bar{k}\right]$, and some sufficiently large
$\bar{k}\in\mathbb{N}$. Here, the extra stipulation that $f_{0}\notin{\cal M}_{k_{0}-1}^{\phi}$
is required because $\left({\cal M}_{k}^{\phi}\right)_{k\in\left[\bar{k}\right]}$
is a nested sequence in the sense that ${\cal M}_{k}^{\phi}\subset{\cal M}_{k+1}^{\phi}$,
and thus if $f_{0}\in{\cal M}_{k_{0}}^{\phi}$, it also holds that
$f_{0}\in{\cal M}_{k_{0}+1}^{\phi}$, therefore the condition alone
does not uniquely define $k_{0}$, which we refer to as the true order
of $f_{0}$. We may interpret $k_{0}$ as the order of the smallest
(or most parsimonious) model ${\cal M}_{k}^{\phi}$ that contains
$f_{0}$. Let $\ell_{k}\left(\psi_{k}\right)=-\mathrm{E}\left\{ \log f_{k}\left(X;\psi_{k}\right)\right\} $
denote the average negative log-density of $f_{k}\left(\cdot;\psi_{k}\right)\in{\cal M}_{k}^{\phi}$,
for each $k\in\left[\bar{k}\right]$. Then, by virtue of the divergence
property of the Kullback--Leibler divergence (cf. e.g., \citealp{csiszar1995generalized}),
it also holds that 
\[
k_{0}=\min\underset{k\in\left[\bar{k}\right]}{\arg\min}\left\{ \min_{\psi_{k}\in\mathbb{S}_{k}}\ell_{k}\left(\psi_{k}\right)\right\} ,
\]
assuming that the minimum is well-defined in each case, which is a
useful characterisation that we will exploit in the sequel.

Suppose that we observe a data set of $n$ independent and identically
distributed (IID) replicates of $X$: $\mathbf{X}_{n}=\left(X_{1},\dots,X_{n}\right)\subset\mathbb{X}$.
Then, the problem of order selection can be characterised as the construction
of an estimator $\hat{k}_{n}=\hat{k}\left(\mathbf{X}_{n}\right)$
of $k_{0}$ with desirable properties. In particular, we want $\hat{k}_{n}$
to be consistent in the sense that 
\[
\lim_{n\to\infty}\mathrm{P}\left(\hat{k}_{n}=k_{0}\right)=1.
\]
If we let $\ell_{k,n}\left(\psi_{k}\right)=-n^{-1}\sum_{i=1}^{n}\log f_{k}\left(X_{i};\psi_{k}\right)$
denote the empirical average negative log-likelihood of $f_{k}\left(\cdot;\psi_{k}\right)$,
the natural estimator of $\ell_{k}\left(\psi_{k}\right)$, then the
method of IC suggests that we construct the estimator of $k_{0}$
in the form 
\[
\hat{k}_{n}=\min\underset{k\in\left[\bar{k}\right]}{\arg\min}\left\{ \min_{\psi_{k}\in\mathbb{S}_{k}}\ell_{k,n}\left(\psi_{k}\right)+\mathrm{pen}_{k,n}\right\} ,
\]
for some appropriately chosen penalties $\mathrm{pen}_{k,n}$. Here,
we may recognise $\min_{\psi_{k}\in\mathbb{S}_{k}}\ell_{k,n}\left(\psi_{k}\right)$
as the negative maximum likelihood for the mixture model of order
$k$.

We can easily identify this approach with some traditional method
such as the Akaike information criterion (AIC; \citealp{akaike1974new})
and the Bayesian information criterion (BIC; \citealp{schwarz1978estimating})
by taking $\mathrm{pen}_{k,n}^{\mathrm{AIC}}=\mathrm{dim}\left(\mathbb{S}_{k}\right)n^{-1}$
and $\mathrm{pen}_{k,n}^{\mathrm{BIC}}=\mathrm{dim}\left(\mathbb{S}_{k}\right)\left(2n\right)^{-1}\log n$,
respectively, where $\mathrm{dim}\left(\mathbb{S}\right)$ is the
vector space dimension of $\mathbb{S}$. Let $\mathrm{Ln}\left(n\right)=\log\left(e\vee n\right)$
denote the logarithm function truncated to $1$, from below, where
$a\vee b=\max\left\{ a,b\right\} $, and denote its $\nu$-fold composition
by $\mathrm{Ln}^{\circ\nu}\left(n\right)$ (e.g., $\mathrm{Ln}^{\circ3}\left(n\right)=\mathrm{Ln}\circ\mathrm{Ln}\circ\mathrm{Ln}\left(n\right)$).
In this work, we introduce the $\nu$-BIC and $\epsilon$-BIC families
of penalties of the forms $\mathrm{pen}_{k,n}^{\nu}=\alpha\left(k\right)n^{-1}\mathrm{Ln}^{\circ\nu}\left(n\right)\log n$
and $\mathrm{pen}_{k,n}^{\epsilon}=\alpha\left(k\right)n^{-1}\left(\log n\right)^{1+\epsilon}$,
for any choice of $\epsilon>0$ and $\nu\in\mathbb{N}$, and for $\alpha:\mathbb{N}\to\mathbb{R}_{>0}$
strictly increasing, where the choice of $\alpha\left(k\right)=\mathrm{dim}\left(\mathbb{S}_{k}\right)/2=\left(m+1\right)k/2$
can be directly compared to the BIC. We shall demonstrate that this
modest modification to the BIC is consistent under weaker assumptions
than those required for the BIC, as we shall elaborate upon in the
sequel. We also provide a complementary misspecification result: when
$f_{0}$ need not belong to any ${\cal M}_{k}^{\phi}$, any IC satisfying
a mild vanishing-penalty condition will eventually select an order
whose best-fitting mixture model is Kullback--Leibler optimal among
the candidate orders.

For regular models, order selection consistency of the BIC was proved
in the works of \citet{nishii1988maximum}, \citet{vuong1989likelihood},
and \citet{sin1996information}. Unfortunately, the mixture model
setting is not regular and thus alternative proofs of consistency
are required. A prime and often cited example of such a result is
that of \citet{keribin2000consistent} who establish the consistency
of the BIC for mixture models under strong assumptions. Namely, among
other things, the proof requires that the component density have up
to five derivatives and that all of these higher derivatives have
finite third moments. A remarkable result of \citet{gassiat2012consistent}
(see also \citealp[Sec. 4.3]{gassiat2018universal}) shows that the
BIC is order consistent for choosing between models when $\bar{k}=\infty$,
albeit at the price of much stronger assumptions. In particular, even
when restricted to the class of location mixtures, the proof requires
that $\phi$ has up to three derivatives and that all derivatives
up to second order have finite third moments, while third order derivatives
have second moments.

Beyond the BIC, \citet{drton2017bayesian} proposed the so-called
singular BIC (sBIC) as an approach that operationalises the expansion
$n\ell_{k,n}\left(\psi_{k}\right)+\lambda_{k}\left(f_{0}\right)\log n+\left\{ m_{k}\left(f_{0}\right)-1\right\} \log\log n+O_{\mathrm{P}}\left(1\right)$
of the so-called Bayes free energy of \citet{watanabe2013widely},
related to the widely-applicable BIC (WBIC), where the coefficients
$\left(\lambda_{k}\left(f_{0}\right),m_{k}\left(f_{0}\right)\right)_{k\in\left[\bar{k}\right]}$
are dependent on the true and unknown model $f_{0}$. As noted in
\citet[Sec. 16.3]{chen2023statistical}, the specification of the
sBIC arises as a solution to a nonlinear system of stochastic equations
making it challenging to faithfully implement in practice. In \citet{baudry2015estimation},
the log-likelihood is replaced by the so-called complete data log-likelihood
and a consistency result is obtained with the usual BIC-form penalty.
However, the proof makes the requirement that the limiting objective
functionals of all order $k\in\left[\bar{k}\right]$ have positive
definite Hessian, which is not possible in the context of mixture
models, as typically if $k_{0}<k$, then the minimisers of any risk
over ${\cal M}_{k}^{\phi}$ will have connected components. In the
context of test-based order selection, this phenomenon is discussed
in \citet[Ch. 9]{chen2023statistical}.

Sacrificing the requirement that $\mathrm{pen}_{n,k}=\tilde{O}\left(n^{-1}\right)$
(where we use the soft-Oh notation $\tilde{O}$ to denote the Landau
order up to logarithmic factors; cf. \citealp[Sec. 3-5]{Cormen:2002aa}),
\citet{nguyen2024panic} propose to take penalties such that $\lim_{n\to\infty}\sqrt{n}\mathrm{pen}_{n,k}=\infty$,
e.g. $\mathrm{pen}_{n,k}=\tilde{O}\left(n^{-1/2}\right)$, within
their PanIC framework for model selection in general risk minimisation
problems (see also \citealp{westerhout2024asymptotic}). Using these
larger penalties, it is possible to prove model selection consistency
under very mild conditions, with data potentially dependent, and where
the objective functions (including but not necessarily log-likelihoods)
may be non-differentiable or even discontinuous and unmeasurable.
Although these results are very general, it is clear that for sufficiently
well-behaved models, $\tilde{O}\left(n^{-1/2}\right)$ penalties are
too large and as a consequence, the corresponding IC $\hat{k}_{n}$
may be inefficient and thus larger than necessary values of $n$ are
required for consistency to manifest, in practice.

In this work, we will prove that taking $\mathrm{pen}_{n,k}^{\nu}$
or $\mathrm{pen}_{n,k}^{\epsilon}$ sufficiently increases the penalisation
as to allow for dramatic relaxation of the conditions required for
order consistency. In particular, we will show that consistency can
be proved with only the requirement that the parameter space $\mathbb{S}_{k}$
is compact, the log-density is a Glivenko--Cantelli (GC) class, and
the component density functions are Lipschitz over $\mathbb{T}$ with
an envelope having finite second moment. These conditions are far
less onerous than those of \citet{keribin2000consistent} and allow
for application in cases where the component densities $\phi$ need
not even be differentiable. In fact, penalties $\mathrm{pen}_{n,k}^{\epsilon}$
were considered within the framework of \citet{keribin2000consistent}
but the larger penalisation was not exploited to weaken the regularity
conditions.

In fact, for $\nu$ chosen sufficiently large or $\epsilon$ chosen
sufficiently small, the decision made by using the $\nu$-BIC or $\epsilon$-BIC
versus that by the BIC will be the same, in practical settings. For
instance, when $\nu=3$, $\mathrm{Ln}^{\circ\nu}\left(n\right)\le1$
for all $n\le\exp^{\circ3}\left(1\right)\approx3.8\times10^{6}$ and
$\mathrm{Ln}^{\circ\nu}\left(n\right)\le1.1$ for all $n\le\exp^{\circ3}\left(1.1\right)\approx5.7\times10^{8}$.
Similarly, taking $\epsilon=0.02$, $\log\left(n\right)^{0.02}\le1.1$
for all $n\le\exp\left(117.3909\right)\approx9.6\times10^{50}$. Thus,
one key interpretation of our results is that they provide assurance
for the correctness of the usage of the BIC outside the regularity
framework of \citet{keribin2000consistent}, since we can interpret
usage of the BIC in finite sample studies as applications of the $\nu$-BIC
or $\epsilon$-BIC penalties up to a negligible numerical factor.
We also use our framework to clarify two limitations of consistent
IC-based order selection in mixtures. First, within the general sufficient
conditions used to establish consistency, there is no universally
minimal BIC-scale penalty, so consistency alone does not single out
an optimal IC. Second, we establish a tension between order consistency
and minimax optimality in mixtures in the style of \citet{Yang2005aicbic}.
In particular, we show that in a simple Gaussian mixture family, an
AIC-like criterion achieves the parametric minimax Hellinger risk,
whereas any order-consistent criterion (including the BIC and our
proposals) must be minimax-inefficient by an unbounded factor. To
the best of our knowledge, this mixture model AIC/BIC tension in mixture
models has not previously been stated explicitly.

Indeed, the method of IC is not the only approach for order selection
and estimation under model uncertainty in mixture models. For example,
model selection using hypothesis testing has been considered in \citet{mclachlan1987bootstrapping},
\citet{polymenis1998determination}, \citet{wasserman2020universal},
\citet{nguyen2023order}, and the many approaches covered in \citet{chen2023statistical}.
Shrinkage estimator approaches have been proposed in \citet{chen2009order},
\citet{huang2017model}, \citet{huang2022statistical}, and \citet{budanova2025penalized}.
And bounds of risk functions for estimation under model uncertainty
have been considered by \citet{li1999mixture}, \citet{rakhlin2005risk},
\citet{klemela2007density}, \citet{maugis2011non}, \citet{manole2022refined},
and \citet{Chong:2024aa}, among many others. We do not comment or
compare these disparate approaches to model selection, and direct
the reader to dedicated texts, such as \citet[Ch. 6]{peel2000finite},
\citet{mclachlan2014number}, \citet[Ch. 7]{fruhwirth2019handbook},
and \citet[Ch. 16]{chen2023statistical}.

To conclude, we note that beyond mixtures of density functions, one
can consider IC for conditional density models, as well. Results in
this setting include the works of \citet{naik2007extending}, \citet{hafidi2010kullback},
\citet{depraetere2014order}, and \citet{hui2015order}. It is possible
to adapt our approach to provide inference in this setting, and we
will consider such situations as an example application of our theory.

The remainder of the manuscript is organised as follows. Section \ref{sec:Main_Results}
introduces the notational conventions used throughout the paper, presents
the technical preliminaries, and states our main consistency results,
including a misspecification result in which order selection targets
Kullback--Leibler optimal approximating models. Section \ref{sec:Example-applications}
provides illustrative examples that demonstrate applications of our
theory. Section \ref{sec:Limitations} then discusses limitations
of consistent model selection in mixture models, including penalty
flexibility and the aforementioned AIC/BIC tension. Concluding remarks
and discussion are then made in Section \ref{sec:Discussion}. Complete
proofs of our technical results and primary theorems are provided
in the Appendix A. Appendix B then contains auxiliary details. 

\section{Main results}

\label{sec:Main_Results}

\subsection{Technical preliminaries}

We retain the setup from Section \ref{sec:Introduction} and follow
the usual notational conventions for empirical processes (see, e.g.,
\citealp{van-de-Geer:2000aa}). For a functional class ${\cal F}$,
we denote by $\ell^{\infty}\left({\cal F}\right)$ the space of bounded
functionals supported on ${\cal F}$ equipped with the uniform norm
$\left\Vert h\right\Vert _{{\cal F}}=\sup_{f\in{\cal F}}\left|h\left(f\right)\right|$,
for $h\in\ell^{\infty}\left({\cal F}\right)$. For a particular ${\cal F}\ni f:\mathbb{X}\to\mathbb{R}$,
we shall write $Pf=\mathrm{E}f\left(X\right)$ and $P_{n}f=n^{-1}\sum_{i=1}^{n}f\left(X_{i}\right)$,
where we can identify $P_{n}-P$ as an element of $\ell^{\infty}\left({\cal F}\right)$.
We shall say that ${\cal F}$ is a GC class if it satisfies a uniform
strong law of large numbers in the sense that $\left\Vert P_{n}-P\right\Vert _{{\cal F}}\stackrel[n\to\infty]{\mathrm{a.s.}}{\longrightarrow}0$.
We write $a_{n}\asymp b_{n}$ if there exist constants $0<c\le C<\infty$
and $n_{0}\in\mathbb{N}$ such that $cb_{n}\le a_{n}\le Cb_{n}$,
for all $n\ge n_{0}$.

For a norm $\left\Vert \cdot\right\Vert $, we say that $\left(\left[f_{j}^{\mathrm{L}},f_{j}^{\mathrm{U}}\right]\right)_{j\in\left[N\right]}$
is a $\delta$-bracketing of ${\cal F}$, if for each $f\in{\cal F}$,
there is a $j$ such that $f_{j}^{\mathrm{L}}\le f\le f_{j}^{\mathrm{U}}$,
where $\left\Vert f_{j}^{\mathrm{U}}-f_{j}^{\mathrm{L}}\right\Vert \le\delta$,
for every $j\in\left[N\right]$. The smallest number $N$ required
for ${\cal F}$ to have a $\delta$-bracketing is called the bracketing
number and denoted by $N_{\left[\right]}\left(\delta,{\cal F},\left\Vert \cdot\right\Vert \right)$,
where $N_{\left[\right]}\left(\delta,{\cal F},\left\Vert \cdot\right\Vert \right)=\infty$
if no such bracketing exists. Further, we call $H_{\left[\right]}\left(\delta,{\cal F},\left\Vert \cdot\right\Vert \right)=\log N_{\left[\right]}\left(\delta,{\cal F},\left\Vert \cdot\right\Vert \right)$
the $\delta$-bracketing entropy of ${\cal F}$. When ${\cal F}$
is a subset of the Lebesgue space ${\cal L}_{p}\left(P\right)={\cal L}_{p}\left(\mathbb{X},{\cal B}\left(\mathbb{X}\right),P\right)$
and the norm $\left\Vert \cdot\right\Vert $ is taken to be the corresponding
norm, for some $p\ge1$, we will use the shorthand $N_{\left[\right]}\left(\delta,{\cal F},{\cal L}_{p}\left(P\right)\right)$.
For $\psi\in\mathbb{R}^{q}$, we will use $\left\Vert \psi\right\Vert $
and $\left\Vert \psi\right\Vert _{p}$ to denote the Euclidean ($\ell_{2}$)
norm and $\ell_{p}$ norm of $\psi$, respectively.

Relatedly, for a metric space $\left({\cal G},d\right)$, we say that
a collection $\left(g_{j}\right)_{j\in\left[N\right]}$ is a $\delta$-covering
of ${\cal G}$ if for each $g\in{\cal G}$, there exists a $j\in\left[N\right]$
such that $d\left(g,g_{j}\right)\le\delta$. The smallest $N$ required
for ${\cal G}$ to have a $\delta$-covering is called the covering
number and denoted $N\left(\delta,{\cal G},d\right)$, where $N\left(\delta,{\cal G},d\right)=\infty$
if no such covering exists. With these definitions, the following
results are useful for proving that ${\cal F}$ is a GC class. 
\begin{lem}
\label{lem:GCclass}If $\mathbf{X}_{n}$ is IID and ${\cal F}\subset{\cal L}_{1}\left(P\right)$
is a class of measurable functions admitting an envelope $F\in{\cal L}_{1}\left(P\right)$
and $N_{\left[\right]}\left(\delta,{\cal F},{\cal L}_{1}\left(P\right)\right)<\infty$,
for every $\delta>0$, then ${\cal F}$ is a GC class. In particular,
if 
\begin{equation}
{\cal F}=\left\{ f_{g}:\mathbb{X}\ni x\mapsto f_{g}\left(x\right),g\in{\cal G}\right\} ,\label{eq:parameterization}
\end{equation}
$\left({\cal G},d\right)$ is compact and $g\mapsto f_{g}\left(x\right)$
is continuous for $P$-almost all $x\in\mathbb{X}$, then $N_{\left[\right]}\left(\delta,{\cal F},{\cal L}_{1}\left(P\right)\right)<\infty$,
for every $\delta>0$, if there exists an envelope function $F\in{\cal L}_{1}\left(P\right)$,
such that $\sup_{g\in{\cal G}}\left|f_{g}\right|\le F$. 
\end{lem}

\begin{proof}
See \citet[Thm. 2.4.1 and Lem. 3.10]{van-de-Geer:2000aa}. 
\end{proof}
\begin{rem}
As an alternative to the bracketing entropy condition, above, one
can also verify that ${\cal F}$ is a GC class via random covering
entropy methods (see, e.g., \citealp[Sec. 3.6]{van-de-Geer:2000aa}).
We do not use such conditions in this text since bracketing entropy
bounds pair better with the rest of our theoretical approach as shall
be seen, in the sequel. 
\end{rem}

\begin{rem}
Throughout this text, we shall assume that all random quantities of
interest are measurable. Indeed, all of the classes ${\cal F}$ considered
in this work can be understood as classes of Carathéodory integrands
that are indexed by Euclidean spaces and thus are measurable, with
measurable maxima/minima, sums, compositions, and other required manipulations
(cf. \citealp[Ch. 14]{rockafellar2009variational}). 
\end{rem}

\begin{lem}
\label{lem:ParametricBracket}If ${\cal F}$ is parameterised as per
\ref{eq:parameterization}, such that there exists an envelope function
$F:\mathbb{X}\to\mathbb{R}_{>0}$ satisfying 
\[
\left|f_{g}\left(x\right)-f_{h}\left(x\right)\right|\le d\left(g,h\right)F\left(x\right),
\]
then $N_{\left[\right]}\left(2\delta\left\Vert F\right\Vert ,{\cal F},\left\Vert \cdot\right\Vert \right)\le N\left(\delta,{\cal G},d\right)$.
Furthermore, if ${\cal G}=\mathbb{S}\subset\mathbb{R}^{q}$ is compact
and $d$ is the Euclidean distance, then there exists a constant $K>0$
such that 
\[
N\left(\delta,\mathbb{S},d\right)\le K\left(\frac{\mathrm{diam}\left(\mathbb{S}\right)}{\delta}\right)^{q},
\]
for every $0<\delta<\mathrm{diam}\left(\mathbb{S}\right)$. 
\end{lem}

\begin{proof}
See \citet[Thm. 2.7.17]{Vaart:2023aa} and \citet[Example 19.7]{van-der-Vaart:2000aa}. 
\end{proof}
Next, suppose that ${\cal F}$ is a subset of probability density
functions on $\mathbb{X}$, in the sense that for each $f\in{\cal F}$,
$f\left(x\right)>0$ for each $x\in\mathbb{X}$ and $\int_{\mathbb{X}}f\mathrm{d}\mathfrak{m}=1$.
For each $n\in\mathbb{N}$, we let $\hat{f}_{n}:\Omega\to{\cal F}$
be a maximum likelihood estimator (MLE), defined by 
\begin{equation}
\hat{f}_{n}\in\underset{f\in{\cal F}}{\arg\max}P_{n}\log f,\label{eq:MLE}
\end{equation}
where the existence of MLEs is guaranteed in our application. For
each $f,g\in{\cal F}$, denote the Hellinger divergence by 
\[
\mathfrak{h}\left(f,g\right)=\left\{ \frac{1}{2}\int_{\mathbb{X}}\left(f^{1/2}-g^{1/2}\right)^{2}\mathrm{d}\mathfrak{m}\right\} ^{1/2},
\]
and write the subset of ${\cal F}$ that are within the $\delta$-Hellinger
ball of $f_{0}$, for $0<\delta\le1$, as ${\cal F}\left(\delta\right)=\left\{ f\in{\cal F}:\mathfrak{h}\left(f,f_{0}\right)\le\delta\right\} $.
Further define the following entropy integral: 
\begin{equation}
J\left(\delta\right)=\delta\vee\int_{0}^{\delta}\sqrt{H_{\left[\right]}\left(u^{2},{\cal F}\left(\delta\right),{\cal L}_{1}\left(\mathfrak{m}\right)\right)}\mathrm{d}u.\label{eq:entropyintegral}
\end{equation}
The next result, our main technical tool, is adapted from \citet[Cor. 7.5]{van-de-Geer:2000aa}
(see the Appendix). 
\begin{prop}
\label{prop:MainTechProp}If $f_{0}\in{\cal F}$, $\mathbf{X}_{n}$
is IID and there is some $\Psi:\mathbb{R}_{>0}\to\mathbb{R}_{>0}$
such that $\Psi\left(\delta\right)\ge J\left(\delta\right)$, where
$\Psi\left(\delta\right)/\delta^{2}$ is a non-increasing function
of $\delta$, then, for some constant $c_{1}>0$, 
\[
\mathrm{P}\left(\int_{\mathbb{X}}\log\frac{\hat{f}_{n}}{f_{0}}\mathrm{d}P_{n}\ge\delta^{2}\right)\le c_{1}\exp\left\{ -\frac{n\delta^{2}}{c_{1}^{2}}\right\} ,
\]
for every $\delta\ge\delta_{n}$, such that $\delta_{n}$ satisfies
$\sqrt{n}\delta_{n}^{2}\ge c\Psi\left(\delta_{n}\right)$, for a universal
constant $c>0$. 
\end{prop}

For a pair of probability measures $P$ and $Q$ on $\left(\mathbb{X},{\cal B}\left(\mathbb{X}\right)\right)$,
we define the Kullback--Leibler divergence as 
\[
\mathfrak{K}\left(P,Q\right)=\begin{cases}
\int_{\mathbb{X}}\log\left(\frac{\mathrm{d}P}{\mathrm{d}Q}\right)\mathrm{d}P & \text{if }P\ll Q,\\
\infty & \text{otherwise.}
\end{cases}
\]
If both $P$ and $Q$ have densities $f$ and $g$ (with respect to
$\mathfrak{m}$), then we can also write 
\[
\mathfrak{K}\left(P,Q\right)=\int_{\mathbb{X}}f\left(x\right)\log\frac{f\left(x\right)}{g\left(x\right)}\mathfrak{m}\left(\mathrm{d}x\right).
\]

\subsection{Assumptions}

\label{sec:MainResults}

We state the following assumptions, repeating some previously stated
items, for convenience. 
\begin{description}
\item [{A1}] $X\in\mathbb{X}$ has PDF $f_{0}\in{\cal M}_{k_{0}}^{\phi}\setminus{\cal M}_{k_{0}-1}^{\phi}$,
for some $k_{0}\in\left[\bar{k}\right]$ (with the convention ${\cal M}_{0}^{\phi}=\varnothing$),
and $\mathbf{X}_{n}=\left(X_{1},\dots,X_{n}\right)$ is a sample of
IID replicates of $X$. 
\item [{A2}] $\phi$ is strictly positive; i.e., $\phi\left(x;\theta\right)>0$,
for each $x\in\mathbb{X}$ and $\theta\in\mathbb{T}\subset\mathbb{R}^{m}$,
where $\mathbb{T}$ is compact. 
\item [{A3}] $\phi$ is a Carathéodory integrand in the sense that $\phi\left(x;\cdot\right):\mathbb{T}\to\mathbb{R}$
is continuous for each fixed $x\in\mathbb{X}$, and $\phi\left(\cdot;\theta\right):\mathbb{X}\to\mathbb{R}$
is measurable, for each fixed $\theta\in\mathbb{T}$. 
\item [{A4}] There exists a $G_{1}\in{\cal L}_{1}\left(P\right)$ such
that, for every $x\in\mathbb{X}$, 
\[
\max_{\theta\in\mathbb{T}}\left|\log\phi\left(x;\theta\right)\right|<G_{1}\left(x\right).
\]
\item [{A5}] There exist functions $L_{\phi}:\mathbb{X}\to\mathbb{R}_{\ge0}$
and $G_{2}\in{\cal L}_{1}\left(\mathfrak{m}\right)$ such that, for
every $\theta_{1},\theta_{2}\in\mathbb{T}$ and $x\in\mathbb{X}$,
\begin{equation}
\left|\phi\left(x;\theta_{1}\right)-\phi\left(x;\theta_{2}\right)\right|\le L_{\phi}\left(x\right)\left\Vert \theta_{1}-\theta_{2}\right\Vert _{1},\label{eq:LipschitzCondition}
\end{equation}
where $\max_{\theta\in\mathbb{T}}\phi\left(x;\theta\right)+L_{\phi}\left(x\right)\le G_{2}\left(x\right)$. 
\end{description}
When $\phi\left(x,\cdot\right):\mathbb{T}\to\mathbb{R}$ is continuously
differentiable, for each $x\in\mathbb{X}$, we can replace A5 with
the following condition. 
\begin{description}
\item [{A5{*}}] There exists a function $G_{2}\in{\cal L}_{1}\left(\mathfrak{m}\right)$
such that, for each $x\in\mathbb{X}$, 
\[
\max_{\theta\in\mathbb{T}}\phi\left(x;\theta\right)+\max_{j\in\left[m\right]}\max_{\theta\in\mathbb{T}}\left|\frac{\partial\phi\left(x;\theta\right)}{\partial\vartheta_{j}}\right|\le G_{2}\left(x\right).
\]
\item [{B1}] For each $k\in\left[\bar{k}\right]$, $\lim_{n\to\infty}\mathrm{pen}_{k,n}=0$. 
\item [{B2}] For each $1\le k<l\le\bar{k}$, 
\[
\lim_{n\to\infty}\frac{n}{\log n}\left\{ \mathrm{pen}_{l,n}-\mathrm{pen}_{k,n}\right\} =\infty.
\]
\end{description}
Let us provide some reasoning for the assumptions. A1 identifies that
the true PDF $f_{0}$ is within the class of mixture models ${\cal M}_{k}^{\phi}$,
where $k\le\bar{k}$. We also have the availability of an IID sample
of $n$ replicates $\mathbf{X}_{n}$ from the true probability model,
required for application of Lemma \ref{lem:GCclass} and Proposition
\ref{prop:MainTechProp}. A2 provides the necessary compactness of
the parameter space for bounding the metric entropy using Lemma \ref{lem:ParametricBracket}
and ensures that the log-densities $\log\phi\left(x;\theta\right)\ne\infty$,
for each pair of $x$ and $\theta$. A3 implies that the density expressions,
log-density expressions, and their minima/maxima and averages remain
measurable. A4 provides envelope functions so that we can apply Lemma
\ref{lem:GCclass} to prove the convergence of the average log-likelihoods
to their limits, for each $k$. And A5 provides the required envelope
functions for bounding the bracketing entropy using Lemma \ref{lem:ParametricBracket}.
Note that when $\phi$ is continuously differentiable in $\theta$,
for each fixed $x$, A5{*} implies A5 by the mean value theorem. B1
guarantees that the penalty does not interfere with comparisons between
models ${\cal M}_{k}^{\phi}$ and ${\cal M}_{k_{0}}^{\phi}$, with
respect to their maximum average log-likelihoods, when $k<k_{0}$.
B2 then provides sufficient penalisation to distinguish between ${\cal M}_{k}^{\phi}$
and ${\cal M}_{k_{0}}^{\phi}$, with respect to their model complexities,
when $k>k_{0}$.

\subsection{Consistency results}

\label{subsec:Consistency-results}

The following pair of lemmas, proved in the Appendix, constitute the
main technical contribution of the text and will imply the main result. 
\begin{lem}
\label{lem:Convergenceofmin}Under A1--A4, for each $k\in\left[\bar{k}\right]$,
\[
\min_{\psi_{k}\in\mathbb{S}_{k}}\ell_{k,n}\left(\psi_{k}\right)\stackrel[n\to\infty]{\mathrm{a.s.}}{\longrightarrow}\min_{\psi_{k}\in\mathbb{S}_{k}}\ell_{k}\left(\psi_{k}\right).
\]
\end{lem}

\begin{lem}
\label{lem:Rateofmin}Under A1--A3 and A5, for each $k,l\ge k_{0}$,
\[
\frac{n}{\log n}\left\{ \min_{\psi_{k}\in\mathbb{S}_{k}}\ell_{k,n}\left(\psi_{k}\right)-\min_{\psi_{l}\in\mathbb{S}_{l}}\ell_{l,n}\left(\psi_{l}\right)\right\} =O_{\mathrm{P}}\left(1\right).
\]
\end{lem}

As a consequence, we have the following generic consistency result
and our main result as a corollary. 
\begin{thm}
\label{thm:GenericConsistency}Under A1--A5, B1, and B2, $\lim_{n\to\infty}\mathrm{P}\left(\hat{k}_{n}=k_{0}\right)=1$. 
\end{thm}

\begin{proof}
Given Lemmas \ref{lem:Convergenceofmin} and \ref{lem:Rateofmin},
the proof follows identical steps to that of \citet[Thm. 1]{nguyen2024panic}
and \citet[Prop 5.5]{westerhout2024asymptotic}. 
\end{proof}
\begin{cor}
\label{cor:Main=00003D00003D000020Result}Under A1--A5, for any increasing
$\alpha:\mathbb{N}\to\mathbb{R}_{>0}$, $\nu\in\mathbb{N}$ and $\epsilon>0$,
the $\nu$-BIC 
\[
\hat{k}_{n}^{\nu}=\min\arg\min_{k\in\left[\bar{k}\right]}\left\{ \min_{\psi_{k}\in\mathbb{S}_{k}}\ell_{k,n}\left(\psi_{k}\right)+\alpha\left(k\right)\frac{\mathrm{Ln}^{\circ\nu}\left(n\right)\log n}{n}\right\} 
\]
and $\epsilon$-BIC 
\[
\hat{k}_{n}^{\epsilon}=\min\arg\min_{k\in\left[\bar{k}\right]}\left\{ \min_{\psi_{k}\in\mathbb{S}_{k}}\ell_{k,n}\left(\psi_{k}\right)+\alpha\left(k\right)\frac{\log^{1+\epsilon}n}{n}\right\} 
\]
are consistent estimators of $k_{0}$. 
\end{cor}

\begin{proof}
We have $\mathrm{Ln}^{\circ\nu}\left(n\right)\log n\ll\log^{1+\epsilon}n\ll n$,
where $g\left(n\right)\ll h\left(n\right)$ if and only if $\lim_{n\to\infty}g\left(n\right)/h\left(n\right)=0$,
thus $\mathrm{pen}_{k,n}^{\nu}$ and $\mathrm{pen}_{k,n}^{\epsilon}$
satisfy B1. For $k<l$, observe that 
\[
\frac{n}{\log n}\left\{ \mathrm{pen}_{l,n}^{\nu}-\mathrm{pen}_{k,n}^{\nu}\right\} =\left\{ \alpha\left(l\right)-\alpha\left(k\right)\right\} \mathrm{Ln}^{\circ\nu}\left(n\right)\underset{n\to\infty}{\longrightarrow}\infty,
\]
and 
\[
\frac{n}{\log n}\left\{ \mathrm{pen}_{l,n}^{\epsilon}-\mathrm{pen}_{k,n}^{\epsilon}\right\} =\left\{ \alpha\left(l\right)-\alpha\left(k\right)\right\} \log^{\epsilon}n\underset{n\to\infty}{\longrightarrow}\infty,
\]
as required, since $\log$ and positive powers are increasing, and
since $\alpha$ is increasing. 
\end{proof}
To illustrate the advantage of Theorem \ref{thm:GenericConsistency}
over the main result of \citet{keribin2000consistent} (i.e., Theorem
2.1), it is helpful to compare the assumptions, which we supply for
convenience in Appendix B. Firstly, \citet{keribin2000consistent}
shares A1--A3, and makes the additional condition (Id) requiring
identifiability regarding mappings between the functional spaces ${\cal M}_{k}^{\phi}$
and the parameter spaces $\mathbb{S}_{k}$, which is unnecessary in
our approach. Condition (P1-a) then requires that there exists a $G\in{\cal L}_{1}\left(P\right)$
such that $\left|\log f\right|\le G$, for every $f\in{\cal M}_{\bar{k}}^{\phi}$,
which is implied by A4. Condition (P1-b) then requires that the component
density $\phi$ possesses partial derivatives up to order $5$, where
all the partial derivatives of orders $2$, $3$, and $5$, are each
enveloped over $\mathbb{T}$ by integrable functions in ${\cal L}_{3}\left(\mathfrak{m}\right)$.
This compares directly with A5, where we assume that $\phi$ need
not even have one derivative, only requiring that it is Lipschitz,
with Lipschitz constant in ${\cal L}_{1}\left(\mathfrak{m}\right)$
and that $\phi$ itself is enveloped by an ${\cal L}_{1}\left(\mathfrak{m}\right)$
function. This is a dramatic simplification of assumptions and makes
our approach useful in many situations where the method of \citet{keribin2000consistent}
does not apply.

Conditions (P2) and (P3), which our theory does not require, then
make linear independence assumptions between the first and second
order partial derivatives of $\phi$, and requires that a class of
linear combinations of partial first and zeroth order partial derivatives
of $\phi$ be a Donsker class of functionals with continuous sample
paths, respectively. Finally, Condition (C1) corresponds to B1 but
where B2 is replaced by the assumption that for each $1\le k_{0}<k\le\bar{k}$,
\[
\lim_{n\to\infty}n\left\{ \mathrm{pen}_{k,n}-\mathrm{pen}_{k_{0},n}\right\} =\infty.
\]
Thus B1 and B2 are stronger than Condition (C1), although only up
to a logarithmic factor, which we believe is a small price to pay
for the dramatic reduction or elimination of the other conditions,
since the effect is that we only require an infinitesimally larger
penalty than that of the BIC and thus, as argued, is negligible in
practice.

A complementary comparison can be made to the PanIC framework of \citet{nguyen2024panic},
which treats model selection problems in a general empirical-risk
minimisation form. For convenience, we reproduce in Appendix~B the
assumptions (N-A1)--(N-A2) and the penalty conditions (N-B1)--(N-B2)
of \citet{nguyen2024panic}. Specialising their setting to likelihood-based
order selection by identifying $\mathbb{T}_{k}$ with $\mathbb{S}_{k}$
and taking the loss to be $\ell_{k}(x;\psi_{k})=-\log f_{k}(x;\psi_{k})$,
the hypotheses (N-A1)--(N-A2) amount to compactness of $\mathbb{S}_{k}$
together with a global Lipschitz condition on the negative log-likelihood
with an ${\cal L}_{2}(P)$ envelope, which are slightly simpler to
verify than A2--A5.

Under (N-A1)--(N-A2), \citet[Thm.~1]{nguyen2024panic} establishes
consistency for any penalised estimator whose penalties satisfy (N-B1)--(N-B2),
where (N-B2) requires penalty size condition $\sqrt{n}\{\mathrm{pen}_{l,n}-\mathrm{pen}_{k,n}\}\to\infty$
for $k<l$, implying that $\mathrm{pen}_{l,n}-\mathrm{pen}_{k,n}$
must dominate $O\left(n^{-1/2}\right)$. By contrast, in the finite
mixture likelihood setting, Lemma~\ref{lem:Rateofmin} shows that
the relevant overfitting fluctuations are of order $O\left(\log n/n\right)$.
Consequently, the separation condition B2 in Theorem~\ref{thm:GenericConsistency}
only requires $\mathrm{pen}_{l,n}-\mathrm{pen}_{k,n}$ to dominate
$O\left(\log n/n\right)$, permitting penalties that are asymptotically
much smaller than those required by (N-B2) while delivering the same
consistency guarantee that $\mathrm{P}(\hat{k}_{n}=k_{0})\to1$, under
the same assumptions on the model class. 
\begin{rem}
\label{rem:Choice-epsilon-and-nu}As discussed in Section \ref{sec:Introduction},
the penalty induced by the $\epsilon$-BIC and $\nu$-BIC can be made
indistinguishable from that of the BIC when $\epsilon$ is small or
when $\nu$ is large. In Table \ref{tab:Schedule-of-penalty}, we
provide a schedule for the ratio of the rates of the $\epsilon/\nu$-BIC
to that of the BIC, computed as $\mathrm{pen}_{k,n}^{\epsilon}/\mathrm{pen}_{k,n}^{\mathrm{BIC}}$
or $\mathrm{pen}_{k,n}^{\nu}/\mathrm{pen}_{k,n}^{\mathrm{BIC}}$,
for the same choice of $\alpha$, where $\mathrm{pen}_{k,n}^{\mathrm{BIC}}=\alpha\left(k\right)n^{-1}\log n$.
We observe that for sufficiently small choices of $\epsilon$ or large
choices of $\nu$, the penalty ratios can be made to stay close to
one for a larger range of sample sizes $n$. 
\end{rem}

\begin{table}
\caption{Schedule of penalty ratios $\mathrm{pen}_{k,n}^{\epsilon}/\mathrm{pen}_{k,n}^{\mathrm{BIC}}$
or $\mathrm{pen}_{k,n}^{\nu}/\mathrm{pen}_{k,n}^{\mathrm{BIC}}$ for
the same choice of $\alpha$.}
\label{tab:Schedule-of-penalty}
\centering{}%
\begin{tabular}{|c|cccccccccc|}
\hline 
\multicolumn{1}{|c|}{Penalty} & $n=10$  & $10^{2}$  & $10^{3}$  & $10^{4}$  & $10^{5}$  & $10^{6}$  & $10^{7}$  & $10^{8}$  & $10^{9}$  & $10^{10}$\tabularnewline
\hline 
\hline 
$\epsilon=0.1$  & 1.09  & 1.16  & 1.21  & 1.25  & 1.28  & 1.30  & 1.32  & 1.34  & 1.35  & 1.37\tabularnewline
$\epsilon=0.05$  & 1.04  & 1.08  & 1.10  & 1.12  & 1.13  & 1.14  & 1.15  & 1.16  & 1.16  & 1.17\tabularnewline
$\epsilon=0.02$  & 1.02  & 1.03  & 1.04  & 1.05  & 1.05  & 1.05  & 1.06  & 1.06  & 1.06  & 1.06\tabularnewline
$\epsilon=0.01$  & 1.01  & 1.02  & 1.02  & 1.02  & 1.02  & 1.03  & 1.03  & 1.03  & 1.03  & 1.03\tabularnewline
\hline 
$\nu=1$ ($\epsilon=1$)  & 2.30  & 4.61  & 6.91  & 9.21  & 11.51  & 13.82  & 16.12  & 18.42  & 20.72  & 23.03\tabularnewline
$\nu=2$  & 1.00  & 1.53  & 1.93  & 2.22  & 2.44  & 2.63  & 2.78  & 2.91  & 3.03  & 3.14\tabularnewline
$\nu=3$  & 1.00  & 1.00  & 1.00  & 1.00  & 1.00  & 1.00  & 1.02  & 1.07  & 1.11  & 1.14\tabularnewline
$\nu=4$  & 1.00  & 1.00  & 1.00  & 1.00  & 1.00  & 1.00  & 1.00  & 1.00  & 1.00  & 1.00\tabularnewline
\hline 
\end{tabular}
\end{table}

\begin{rem}
\label{rem:Choice_of_rate}It would appear that to improve upon the
rate of $O\left(n/\log n\right)$ in Lemma \ref{lem:Rateofmin} (and
subsequently B2), sharper bounds on the entropy $H_{\left[\right]}\left(u,\bar{{\cal F}}^{1/2}\left(\delta\right),{\cal L}_{2}\left(\mathfrak{m}\right)\right)$
are required that make use of the localising condition $\mathfrak{h}\left(f,f_{0}\right)\le\delta$
of $\bar{{\cal F}}^{1/2}\left(\delta\right)$. It is in the effort
of obtaining such sharp bounds that necessitates the additional assumptions
regarding the partial derivatives of $\phi$ employed by \citet{keribin2000consistent}
and \citet{gassiat2012consistent}. Indeed, \citet[Sec. 7.5]{van-de-Geer:2000aa}
notes that in the finite dimensional parametric setting, one cannot
obtain $O\left(n\right)$ rates without exploiting the localising
condition. On the other hand, in Theorem \ref{thm:GenericConsistency},
the penalty need only satisfy B1 and B2, and thus need only go to
zero and scale at a rate larger than $O\left(n^{-1}\log n\right)$.
In fact the two choices that define the $\nu$-BIC and $\epsilon$-BIC
are arbitrary and were chosen because they could be made arbitrarily
small, to support our defense of use of the BIC in practice. However,
one may choose any sufficiently large penalty. For example, in \citet{nguyen2024panic}
the choice of $\mathrm{pen}_{k,n}=O\left(n^{-1/2}\sqrt{\mathrm{Ln}^{\circ\nu}\left(n\right)}\right)$
is proposed, for $\nu\ge1$, in conforming with the suggestion in
the treatment of \citet{sin1996information}, and applied successfully
to mixture model settings in \citet{nguyen2024panic} and \citet{westerhout2024asymptotic},
within the PanIC framework. 
\end{rem}

\begin{rem}
\label{rem:Choice_of_alpha}In Corollary \ref{cor:Main=00003D00003D000020Result},
the choice of $\alpha$ need only be an increasing function in $k$
and as such any such choice is valid for the conclusion to hold. For
example, one may ignore the multiplier due to $m$ and simply take
$\alpha\left(k\right)=k$, or one may take $\alpha\left(k\right)$
to be proportional to the number of free parameters in ${\cal M}_{k}^{\phi}$,
as in the usual BIC. Concretely, since $\psi_{k}$ contains $(k-1)$
free mixing weights and $km$ component parameters, a natural choice
is 
\[
\alpha\left(k\right)=\frac{(m+1)k-1}{2},
\]
or the simplified version $\alpha\left(k\right)=(m+1)k/2$ which we
propose throughout, which ignores the simplex constraint. This indifference
to the count of the number of so-called free parameters shows that
one need not be so judicious regarding what defines a parameter, as
in \citet{fraley2002} and \citet{baudry2010}. Other admissible choices
include $\alpha\left(k\right)=k$ (counting mixture weights only;
as considered in \citealp{keribin2000consistent}), $\alpha\left(k\right)=km/2$
(counting component parameters only), or more aggressively $\alpha\left(k\right)=k\log k$
or $\alpha\left(k\right)=k^{\gamma}$ for $\gamma>1$. Such freedom
mirrors that in other consistent IC frameworks; for instance, the
PanIC penalties of \citet[Rem.~3]{nguyen2024panic} permit any strictly
increasing sequence $\left(\alpha\left(k\right)\right)_{k}$ to weight
the penalty. Indeed it is also possible to consider random $\alpha$. 
\end{rem}

\subsection{Misspecification}

The results of Section \ref{subsec:Consistency-results} depend crucially
on Assumption~A1, which requires that $f_{0}\in{\cal M}_{k_{0}}^{\phi}$
for some $k_{0}\in\left[\bar{k}\right]$. In practice, however, it
is not always known that $f_{0}$ belongs to the family of mixtures
under study (and in particular, a ``true order'' $k_{0}$ need not
exist). In such a case, the conclusions of Theorem~\ref{thm:GenericConsistency}
and Corollary~\ref{cor:Main=00003D00003D000020Result} are no longer
meaningful as statements about recovering $k_{0}$, but the information
criteria considered above still admit a sensible population interpretation.
We therefore replace A1 by the following assumption. 
\begin{description}
\item [{A1{*}}] $X\in\mathbb{X}$ has PDF $f_{0}$ with respect to $\mathfrak{m}$,
and $\mathbf{X}_{n}=\left(X_{1},\dots,X_{n}\right)$ is a sample of
IID replicates of $X$. 
\end{description}
For each $k\in\left[\bar{k}\right]$ and $n\in\mathbb{N}$, write
\[
L_{k}=\min_{\psi_{k}\in\mathbb{S}_{k}}\ell_{k}\left(\psi_{k}\right),\qquad\hat{L}_{k,n}=\min_{\psi_{k}\in\mathbb{S}_{k}}\ell_{k,n}\left(\psi_{k}\right).
\]
Define the set of population risk minimising orders by 
\[
{\cal K}^{*}=\arg\min_{k\in\left[\bar{k}\right]}L_{k}.
\]
Moreover, for any $\psi_{k}\in\mathbb{S}_{k}$, whenever $\mathrm{E}\left\{ \left|\log f_{0}\left(X\right)\right|\right\} <\infty$,
we have 
\[
\ell_{k}\left(\psi_{k}\right)=-\mathrm{E}\left\{ \log f_{0}\left(X\right)\right\} +\mathfrak{K}\!\left(f_{0}\Vert f_{k}\left(\cdot;\psi_{k}\right)\right),
\]
and therefore minimising $L_{k}$ over $k$ is equivalent to minimising
the Kullback--Leibler divergence $\min_{\psi_{k}\in\mathbb{S}_{k}}\mathfrak{K}\!\left(f_{0}\Vert f_{k}\left(\cdot;\psi_{k}\right)\right)$
over $k$. Equivalently, 
\[
{\cal K}^{*}=\arg\min_{k\in\left[\bar{k}\right]}\left\{ \min_{\psi_{k}\in\mathbb{S}_{k}}\mathfrak{K}\!\left(f_{0}\Vert f_{k}\left(\cdot;\psi_{k}\right)\right)\right\} ,
\]
the set of KL minimising models among $\left\{ {\cal M}_{k}^{\phi}\right\} _{k\in\left[\bar{k}\right]}$.
Without any further assumptions, we have the following result. 
\begin{prop}
\label{prop:Misspec_convergence}Under A1{*}, A2--A4, and B1, $\lim_{n\to\infty}\mathrm{P}\!\left(\hat{k}_{n}\in{\cal K}^{*}\right)=1$. 
\end{prop}

Proposition~\ref{prop:Misspec_convergence} implies that under misspecification,
any IC that satisfies B1 (and thus also any IC that satisfies B1 and
B2, including the $\epsilon$-BIC and $\nu$-BIC) will eventually
select an order whose best-fitting mixture model is KL-optimal among
the candidate orders in $[\bar{k}]$. Besides the ICs studied in Section~\ref{sec:MainResults},
many common ICs satisfy B1, including the AIC and the BIC. Consequently,
if A1{*} and A2--A4 are satisfied, then all of these ICs admit the
conclusion of Proposition~\ref{prop:Misspec_convergence}.

The difference between Proposition~\ref{prop:Misspec_convergence}
and Theorem~\ref{thm:GenericConsistency} is that Proposition~\ref{prop:Misspec_convergence}
only concludes that $\hat{k}_{n}$ lies in the set ${\cal K}^{*}$
with high probability, rather than that it converges to the smallest
KL-minimising order (i.e., it is not a guarantee of parsimonious order
selection). If one wishes to recover parsimonious order selection
in the misspecified case, such a guarantee can be obtained via the
PanIC framework of \citet{nguyen2024panic} by imposing the stronger
penalty separation assumption (N-B2) from Appendix~B; for instance,
one may use a penalty of the form 
\[
\mathrm{pen}_{k,n}=\alpha\left(k\right)n^{-1/2}\mathrm{Ln}^{\circ\nu}\left(n\right),
\]
with increasing function $\alpha$.

We are not aware of general results proving parsimonious selection
under misspecification for BIC-scale penalties, including the classical
BIC, nor for our larger penalties satisfying B2, in finite mixture
settings. In particular, \citet{keribin2000consistent}, \citet{gassiat2012consistent},
and \citet[Sec.~4.3]{gassiat2018universal} all assume correct specification
in their derivations. In \citet{Patilea2001Convex}, the MLE theory
of \citet[Ch.~7]{van-de-Geer:2000aa} is developed for misspecified
models. These results suggest a possible route to extending parsimonious
order selection to misspecified settings, but carrying out such a
program in the present finite-mixture framework would be nontrivial
and goes beyond the scope of this work. 

\section{Example applications}

\label{sec:Example-applications}

In this section, we provide example applications of Theorem \ref{thm:GenericConsistency}
and Corollary \ref{cor:Main=00003D00003D000020Result} via illustrations
of how the sufficient conditions are verified.

\subsection{Gaussian mixture models}

\label{subsec:Gaussian-mixture-models}

We firstly make our theory concrete by considering the ubiquitous
Gaussian mixture model, defined by taking 
\[
\phi\left(x;\theta\right)=\mathrm{det}\left(2\pi\Sigma\right)^{-1/2}\exp\left\{ -\frac{1}{2}\left(x-\mu\right)^{\top}\Sigma^{-1}\left(x-\mu\right)\right\} ,
\]
with $x\in\mathbb{X}=\mathbb{R}^{p}$ and $\theta=\left(\mu,\Sigma\right)\in\mathbb{T}$,
where $\mathbb{T}$ is a compact subset of $\mathbb{R}^{p}\times\mathbb{S}_{p}$.
Here, $\mathbb{S}_{p}$ denotes the set of positive definite matrices
in $\mathbb{R}^{p\times p}$ and $\left(\cdot\right)^{\top}$ denotes
matrix transposition. Such a compact $\mathbb{T}$ can be constructed,
as per \citet[Sec. B.23]{ritter2014robust}, by taking 
\[
\mathbb{T}=\left\{ \left(\mu,\Sigma\right)\in\mathbb{R}^{p}\times\mathbb{S}_{p}:\left\Vert \mu\right\Vert \le b,c^{-1}\le\lambda_{1}\left(\Sigma\right),\lambda_{p}\left(\Sigma\right)\le c\right\} ,
\]
where $\lambda_{1}\left(\Sigma\right)$ and $\lambda_{p}\left(\Sigma\right)$
are the smallest and largest eigenvalues of $\Sigma$, respectively,
and $b\ge0$ and $c\ge1$ are fixed constants. By this construction,
A2 and A3 are naturally verified.

Next, we can write 
\[
\log\phi\left(x;\theta\right)=-\frac{1}{2}\log\mathrm{det}\left(2\pi\Sigma\right)-\frac{1}{2}\mathrm{tr}\left\{ \left(x-\mu\right)\left(x-\mu\right)^{\top}\Sigma^{-1}\right\} ,
\]
and thus, with $\mathrm{E}\left\Vert X\right\Vert ^{2}<\infty$, we
can verify A4 by taking 
\[
G_{1}\left(x\right)=\frac{1}{2}\max_{\theta\in\mathbb{T}}\left|\log\mathrm{det}\left(2\pi\Sigma\right)\right|+\frac{1}{2}\max_{\theta\in\mathbb{T}}\left|\mathrm{tr}\left\{ \left(x-\mu\right)\left(x-\mu\right)^{\top}\Sigma^{-1}\right\} \right|.
\]
To verify A5{*}, we use standard expressions for the gradient of the
normal density with respect to $\theta$ (see, e.g., \citealt[Sec. 3.4.2]{loos2016analysis}).
Write $\theta=(\mu,\Sigma)$, and let $b\ge0$ and $c\ge1$ be as
in the definition of $\mathbb{T}$. Define $r(x)=\left[\left\Vert x\right\Vert -b\right]_{+}$.
Since $\lambda_{1}(\Sigma)\ge c^{-1}$ and $\lambda_{p}(\Sigma)\le c$,
we have $\mathrm{det}(\Sigma)\ge c^{-p}$ and $\Sigma^{-1}\succeq c^{-1}I$,
and therefore, for every $x\in\mathbb{R}^{p}$ and $\theta\in\mathbb{T}$,
\[
\phi\left(x;\theta\right)\le(2\pi)^{-p/2}c^{p/2}\exp\left\{ -\frac{1}{2c}\left\Vert x-\mu\right\Vert ^{2}\right\} \le(2\pi)^{-p/2}c^{p/2}\exp\left\{ -\frac{1}{2c}r(x)^{2}\right\} .
\]
Moreover, for the derivatives with respect to $\mu$ and $\Sigma$,
we have 
\[
\frac{\partial\phi\left(x;\theta\right)}{\partial\mu}=\Sigma^{-1}\left(x-\mu\right)\phi\left(x;\theta\right),
\]
so that $\left\Vert \partial\phi(x;\theta)/\partial\mu\right\Vert \le c\left\Vert x-\mu\right\Vert \phi(x;\theta)$.
Similarly, the partial derivatives with respect to the coordinates
of $\Sigma$ can be written as $\phi(x;\theta)$ multiplied by linear
combinations of entries of $\Sigma^{-1}$ and of $\Sigma^{-1}(x-\mu)(x-\mu)^{\top}\Sigma^{-1}$,
and hence there exists a constant $C_{0}>0$ depending only on $b$,
$c$, and $p$, such that, for all $x$ and $\theta\in\mathbb{T}$,
\[
\max_{j\in\left[m\right]}\left|\frac{\partial\phi\left(x;\theta\right)}{\partial\vartheta_{j}}\right|\le C_{0}\left\{ 1+\left\Vert x-\mu\right\Vert ^{2}\right\} \phi\left(x;\theta\right).
\]
Combining the preceding results and using $\left\Vert x-\mu\right\Vert \le\left\Vert x\right\Vert +b$,
we obtain 
\[
\max_{\theta\in\mathbb{T}}\phi\left(x;\theta\right)+\max_{j\in\left[m\right]}\max_{\theta\in\mathbb{T}}\left|\frac{\partial\phi\left(x;\theta\right)}{\partial\vartheta_{j}}\right|\le C_{1}\left\{ 1+\left\Vert x\right\Vert ^{2}\right\} \exp\left\{ -\frac{1}{2c}r(x)^{2}\right\} ,
\]
for a constant $C_{1}>0$ depending only on $b$, $c$, and $p$.
Therefore, A5{*} is satisfied by taking 
\[
G_{2}\left(x\right)=C_{1}\left\{ 1+\left\Vert x\right\Vert ^{2}\right\} \exp\left\{ -\frac{1}{2c}r(x)^{2}\right\} ,
\]
noting that $G_{2}\in{\cal L}_{1}\left(\mathfrak{m}\right)$ since
it has a Gaussian tail. Since $\mathrm{E}\left\Vert X\right\Vert ^{2}<\infty$
holds under A1 for Gaussian mixture models, the conclusion of Theorem
\ref{thm:GenericConsistency} holds for Gaussian mixture models whenever
A1 holds. This is the archetypal example of \citet{keribin2000consistent}
and verifies the assumptions in the text.

\subsection{Laplace mixture models}

\label{subsec:Laplace-mixture-models}

The Laplace mixture model is defined by taking the density of the
Laplace distribution: 
\begin{equation}
\phi\left(x;\theta\right)=\frac{\gamma}{2}\exp\left\{ -\gamma\left|x-\mu\right|\right\} ,\label{eq:Laplace=00003D00003D000020Density}
\end{equation}
with $x\in\mathbb{R}$ and $\theta=\left(\mu,\gamma\right)\in\mathbb{T}$,
where we can choose 
\[
\mathbb{T}=\left\{ \left(\mu,\gamma\right)\in\mathbb{R}\times\mathbb{R}_{>0}:\left|\mu\right|\le b,c^{-1}\le\gamma\le c\right\} ,
\]
which is convex and compact for each $b>0$ and $c>1$. Such models
have been considered, for example, by \citet{cord2006feature}, \citet{mitianoudis2007batch},
and \citet{rabbani2008image}. Here, A2 and A3 hold by construction,
although it is noteworthy that \ref{eq:Laplace=00003D00003D000020Density}
is a nondifferentiable function of $\mu$.

Similarly to Section \ref{subsec:Gaussian-mixture-models}, we can
take 
\[
G_{1}\left(x\right)=\max_{\theta\in\mathbb{T}}\left|\log\left(\frac{\gamma}{2}\right)\right|+\max_{\theta\in\mathbb{T}}\left\{ \gamma\left|x-\mu\right|\right\} 
\]
to verify A4, under the condition that $\mathrm{E}\left|X\right|<\infty$.
Since $\phi\left(x;\theta\right)$ is not differentiable in $\mu$,
we seek to verify A5 instead of A5{*}. Write $\theta=(\mu,\gamma)$
and let $b>0$ and $c>1$ be as in the definition of $\mathbb{T}$.
Define $r(x)=\left[|x|-b\right]_{+}$. Then, for every $x\in\mathbb{R}$
and $\theta\in\mathbb{T}$, we have $|x-\mu|\ge r(x)$ and $\gamma\ge c^{-1}$,
so that 
\[
\phi\left(x;\theta\right)=\frac{\gamma}{2}\exp\left\{ -\gamma|x-\mu|\right\} \le\frac{c}{2}\exp\left\{ -\frac{1}{c}r(x)\right\} .
\]
To bound differences in $\mu$, fix $\gamma\in[c^{-1},c]$ and note
that $\bigl||x-\mu_{1}|-|x-\mu_{2}|\bigr|\le|\mu_{1}-\mu_{2}|$. Using
the mean value theorem for $t\mapsto\exp\{-\gamma t\}$, we obtain
\[
\left|\exp\left\{ -\gamma|x-\mu_{1}|\right\} -\exp\left\{ -\gamma|x-\mu_{2}|\right\} \right|\le\gamma|\mu_{1}-\mu_{2}|\exp\left\{ -\gamma r(x)\right\} \le c|\mu_{1}-\mu_{2}|\exp\left\{ -\frac{1}{c}r(x)\right\} ,
\]
and therefore 
\[
\left|\phi\left(x;(\mu_{1},\gamma)\right)-\phi\left(x;(\mu_{2},\gamma)\right)\right|\le\frac{c^{2}}{2}|\mu_{1}-\mu_{2}|\exp\left\{ -\frac{1}{c}r(x)\right\} .
\]
To bound differences in $\gamma$, fix $\mu\in[-b,b]$ and use the
mean value theorem in $\gamma$. Since 
\[
\frac{\partial}{\partial\gamma}\phi\left(x;(\mu,\gamma)\right)=\frac{1}{2}\left\{ 1-\gamma|x-\mu|\right\} \exp\left\{ -\gamma|x-\mu|\right\} ,
\]
we have 
\[
\left|\frac{\partial}{\partial\gamma}\phi\left(x;(\mu,\gamma)\right)\right|\le\frac{1}{2}\left\{ 1+\gamma|x-\mu|\right\} \exp\left\{ -\gamma|x-\mu|\right\} \le\frac{1}{2}\left\{ 1+c(|x|+b)\right\} \exp\left\{ -\frac{1}{c}r(x)\right\} .
\]
Combining the preceding bounds yields, for all $\theta_{1},\theta_{2}\in\mathbb{T}$,
\[
\left|\phi\left(x;\theta_{1}\right)-\phi\left(x;\theta_{2}\right)\right|\le L_{\phi}\left(x\right)\left\Vert \theta_{1}-\theta_{2}\right\Vert _{1},
\]
where we can take 
\[
L_{\phi}\left(x\right)=C_{0}\left\{ 1+|x|\right\} \exp\left\{ -\frac{1}{c}r(x)\right\} 
\]
for a suitable constant $C_{0}>0$ depending only on $b$ and $c$.
Moreover, the bound on $\phi(x;\theta)$ above shows that $\max_{\theta\in\mathbb{T}}\phi(x;\theta)\le L_{\phi}(x)$
for $C_{0}$ sufficiently large, and thus A5 is satisfied by taking
$G_{2}(x)=2L_{\phi}(x)$. Since $G_{2}\in{\cal L}_{1}\left(\mathfrak{m}\right)$
by the exponential tails, we have the conclusion of Theorem \ref{thm:GenericConsistency}
for Laplace mixture models whenever A1 is satisfied.

In this example, we observe that since $\phi$ is not differentiable
for all $\theta\in\mathbb{T}$, the conditions (P1-b), (P2), and (P3)
of \citet{keribin2000consistent}, defined via the partial derivatives
of $\phi$, cannot be verified.

\subsection{Student $t$ mixture models}

\label{subsec:Student--mixture-example}

The Student $t$ mixture model is defined by taking the density of
the Student $t$ distribution: 
\begin{equation}
\phi(x;\theta)=c_{v}\gamma\left(1+\frac{\gamma^{2}(x-\mu)^{2}}{v}\right)^{-\frac{v+1}{2}},\label{eq:t-density-variable-v}
\end{equation}
with $x\in\mathbb{R}$ and $\theta=(\mu,\gamma,v)\in\mathbb{T}$,
where 
\[
c_{v}=\frac{\Gamma\left(\frac{v+1}{2}\right)}{\Gamma\left(\frac{v}{2}\right)\sqrt{v\pi}}.
\]
We can choose 
\[
\mathbb{T}=\left\{ (\mu,\gamma,v)\in\mathbb{R}\times\mathbb{R}_{>0}\times\mathbb{R}_{>0}:|\mu|\le b,\ c^{-1}\le\gamma\le c,\ \underline{v}\le v\le\overline{v}\right\} ,
\]
which is compact for each $b>0$, $c>1$, and $0<\underline{v}<\overline{v}<\infty$.
Mixtures of such distributions are considered in the works of \citet{peel2000robust}
and \citet{ForbesWraith2014}, among others.

Here, A2 and A3 hold by construction, and \eqref{eq:t-density-variable-v}
is continuously differentiable in $\theta$ on $\mathbb{T}$. Similarly
to Section \ref{subsec:Gaussian-mixture-models}, we can write 
\[
\log\phi(x;\theta)=\log c_{v}+\log\gamma-\frac{v+1}{2}\log\!\left(1+\frac{\gamma^{2}(x-\mu)^{2}}{v}\right),
\]
and thus we can take 
\[
G_{1}(x)=\max_{\theta\in\mathbb{T}}\left|\log(c_{v}\gamma)\right|+\frac{\overline{v}+1}{2}\log\!\left(1+\frac{\bar{c}^{2}(|x|+b)^{2}}{\underline{v}}\right),\qquad\bar{c}=\max\{1/c,c\},
\]
to verify A4, under the condition that $\mathrm{E}X^{2}<\infty$ (using
$\log(1+u)\le u$ and $|x-\mu|\le|x|+b$).

Since $\phi(x;\theta)$ is differentiable in $\theta$, we seek to
verify A5$^{*}$. Define $r(x)=\left[|x|-b\right]_{+}$, so that $|x-\mu|\ge r(x)$
for all $|\mu|\le b$. First, since $\gamma\le c$ and $v\in[\underline{v},\overline{v}]$,
we have $\max_{\theta\in\mathbb{T}}(c_{v}\gamma)<\infty$. Moreover,
since $\gamma\ge c^{-1}$ and $v\le\overline{v}$, we have 
\[
1+\frac{\gamma^{2}(x-\mu)^{2}}{v}\ge1+\frac{r(x)^{2}}{c^{2}\overline{v}}.
\]
Since $v\ge\underline{v}$, the map $t\mapsto t^{-(v+1)/2}$ is bounded
above by $t^{-(\underline{v}+1)/2}$ for $t\ge1$, and hence there
exists a constant $C_{0}>0$ depending only on $b$, $c$, $\underline{v}$,
and $\overline{v}$, such that 
\[
\max_{\theta\in\mathbb{T}}\phi(x;\theta)\le C_{0}\left(1+\frac{r(x)^{2}}{c^{2}\overline{v}}\right)^{-\frac{\underline{v}+1}{2}}.
\]
Next, for the derivatives, we recall that for $\theta=(\mu,\gamma,v)$,
\[
\frac{\partial\phi(x;\theta)}{\partial\mu}=(v+1)\frac{\gamma^{2}(x-\mu)}{v+\gamma^{2}(x-\mu)^{2}}\phi(x;\theta),\frac{\partial\phi(x;\theta)}{\partial\gamma}=\left\{ \frac{1}{\gamma}-(v+1)\frac{\gamma(x-\mu)^{2}}{v+\gamma^{2}(x-\mu)^{2}}\right\} \phi(x;\theta),
\]
and 
\[
\frac{\partial\phi(x;\theta)}{\partial v}=\phi(x;\theta)\left\{ \frac{\partial}{\partial v}\log c_{v}-\frac{1}{2}\log\left(1+\frac{\gamma^{2}(x-\mu)^{2}}{v}\right)+\frac{v+1}{2}\frac{\gamma^{2}(x-\mu)^{2}}{v\{v+\gamma^{2}(x-\mu)^{2}\}}\right\} .
\]
Since $t\mapsto t/(v+t^{2})$ is maximised at $t=\sqrt{v}$ and $\gamma\in[c^{-1},c]$,
we have $\gamma^{2}|x-\mu|/\{v+\gamma^{2}(x-\mu)^{2}\}\le c/(2\sqrt{\underline{v}})$,
and thus 
\[
\left|\frac{\partial\phi(x;\theta)}{\partial\mu}\right|\le C_{1}\phi(x;\theta),\qquad\left|\frac{\partial\phi(x;\theta)}{\partial\gamma}\right|\le C_{2}\phi(x;\theta),
\]
for constants $C_{1},C_{2}>0$ depending only on $c$, $\underline{\upsilon}$,
and $\overline{\upsilon}$. Furthermore, $\max_{v\in[\underline{v},\overline{v}]}|\partial_{v}\log c_{v}|<\infty$
and $\max_{v\in[\underline{v},\overline{v}]}(v+1)/(2v)<\infty$, and
therefore there exists a constant $C_{3}>0$ such that 
\[
\left|\frac{\partial\phi(x;\theta)}{\partial v}\right|\le\left\{ C_{3}+\frac{1}{2}\log\left(1+\frac{\bar{c}^{2}(|x|+b)^{2}}{\underline{v}}\right)\right\} \phi(x;\theta),\qquad\bar{c}=\max\{1/c,c\}.
\]
Combining these bounds with the envelope for $\max_{\theta}\phi(x;\theta)$
above, we obtain that 
\[
\max_{\theta\in\mathbb{T}}\phi(x;\theta)+\max_{j\in\left[m\right]}\max_{\theta\in\mathbb{T}}\left|\frac{\partial\phi(x;\theta)}{\partial\vartheta_{j}}\right|\le G_{2}(x),
\]
where we can take 
\[
G_{2}(x)=C_{4}\left\{ 1+\log\left(1+\frac{\bar{c}^{2}(|x|+b)^{2}}{\underline{v}}\right)\right\} \left(1+\frac{r(x)^{2}}{c^{2}\overline{v}}\right)^{-\frac{\underline{v}+1}{2}}
\]
for a suitable constant $C_{4}>0$. Since $\underline{v}>0$, the
right-hand side has a polynomial tail of order $|x|^{-(\underline{v}+1)}$
up to a logarithmic factor, and hence $G_{2}\in{\cal L}_{1}\left(\mathfrak{m}\right)$.
Therefore, A5$^{*}$ is satisfied.

Since $\mathrm{E}X^{2}<\infty$ holds under A1 whenever the smallest
degrees of freedom among the true components satisfies $v_{0}>2$,
we have the conclusion of Theorem \ref{thm:GenericConsistency} for
Student $t$ mixture models whenever A1 is satisfied with $v_{0}>2$.

In this example, we observe that although $\phi$ is smooth, allowing
$v$ to vary can prevent verification of the integrability requirement
in (P1-b) of \citet{keribin2000consistent} (see Appendix~B). Let
$v_{0}$ denote the smallest degrees of freedom among the true components,
so that $f_{0}(x)\asymp|x|^{-(v_{0}+1)}$ as $|x|\to\infty$, and
fix any candidate component with degrees of freedom $v\in[\underline{v},\overline{v}]$,
for which $\phi(x;\theta)\asymp|x|^{-(v+1)}$. Then 
\[
\int_{-\infty}^{\infty}\left(\frac{\phi(x;\theta)}{f_{0}(x)}\right)^{3}f_{0}(x)\,dx=\int_{-\infty}^{\infty}\frac{\phi(x;\theta)^{3}}{f_{0}(x)^{2}}\,dx\asymp\int_{-\infty}^{\infty}|x|^{-(3v-2v_{0}+1)}\,dx,
\]
which diverges whenever $v\le\frac{2}{3}v_{0}$. Consequently, if
the parameter set $\mathbb{T}$ contains degrees of freedom $v\le2v_{0}/3$
(in particular, if $\underline{v}\le2v_{0}/3$), then (P1-b) cannot
be verified and the conditions of \citet{keribin2000consistent} are
not applicable in such $t$ mixture settings. For instance, if $v_{0}=6$
then any model class allowing $v\le4$ violates (P1-b).

\subsection{Mixture of regression models}

\label{subsec:Mixture-of-regression}

One can extend upon finite mixture models via the mixture of regression
construction which was introduced in \citet{quandt1972new} and subsequently
studied, for example, in \citet{jones1992fitting}, \citet{wedel1995mixture},
\citet{naik2007extending}, \citet{hafidi2010kullback}, \citet{song2014robust},
\citet{depraetere2014order}, and \citet{hui2015order}, among many
others. A good recent account of mixture regression models appears
in \citet{yao2024mixture}. Suppose that we can write $X=\left(U,Y\right)$,
with $\left(U,Y\right):\Omega\to\mathbb{U}\times\mathbb{Y}=\mathbb{X}$
and that we can write the PDF of $X$ in the form 
\begin{equation}
\phi\left(x;\theta\right)=\rho\left(y|u;\theta\right)\varphi\left(u\right),\label{eq:=00003D00003D000020conditionaldecomp}
\end{equation}
for each $x=\left(u,y\right)\in\mathbb{X}$, where $\rho:\mathbb{U}\times\mathbb{Y}\times\mathbb{T}\to\mathbb{R}_{\ge0}$
characterises a conditional PDF in the sense that for every $\left(u,\theta\right)\in\mathbb{U}\times\mathbb{T}$,
$\int_{\mathbb{Y}}\rho\left(y|u;\theta\right)\mathfrak{m}_{2}\left(\mathrm{d}y\right)=1$,
and $\varphi$ is the generative PDF of $U$ with respect to the dominating
measure $\mathfrak{m}_{1}$, where $\mathfrak{m}=\mathfrak{m}_{1}\times\mathfrak{m}_{2}$.
Note here that $\varphi$ is not parameterised by $\theta$, nor is
it given any structure and is merely assumed to exist.

Naturally, when fitting mixture of regression models, one considers
only the conditional part of \ref{eq:=00003D00003D000020conditionaldecomp}.
As such, with 
\[
{\cal N}_{k}^{\phi}=\left\{ \mathbb{U}\times\mathbb{Y}\ni\left(u,y\right)\mapsto\mathfrak{f}_{k}\left(y|u;\psi_{k}\right)=\sum_{z=1}^{k}\pi_{z}\rho\left(y|u;\theta_{z}\right):\theta_{z}\in\mathbb{T},\pi_{z}\in\left[0,1\right],\sum_{z=1}^{k}\pi_{z}=1,z\in\left[k\right]\right\} ,
\]
for each $k\in\left[\bar{k}\right]$, one seeks to identify the smallest
$k_{0}\in\left[\bar{k}\right]$, for which $f_{0}=\mathfrak{f}_{0}\varphi$,
where $\mathfrak{f}_{0}\in{\cal N}_{k_{0}}^{\phi}$, or more concisely
and in conforming to Section \ref{sec:Introduction}, $f_{0}\in{\cal M}_{k_{0}}^{\phi}$.
To this end, the usual risk, the average negative log-densities $\ell_{k}\left(\psi_{k}\right)=-\mathrm{E}\left\{ \log f_{k}\left(X;\psi_{k}\right)\right\} $,
for each $k\in\left[\bar{k}\right]$ with $f_{k}\in{\cal M}_{k}^{\phi}$,
is not estimated by the natural estimator $\ell_{k,n}\left(\psi_{k}\right)$,
but instead by the empirical average negative log-conditional likelihood
\[
\mathfrak{L}_{k,n}\left(\psi_{k}\right)=-\frac{1}{n}\sum_{i=1}^{n}\log\mathfrak{f}_{k}\left(Y_{i}|U_{i};\psi_{k}\right),
\]
where $\mathfrak{f}_{k}\in{\cal N}_{k}^{\phi}$.

A useful fact is that, since $f_{k}\left(u,y;\psi_{k}\right)=\mathfrak{f}_{k}\left(y|u;\psi_{k}\right)\varphi\left(u\right)$,
we have, for each $k\in\left[\bar{k}\right]$ and $\psi_{k}\in\mathbb{S}_{k}$,
\[
\ell_{k,n}\left(\psi_{k}\right)=-P_{n}\log f_{k}\left(\cdot;\psi_{k}\right)=-P_{n}\log\mathfrak{f}_{k}\left(\cdot;\psi_{k}\right)-P_{n}\log\varphi=\mathfrak{L}_{k,n}\left(\psi_{k}\right)-\frac{1}{n}\sum_{i=1}^{n}\log\varphi\left(U_{i}\right).
\]
Taking expectations yields $\mathrm{E}\left\{ \mathfrak{L}_{k,n}\left(\psi_{k}\right)\right\} =\ell_{k}\left(\psi_{k}\right)+\mathrm{E}\left\{ \log\varphi\left(U\right)\right\} $,
where the $\mathrm{E}\left\{ \log\varphi\left(U\right)\right\} $
term is the same across all $k$ and has no influence on comparisons
between risks for different values of $k$ or different models within
${\cal M}_{k}^{\phi}$ or ${\cal N}_{k}^{\phi}$, since $\varphi$
has no parameters. Moreover, the sample-level identity above shows
that $\ell_{k,n}\left(\psi_{k}\right)$ and $\mathfrak{L}_{k,n}\left(\psi_{k}\right)$
differ by an additive term that does not depend on $k$ or $\psi_{k}$.
Consequently, Lemma \ref{lem:Convergenceofmin} and Lemma \ref{lem:Rateofmin}
remain valid with $\mathfrak{L}_{k,n}\left(\psi_{k}\right)$ in place
of $\ell_{k,n}\left(\psi_{k}\right)$.

Furthermore, since $P_{n}\log f=P_{n}\log\mathfrak{f}+P_{n}\log\varphi$,
maximisers of the joint likelihood over ${\cal M}_{k}^{\phi}$ correspond
to maximisers of the conditional likelihood over ${\cal N}_{k}^{\phi}$.
In particular, if 
\[
\hat{f}_{k,n}\in\underset{f\in{\cal M}_{k}^{\phi}}{\arg\max}P_{n}\log f,
\]
then there exists 
\[
\hat{\mathfrak{f}}_{k,n}\in\underset{\mathfrak{f}\in{\cal N}_{k}^{\phi}}{\arg\max}P_{n}\log\mathfrak{f},
\]
such that $\hat{f}_{k,n}=\hat{\mathfrak{f}}_{k,n}\varphi$. If, in
addition, $\varphi\left(u\right)>0$ for all $u\in\mathbb{U}$, then
$\log(\hat{f}_{k,n}/f_{0})=\log(\hat{\mathfrak{f}}_{k,n}/\mathfrak{f}_{0})$,
and Proposition \ref{prop:MainTechProp} is directly applicable to
the maximum conditional likelihood estimator. Together, these facts
imply that we can consistently estimate $k_{0}$ by the conditional
version of $\hat{k}_{n}$: 
\[
\tilde{k}_{n}=\min\arg\min_{k\in\left[\bar{k}\right]}\left\{ \min_{\psi_{k}\in\mathbb{S}_{k}}\mathfrak{L}_{k,n}\left(\psi_{k}\right)+\mathrm{pen}_{k,n}\right\} .
\]
To be precise, if $\phi$ has form \ref{eq:=00003D00003D000020conditionaldecomp},
and Assumptions A1--A5 hold for the corresponding joint mixture class
${\cal M}_{k}^{\phi}$, then Theorem \ref{thm:GenericConsistency}
holds with $\hat{k}_{n}$ replaced by $\tilde{k}_{n}$, with no additional
assumptions beyond A1--A5. This example is general and includes many
models considered in the literature, as noted at the start. In particular,
under compactness restrictions on $\mathbb{T}$ and mild integrability
conditions with respect to $\mathfrak{m}_{1}$, the verification of
A4 and A5 (or A5$^{*}$) for regression kernels $\rho$ proceeds exactly
as in Sections \ref{subsec:Laplace-mixture-models} and \ref{subsec:Student--mixture-example},
with $x$ replaced by $y$ and with the regression location depending
on $u$. It is possible to verify these assumptions in settings with
well-behaved $\phi$, using the results of \citet{keribin2000consistent},
such as in the case of normal mixture regression models (e.g., \citealp{quandt1972new}
and \citealp{jones1992fitting}). However, like in the Laplace mixture
example, the Laplace mixture of regression models, such as that of
\citet{song2014robust}, have non-differentiable PDFs and thus violate
the same conditions of \citet{keribin2000consistent} as in the previous
example. Mixture of Student $t$ regression models such as those of
\citet{GalimbertiSoffritti2014Multivariate} and \citet{yao2014robust}
will also incur the same problems as described in Section \ref{subsec:Student--mixture-example}
if the degree of freedom parameters considered are too small. 
\begin{rem}
Since $\varphi$ is not parameterised by $\theta$, it does not influence
the maximisation of the likelihood over $\psi_{k}$, and it cancels
from likelihood ratios $\log(\hat{f}_{k,n}/f_{0})$ whenever $\varphi\left(u\right)>0$
for all $u\in\mathbb{U}$. However, $\varphi$ may enter the verification
of A1--A5 through integrability requirements with respect to the
dominating measure $\mathfrak{m}=\mathfrak{m}_{1}\times\mathfrak{m}_{2}$.
A convenient convention is to take $\mathfrak{m}_{1}$ to be the marginal
law of $U$, in which case $\varphi=1$ and $\phi\left(x;\theta\right)=\rho\left(y|u;\theta\right)$
is a density with respect to $\mathfrak{m}=\mathfrak{m}_{1}\times\mathfrak{m}_{2}$.
Under this convention, the verification of A1--A5 reduces to verifying
them for $\rho$. 
\end{rem}

\subsection{Numerical simulations}

Here, we provide some numerical evidence towards the comparative performance
of the $\epsilon$-BIC and $\nu$-BIC versus alternatives such as
the AIC, BIC, and PanIC. We consider four scenarios corresponding
to our technical examples in Sections \ref{subsec:Laplace-mixture-models}
and \ref{subsec:Student--mixture-example}, which present pathologies
that do not verify the hypotheses of \citet{keribin2000consistent}.
In each case, we simulate IID data $\mathbf{X}_{n}$ from the measure
with density $f_{0}$, a mixture with true number of components $k_{0}$
of base density $\phi$ with true parameter $\psi_{k_{0}}^{*}$. Then,
we estimate $\hat{k}_{n}$ by considering mixtures of up to $K>k_{0}$
components of the same form, using data of sizes $n\in\left\{ 100,1000,10000\right\} $.
The four scenarios considered are described in Table \ref{tab:Description-of-simulation}.

\begin{table}
\caption{Description of simulation scenarios.}
\label{tab:Description-of-simulation}

\centering{}%
\begin{tabular}{|ccccc|}
\hline 
Scenario  & $\phi$  & $k_{0}$  & $\psi_{k_{0}}^{*}$  & $K$\tabularnewline
\hline 
\hline 
S-L1  & Laplace \eqref{eq:Laplace=00003D00003D000020Density}  & 2  & $\begin{array}{c}
\left(\pi_{1}^{*},\pi_{2}^{*}\right)=\left(1/2,1/2\right)\\
\left(\mu_{1}^{*},\mu_{2}^{*}\right)=\left(0,2\right)\\
\left(\gamma_{1}^{*},\gamma_{2}^{*}\right)=\left(1,4\right)
\end{array}$  & 3\tabularnewline
\hline 
S-L2  & Laplace \eqref{eq:Laplace=00003D00003D000020Density}  & 3  & $\begin{array}{c}
\left(\pi_{1}^{*},\pi_{2}^{*},\pi_{3}^{*}\right)=\left(1/4,1/2,1/4\right)\\
\left(\mu_{1}^{*},\mu_{2}^{*},\mu_{3}^{*}\right)=\left(-3,0,3\right)\\
\left(\gamma_{1}^{*},\gamma_{2}^{*},\gamma_{3}^{*}\right)=\left(1,1,1\right)
\end{array}$  & 5\tabularnewline
\hline 
S-T1  & Student-$t$ \eqref{eq:t-density-variable-v}  & 2  & $\begin{array}{c}
\left(\pi_{1}^{*},\pi_{2}^{*}\right)=\left(1/2,1/2\right)\\
\left(\mu_{1}^{*},\mu_{2}^{*}\right)=\left(0,2\right)\\
\left(\gamma_{1}^{*},\gamma_{2}^{*}\right)=\left(1,4\right)\\
\left(v_{1}^{*},v_{2}^{*}\right)=\left(5,5\right)
\end{array}$  & 3\tabularnewline
\hline 
S-T2  & Student-$t$ \eqref{eq:t-density-variable-v}  & 3  & $\begin{array}{c}
\left(\pi_{1}^{*},\pi_{2}^{*},\pi_{3}^{*}\right)=\left(1/4,1/2,1/4\right)\\
\left(\mu_{1}^{*},\mu_{2}^{*},\mu_{3}^{*}\right)=\left(-3,0,3\right)\\
\left(\gamma_{1}^{*},\gamma_{2}^{*},\gamma_{3}^{*}\right)=\left(1,1,1\right)\\
\left(v_{1}^{*},v_{2}^{*},v_{3}^{*}\right)=\left(5,5,5\right)
\end{array}$  & 5\tabularnewline
\hline 
\end{tabular}
\end{table}

We repeat each scenario 200 times, and evaluate the proportions of
times $\hat{k}_{n}=k$, for each $k\in\left[K\right]$ and for each
IC. Each IC is considered to have penalties in the form $\mathrm{pen}_{k,n}=\alpha\left(k\right)r_{n}$,
where $r_{n}$ is the shape of the penalty. For the $\epsilon$-BIC,
we consider $\epsilon=0.1$ and $\epsilon=0.02$. For the $\nu$-BIC,
we consider $\nu=1$ (also corresponding to the $\epsilon$-BIC with
$\epsilon=1$), and $\nu=3$. For PanIC, we use the Sin--White information
criterion form considered in \citet{nguyen2024panic}. The shapes
for the considered penalties are thus: AIC: $r_{n}=2n^{-1}$; BIC:
$r_{n}=n^{-1}\log n$; $\epsilon$-BIC: $r_{n}=n^{-1}\log^{1+\epsilon}n$;
$\nu$-BIC: $r_{n}=n^{-1}\mathrm{Ln}^{\circ\nu}\left(n\right)\log n$;
PanIC: $r_{n}=n^{-1/2}\sqrt{\mathrm{Ln}^{\circ2}\left(n\right)}$.
The constant multiples $\alpha\left(k\right)$ that we consider are
of forms (I) $\mathrm{dim}\left(\mathbb{S}_{k}\right)/2$ (the classic
BIC penalty), (II) $\left\{ \mathrm{dim}\left(\mathbb{S}_{k}\right)\right\} ^{2}/2$,
(III) $k$ (only the number of mixture components), and (IV) $k^{2}$.
The results are recorded in Tables \ref{tab:S-L1}--\ref{tab:S-T2}.

\begin{table}
\caption{Results from 200 replications of scenario S-L1. The $\hat{k}_{n}=\dots$
columns indicate the proportion of times $\hat{k}_{n}=k$, where $*$
marks the true number of components $k_{0}$. Bold text indicates
the highest proportion of correct identifications of $k_{0}$. }
\label{tab:S-L1}

\centering{}%
\begin{tabular}{|cc|ccc|ccc|ccc|}
\hline 
 &  & \multicolumn{3}{c|}{$n=100$} & \multicolumn{3}{c|}{$n=1000$} & \multicolumn{3}{c|}{$n=10000$}\tabularnewline
Pen. shape  & $\alpha\left(k\right)$  & $\hat{k}_{n}=1$  & $2^{*}$  & $3$  & $1$  & $2^{*}$  & $3$  & $1$  & $2^{*}$  & $3$\tabularnewline
\hline 
\hline 
AIC  & I  & 0.070  & 0.630  & 0.300  & 0.000  & 0.855  & 0.145  & 0.000  & 0.885  & 0.115\tabularnewline
 & II  & 0.995  & 0.005  & 0.000  & 0.000  & \textbf{1.000}  & 0.000  & 0.000  & \textbf{1.000}  & 0.000\tabularnewline
 & III  & 0.000  & 0.280  & 0.720  & 0.000  & 0.455  & 0.545  & 0.000  & 0.420  & 0.580\tabularnewline
 & IV  & 0.085  & \textbf{0.810}  & 0.105  & 0.000  & 0.960  & 0.040  & 0.000  & 0.960  & 0.040\tabularnewline
\hline 
BIC  & I  & 0.495  & 0.495  & 0.010  & 0.000  & \textbf{1.000}  & 0.000  & 0.000  & \textbf{1.000}  & 0.000\tabularnewline
 & II  & 1.000  & 0.000  & 0.000  & 0.820  & 0.180  & 0.000  & 0.000  & \textbf{1.000}  & 0.000\tabularnewline
 & III  & 0.030  & 0.555  & 0.415  & 0.000  & 0.905  & 0.095  & 0.000  & 0.950  & 0.050\tabularnewline
 & IV  & 0.500  & 0.500  & 0.000  & 0.000  & \textbf{1.000}  & 0.000  & 0.000  & \textbf{1.000}  & 0.000\tabularnewline
\hline 
$\epsilon$-BIC  & I  & 0.520  & 0.475  & 0.005  & 0.000  & \textbf{1.000}  & 0.000  & 0.000  & \textbf{1.000}  & 0.000\tabularnewline
$(\epsilon=0.02)$  & II  & 1.000  & 0.000  & 0.000  & 0.885  & 0.115  & 0.000  & 0.000  & \textbf{1.000}  & 0.000\tabularnewline
 & III  & 0.040  & 0.555  & 0.405  & 0.000  & 0.910  & 0.090  & 0.000  & 0.955  & 0.045\tabularnewline
 & IV  & 0.520  & 0.480  & 0.000  & 0.000  & \textbf{1.000}  & 0.000  & 0.000  & \textbf{1.000}  & 0.000\tabularnewline
\hline 
$\epsilon$-BIC  & I  & 0.575  & 0.425  & 0.000  & 0.000  & \textbf{1.000}  & 0.000  & 0.000  & \textbf{1.000}  & 0.000\tabularnewline
$(\epsilon=0.1)$  & II  & 1.000  & 0.000  & 0.000  & 0.990  & 0.010  & 0.000  & 0.000  & \textbf{1.000}  & 0.000\tabularnewline
 & III  & 0.050  & 0.590  & 0.360  & 0.000  & 0.935  & 0.065  & 0.000  & 0.975  & 0.025\tabularnewline
 & IV  & 0.575  & 0.425  & 0.000  & 0.000  & \textbf{1.000}  & 0.000  & 0.000  & \textbf{1.000}  & 0.000\tabularnewline
\hline 
$\nu$-BIC  & I  & 1.000  & 0.000  & 0.000  & 0.805  & 0.195  & 0.000  & 0.000  & \textbf{1.000}  & 0.000\tabularnewline
$\left(\nu=1\right)$  & II  & 1.000  & 0.000  & 0.000  & 1.000  & 0.000  & 0.000  & 1.000  & 0.000  & 0.000\tabularnewline
 & III  & 0.790  & 0.210  & 0.000  & 0.000  & \textbf{1.000}  & 0.000  & 0.000  & \textbf{1.000}  & 0.000\tabularnewline
 & IV  & 1.000  & 0.000  & 0.000  & 0.805  & 0.195  & 0.000  & 0.000  & \textbf{1.000}  & 0.000\tabularnewline
\hline 
$\nu$-BIC  & I  & 0.495  & 0.495  & 0.010  & 0.000  & \textbf{1.000}  & 0.000  & 0.000  & \textbf{1.000}  & 0.000\tabularnewline
$\left(\nu=3\right)$  & II  & 1.000  & 0.000  & 0.000  & 0.820  & 0.180  & 0.000  & 0.000  & \textbf{1.000}  & 0.000\tabularnewline
 & III  & 0.030  & 0.555  & 0.415  & 0.000  & 0.905  & 0.095  & 0.000  & 0.950  & 0.050\tabularnewline
 & IV  & 0.500  & 0.500  & 0.000  & 0.000  & \textbf{1.000}  & 0.000  & 0.000  & \textbf{1.000}  & 0.000\tabularnewline
\hline 
PanIC  & I  & 0.985  & 0.015  & 0.000  & 0.680  & 0.320  & 0.000  & 0.000  & \textbf{1.000}  & 0.000\tabularnewline
 & II  & 1.000  & 0.000  & 0.000  & 1.000  & 0.000  & 0.000  & 1.000  & 0.000  & 0.000\tabularnewline
 & III  & 0.445  & 0.530  & 0.025  & 0.000  & \textbf{1.000}  & 0.000  & 0.000  & \textbf{1.000}  & 0.000\tabularnewline
 & IV  & 0.985  & 0.015  & 0.000  & 0.680  & 0.320  & 0.000  & 0.000  & \textbf{1.000}  & 0.000\tabularnewline
\hline 
\end{tabular}
\end{table}

\begin{sidewaystable}
\caption{Results from 200 replications of scenario S-L2. The $\hat{k}_{n}=\dots$
columns indicate the proportion of times $\hat{k}_{n}=k$, where $*$
marks the true number of components $k_{0}$. Bold text indicates
the highest proportion of correct identifications of $k_{0}$. }
\label{tab:S-L2}

\centering{}%
\begin{tabular}{|cc|ccccc|ccccc|ccccc|}
\hline 
 &  & \multicolumn{5}{c|}{$n=100$} & \multicolumn{5}{c|}{$n=1000$} & \multicolumn{5}{c|}{$n=10000$}\tabularnewline
Pen. shape  & $\alpha\left(k\right)$  & $k=1$  & $2$  & $3^{*}$  & $4$  & $5$  & $1$  & $2$  & $3^{*}$  & $4$  & $5$  & $1$  & $2$  & $3^{*}$  & $4$  & $5$\tabularnewline
\hline 
\hline 
AIC  & I  & 0.010  & 0.085  & \textbf{0.550}  & 0.240  & 0.115  & 0.000  & 0.000  & 0.685  & 0.165  & 0.150  & 0.000  & 0.000  & 0.840  & 0.095  & 0.065\tabularnewline
 & II  & 1.000  & 0.000  & 0.000  & 0.000  & 0.000  & 0.000  & 0.055  & 0.945  & 0.000  & 0.000  & 0.000  & 0.000  & \textbf{1.000}  & 0.000  & 0.000\tabularnewline
 & III  & 0.000  & 0.005  & 0.040  & 0.180  & 0.775  & 0.000  & 0.000  & 0.145  & 0.185  & 0.670  & 0.000  & 0.000  & 0.255  & 0.210  & 0.535\tabularnewline
 & IV  & 0.070  & 0.415  & 0.515  & 0.000  & 0.000  & 0.000  & 0.000  & 0.995  & 0.005  & 0.000  & 0.000  & 0.000  & \textbf{1.000}  & 0.000  & 0.000\tabularnewline
\hline 
BIC  & I  & 0.500  & 0.270  & 0.230  & 0.000  & 0.000  & 0.000  & 0.000  & \textbf{1.000}  & 0.000  & 0.000  & 0.000  & 0.000  & \textbf{1.000}  & 0.000  & 0.000\tabularnewline
 & II  & 1.000  & 0.000  & 0.000  & 0.000  & 0.000  & 1.000  & 0.000  & 0.000  & 0.000  & 0.000  & 0.000  & 0.000  & \textbf{1.000}  & 0.000  & 0.000\tabularnewline
 & III  & 0.000  & 0.035  & 0.380  & 0.355  & 0.230  & 0.000  & 0.000  & 0.755  & 0.155  & 0.090  & 0.000  & 0.000  & 0.965  & 0.030  & 0.005\tabularnewline
 & IV  & 0.545  & 0.440  & 0.015  & 0.000  & 0.000  & 0.000  & 0.000  & \textbf{1.000}  & 0.000  & 0.000  & 0.000  & 0.000  & \textbf{1.000}  & 0.000  & 0.000\tabularnewline
\hline 
$\epsilon$-BIC  & I  & 0.515  & 0.265  & 0.220  & 0.000  & 0.000  & 0.000  & 0.000  & \textbf{1.000}  & 0.000  & 0.000  & 0.000  & 0.000  & \textbf{1.000}  & 0.000  & 0.000\tabularnewline
$(\epsilon=0.02)$  & II  & 1.000  & 0.000  & 0.000  & 0.000  & 0.000  & 1.000  & 0.000  & 0.000  & 0.000  & 0.000  & 0.000  & 0.000  & \textbf{1.000}  & 0.000  & 0.000\tabularnewline
 & III  & 0.000  & 0.035  & 0.400  & 0.360  & 0.205  & 0.000  & 0.000  & 0.785  & 0.135  & 0.080  & 0.000  & 0.000  & 0.980  & 0.015  & 0.005\tabularnewline
 & IV  & 0.560  & 0.430  & 0.010  & 0.000  & 0.000  & 0.000  & 0.000  & \textbf{1.000}  & 0.000  & 0.000  & 0.000  & 0.000  & \textbf{1.000}  & 0.000  & 0.000\tabularnewline
\hline 
$\epsilon$-BIC  & I  & 0.620  & 0.265  & 0.115  & 0.000  & 0.000  & 0.000  & 0.000  & \textbf{1.000}  & 0.000  & 0.000  & 0.000  & 0.000  & \textbf{1.000}  & 0.000  & 0.000\tabularnewline
$(\epsilon=0.1)$  & II  & 1.000  & 0.000  & 0.000  & 0.000  & 0.000  & 1.000  & 0.000  & 0.000  & 0.000  & 0.000  & 0.000  & 0.000  & \textbf{1.000}  & 0.000  & 0.000\tabularnewline
 & III  & 0.005  & 0.065  & 0.450  & 0.325  & 0.155  & 0.000  & 0.000  & 0.875  & 0.090  & 0.035  & 0.000  & 0.000  & \textbf{1.000}  & 0.000  & 0.000\tabularnewline
 & IV  & 0.645  & 0.350  & 0.005  & 0.000  & 0.000  & 0.000  & 0.000  & \textbf{1.000}  & 0.000  & 0.000  & 0.000  & 0.000  & \textbf{1.000}  & 0.000  & 0.000\tabularnewline
\hline 
$\nu$-BIC  & I  & 1.000  & 0.000  & 0.000  & 0.000  & 0.000  & 1.000  & 0.000  & 0.000  & 0.000  & 0.000  & 0.000  & 0.000  & \textbf{1.000}  & 0.000  & 0.000\tabularnewline
$\left(\nu=1\right)$  & II  & 1.000  & 0.000  & 0.000  & 0.000  & 0.000  & 1.000  & 0.000  & 0.000  & 0.000  & 0.000  & 1.000  & 0.000  & 0.000  & 0.000  & 0.000\tabularnewline
 & III  & 0.805  & 0.180  & 0.015  & 0.000  & 0.000  & 0.000  & 0.000  & \textbf{1.000}  & 0.000  & 0.000  & 0.000  & 0.000  & \textbf{1.000}  & 0.000  & 0.000\tabularnewline
 & IV  & 1.000  & 0.000  & 0.000  & 0.000  & 0.000  & 1.000  & 0.000  & 0.000  & 0.000  & 0.000  & 0.000  & 0.000  & \textbf{1.000}  & 0.000  & 0.000\tabularnewline
\hline 
$\nu$-BIC  & I  & 0.500  & 0.270  & 0.230  & 0.000  & 0.000  & 0.000  & 0.000  & \textbf{1.000}  & 0.000  & 0.000  & 0.000  & 0.000  & \textbf{1.000}  & 0.000  & 0.000\tabularnewline
$\left(\nu=3\right)$  & II  & 1.000  & 0.000  & 0.000  & 0.000  & 0.000  & 1.000  & 0.000  & 0.000  & 0.000  & 0.000  & 0.000  & 0.000  & \textbf{1.000}  & 0.000  & 0.000\tabularnewline
 & III  & 0.000  & 0.035  & 0.380  & 0.355  & 0.230  & 0.000  & 0.000  & 0.755  & 0.155  & 0.090  & 0.000  & 0.000  & 0.965  & 0.030  & 0.005\tabularnewline
 & IV  & 0.545  & 0.440  & 0.015  & 0.000  & 0.000  & 0.000  & 0.000  & \textbf{1.000}  & 0.000  & 0.000  & 0.000  & 0.000  & \textbf{1.000}  & 0.000  & 0.000\tabularnewline
\hline 
PanIC  & I  & 0.995  & 0.005  & 0.000  & 0.000  & 0.000  & 0.990  & 0.005  & 0.005  & 0.000  & 0.000  & 0.000  & 0.000  & \textbf{1.000}  & 0.000  & 0.000\tabularnewline
 & II  & 1.000  & 0.000  & 0.000  & 0.000  & 0.000  & 1.000  & 0.000  & 0.000  & 0.000  & 0.000  & 1.000  & 0.000  & 0.000  & 0.000  & 0.000\tabularnewline
 & III  & 0.370  & 0.270  & 0.345  & 0.015  & 0.000  & 0.000  & 0.000  & \textbf{1.000}  & 0.000  & 0.000  & 0.000  & 0.000  & \textbf{1.000}  & 0.000  & 0.000\tabularnewline
 & IV  & 0.995  & 0.005  & 0.000  & 0.000  & 0.000  & 0.995  & 0.005  & 0.000  & 0.000  & 0.000  & 0.000  & 0.000  & \textbf{1.000}  & 0.000  & 0.000\tabularnewline
\hline 
\end{tabular}
\end{sidewaystable}

\begin{table}
\caption{Results from 200 replications of scenario S-T1. The $\hat{k}_{n}=\dots$
columns indicate the proportion of times $\hat{k}_{n}=k$, where $*$
marks the true number of components $k_{0}$. Bold text indicates
the highest proportion of correct identifications of $k_{0}$. }
\label{tab:S-T1}

\centering{}%
\begin{tabular}{|cc|ccc|ccc|ccc|}
\hline 
 &  & \multicolumn{3}{c|}{$n=100$} & \multicolumn{3}{c|}{$n=1000$} & \multicolumn{3}{c|}{$n=10000$}\tabularnewline
Pen. shape  & $\alpha\left(k\right)$  & $k=1$  & $2^{*}$  & $3$  & $1$  & $2^{*}$  & $3$  & $1$  & $2^{*}$  & $3$\tabularnewline
\hline 
\hline 
AIC  & I  & 0.015  & 0.575  & 0.410  & 0.000  & 0.800  & 0.200  & 0.000  & 0.890  & 0.110\tabularnewline
 & II  & 0.990  & 0.010  & 0.000  & 0.000  & \textbf{1.000}  & 0.000  & 0.000  & \textbf{1.000}  & 0.000\tabularnewline
 & III  & 0.000  & 0.175  & 0.825  & 0.000  & 0.450  & 0.550  & 0.000  & 0.505  & 0.495\tabularnewline
 & IV  & 0.020  & \textbf{0.875}  & 0.105  & 0.000  & 0.950  & 0.050  & 0.000  & 0.970  & 0.030\tabularnewline
\hline 
BIC  & I  & 0.240  & 0.755  & 0.005  & 0.000  & \textbf{1.000}  & 0.000  & 0.000  & \textbf{1.000}  & 0.000\tabularnewline
 & II  & 1.000  & 0.000  & 0.000  & 0.070  & 0.930  & 0.000  & 0.000  & \textbf{1.000}  & 0.000\tabularnewline
 & III  & 0.010  & 0.455  & 0.535  & 0.000  & 0.845  & 0.155  & 0.000  & 0.970  & 0.030\tabularnewline
 & IV  & 0.240  & 0.760  & 0.000  & 0.000  & \textbf{1.000}  & 0.000  & 0.000  & \textbf{1.000}  & 0.000\tabularnewline
\hline 
$\epsilon$-BIC  & I  & 0.265  & 0.735  & 0.000  & 0.000  & \textbf{1.000}  & 0.000  & 0.000  & \textbf{1.000}  & 0.000\tabularnewline
$(\epsilon=0.02)$  & II  & 1.000  & 0.000  & 0.000  & 0.105  & 0.895  & 0.000  & 0.000  & \textbf{1.000}  & 0.000\tabularnewline
 & III  & 0.010  & 0.465  & 0.525  & 0.000  & 0.860  & 0.140  & 0.000  & 0.970  & 0.030\tabularnewline
 & IV  & 0.265  & 0.735  & 0.000  & 0.000  & \textbf{1.000}  & 0.000  & 0.000  & \textbf{1.000}  & 0.000\tabularnewline
\hline 
$\epsilon$-BIC  & I  & 0.365  & 0.635  & 0.000  & 0.000  & \textbf{1.000}  & 0.000  & 0.000  & \textbf{1.000}  & 0.000\tabularnewline
$(\epsilon=0.1)$  & II  & 1.000  & 0.000  & 0.000  & 0.425  & 0.575  & 0.000  & 0.000  & \textbf{1.000}  & 0.000\tabularnewline
 & III  & 0.010  & 0.520  & 0.470  & 0.000  & 0.915  & 0.085  & 0.000  & 0.975  & 0.025\tabularnewline
 & IV  & 0.365  & 0.635  & 0.000  & 0.000  & \textbf{1.000}  & 0.000  & 0.000  & \textbf{1.000}  & 0.000\tabularnewline
\hline 
$\nu$-BIC  & I  & 1.000  & 0.000  & 0.000  & 0.040  & 0.960  & 0.000  & 0.000  & \textbf{1.000}  & 0.000\tabularnewline
$\left(\nu=1\right)$  & II  & 1.000  & 0.000  & 0.000  & 1.000  & 0.000  & 0.000  & 0.560  & 0.440  & 0.000\tabularnewline
 & III  & 0.635  & 0.365  & 0.000  & 0.000  & \textbf{1.000}  & 0.000  & 0.000  & \textbf{1.000}  & 0.000\tabularnewline
 & IV  & 1.000  & 0.000  & 0.000  & 0.040  & 0.960  & 0.000  & 0.000  & \textbf{1.000}  & 0.000\tabularnewline
\hline 
$\nu$-BIC  & I  & 0.240  & 0.755  & 0.005  & 0.000  & \textbf{1.000}  & 0.000  & 0.000  & \textbf{1.000}  & 0.000\tabularnewline
$\left(\nu=3\right)$  & II  & 1.000  & 0.000  & 0.000  & 0.070  & 0.930  & 0.000  & 0.000  & \textbf{1.000}  & 0.000\tabularnewline
 & III  & 0.010  & 0.455  & 0.535  & 0.000  & 0.845  & 0.155  & 0.000  & 0.970  & 0.030\tabularnewline
 & IV  & 0.240  & 0.760  & 0.000  & 0.000  & \textbf{1.000}  & 0.000  & 0.000  & \textbf{1.000}  & 0.000\tabularnewline
\hline 
PanIC  & I  & 0.970  & 0.030  & 0.000  & 0.005  & 0.995  & 0.000  & 0.000  & \textbf{1.000}  & 0.000\tabularnewline
 & II  & 1.000  & 0.000  & 0.000  & 1.000  & 0.000  & 0.000  & 1.000  & 0.000  & 0.000\tabularnewline
 & III  & 0.185  & 0.785  & 0.030  & 0.000  & \textbf{1.000}  & 0.000  & 0.000  & \textbf{1.000}  & 0.000\tabularnewline
 & IV  & 0.970  & 0.030  & 0.000  & 0.005  & 0.995  & 0.000  & 0.000  & \textbf{1.000}  & 0.000\tabularnewline
\hline 
\end{tabular}
\end{table}

\begin{sidewaystable}
\caption{Results from 200 replications of scenario S-T2. The $\hat{k}_{n}=\dots$
columns indicate the proportion of times $\hat{k}_{n}=k$, where $*$
marks the true number of components $k_{0}$. Bold text indicates
the highest proportion of correct identifications of $k_{0}$. }
\label{tab:S-T2}

\centering{}%
\begin{tabular}{|cc|ccccc|ccccc|ccccc|}
\hline 
 &  & \multicolumn{5}{c|}{$n=100$} & \multicolumn{5}{c|}{$n=1000$} & \multicolumn{5}{c|}{$n=10000$}\tabularnewline
Pen. shape  & $\alpha\left(k\right)$  & $k=1$  & $2$  & $3^{*}$  & $4$  & $5$  & $1$  & $2$  & $3^{*}$  & $4$  & $5$  & $1$  & $2$  & $3^{*}$  & $4$  & $5$\tabularnewline
\hline 
\hline 
AIC  & I  & 0.060  & 0.225  & \textbf{0.350}  & 0.240  & 0.125  & 0.000  & 0.000  & 0.690  & 0.195  & 0.115  & 0.000  & 0.000  & 0.895  & 0.065  & 0.040\tabularnewline
 & II  & 1.000  & 0.000  & 0.000  & 0.000  & 0.000  & 0.030  & 0.840  & 0.130  & 0.000  & 0.000  & 0.000  & 0.000  & \textbf{1.000}  & 0.000  & 0.000\tabularnewline
 & III  & 0.000  & 0.015  & 0.065  & 0.135  & 0.785  & 0.000  & 0.000  & 0.155  & 0.225  & 0.620  & 0.000  & 0.000  & 0.315  & 0.215  & 0.470\tabularnewline
 & IV  & 0.145  & 0.515  & 0.335  & 0.005  & 0.000  & 0.000  & 0.000  & 0.985  & 0.015  & 0.000  & 0.000  & 0.000  & 0.995  & 0.005  & 0.000\tabularnewline
\hline 
BIC  & I  & 0.690  & 0.235  & 0.070  & 0.005  & 0.000  & 0.000  & 0.000  & \textbf{1.000}  & 0.000  & 0.000  & 0.000  & 0.000  & \textbf{1.000}  & 0.000  & 0.000\tabularnewline
 & II  & 1.000  & 0.000  & 0.000  & 0.000  & 0.000  & 1.000  & 0.000  & 0.000  & 0.000  & 0.000  & 0.000  & 0.000  & \textbf{1.000}  & 0.000  & 0.000\tabularnewline
 & III  & 0.015  & 0.115  & 0.295  & 0.280  & 0.295  & 0.000  & 0.000  & 0.795  & 0.145  & 0.060  & 0.000  & 0.000  & 0.970  & 0.025  & 0.005\tabularnewline
 & IV  & 0.720  & 0.270  & 0.010  & 0.000  & 0.000  & 0.000  & 0.010  & 0.990  & 0.000  & 0.000  & 0.000  & 0.000  & \textbf{1.000}  & 0.000  & 0.000\tabularnewline
\hline 
$\epsilon$-BIC  & I  & 0.715  & 0.220  & 0.065  & 0.000  & 0.000  & 0.000  & 0.000  & \textbf{1.000}  & 0.000  & 0.000  & 0.000  & 0.000  & \textbf{1.000}  & 0.000  & 0.000\tabularnewline
$(\epsilon=0.02)$  & II  & 1.000  & 0.000  & 0.000  & 0.000  & 0.000  & 1.000  & 0.000  & 0.000  & 0.000  & 0.000  & 0.000  & 0.000  & \textbf{1.000}  & 0.000  & 0.000\tabularnewline
 & III  & 0.015  & 0.120  & 0.315  & 0.285  & 0.265  & 0.000  & 0.000  & 0.840  & 0.115  & 0.045  & 0.000  & 0.000  & 0.970  & 0.025  & 0.005\tabularnewline
 & IV  & 0.750  & 0.245  & 0.005  & 0.000  & 0.000  & 0.000  & 0.010  & 0.990  & 0.000  & 0.000  & 0.000  & 0.000  & \textbf{1.000}  & 0.000  & 0.000\tabularnewline
\hline 
$\epsilon$-BIC  & I  & 0.830  & 0.155  & 0.015  & 0.000  & 0.000  & 0.000  & 0.000  & \textbf{1.000}  & 0.000  & 0.000  & 0.000  & 0.000  & \textbf{1.000}  & 0.000  & 0.000\tabularnewline
$(\epsilon=0.1)$  & II  & 1.000  & 0.000  & 0.000  & 0.000  & 0.000  & 1.000  & 0.000  & 0.000  & 0.000  & 0.000  & 0.000  & 0.000  & \textbf{1.000}  & 0.000  & 0.000\tabularnewline
 & III  & 0.035  & 0.185  & 0.325  & 0.270  & 0.185  & 0.000  & 0.000  & 0.910  & 0.070  & 0.020  & 0.000  & 0.000  & 0.995  & 0.005  & 0.000\tabularnewline
 & IV  & 0.845  & 0.155  & 0.000  & 0.000  & 0.000  & 0.000  & 0.045  & 0.955  & 0.000  & 0.000  & 0.000  & 0.000  & \textbf{1.000}  & 0.000  & 0.000\tabularnewline
\hline 
$\nu$-BIC  & I  & 1.000  & 0.000  & 0.000  & 0.000  & 0.000  & 1.000  & 0.000  & 0.000  & 0.000  & 0.000  & 0.000  & 0.000  & \textbf{1.000}  & 0.000  & 0.000\tabularnewline
$\left(\nu=1\right)$  & II  & 1.000  & 0.000  & 0.000  & 0.000  & 0.000  & 1.000  & 0.000  & 0.000  & 0.000  & 0.000  & 1.000  & 0.000  & 0.000  & 0.000  & 0.000\tabularnewline
 & III  & 0.955  & 0.040  & 0.005  & 0.000  & 0.000  & 0.050  & 0.095  & 0.855  & 0.000  & 0.000  & 0.000  & 0.000  & \textbf{1.000}  & 0.000  & 0.000\tabularnewline
 & IV  & 1.000  & 0.000  & 0.000  & 0.000  & 0.000  & 1.000  & 0.000  & 0.000  & 0.000  & 0.000  & 0.000  & 0.000  & \textbf{1.000}  & 0.000  & 0.000\tabularnewline
\hline 
$\nu$-BIC  & I  & 0.690  & 0.235  & 0.070  & 0.005  & 0.000  & 0.000  & 0.000  & \textbf{1.000}  & 0.000  & 0.000  & 0.000  & 0.000  & \textbf{1.000}  & 0.000  & 0.000\tabularnewline
$\left(\nu=3\right)$  & II  & 1.000  & 0.000  & 0.000  & 0.000  & 0.000  & 1.000  & 0.000  & 0.000  & 0.000  & 0.000  & 0.000  & 0.000  & \textbf{1.000}  & 0.000  & 0.000\tabularnewline
 & III  & 0.015  & 0.115  & 0.295  & 0.280  & 0.295  & 0.000  & 0.000  & 0.795  & 0.145  & 0.060  & 0.000  & 0.000  & 0.970  & 0.025  & 0.005\tabularnewline
 & IV  & 0.720  & 0.270  & 0.010  & 0.000  & 0.000  & 0.000  & 0.010  & 0.990  & 0.000  & 0.000  & 0.000  & 0.000  & \textbf{1.000}  & 0.000  & 0.000\tabularnewline
\hline 
PanIC  & I  & 1.000  & 0.000  & 0.000  & 0.000  & 0.000  & 1.000  & 0.000  & 0.000  & 0.000  & 0.000  & 0.000  & 0.000  & \textbf{1.000}  & 0.000  & 0.000\tabularnewline
 & II  & 1.000  & 0.000  & 0.000  & 0.000  & 0.000  & 1.000  & 0.000  & 0.000  & 0.000  & 0.000  & 1.000  & 0.000  & 0.000  & 0.000  & 0.000\tabularnewline
 & III  & 0.595  & 0.280  & 0.120  & 0.005  & 0.000  & 0.025  & 0.060  & 0.915  & 0.000  & 0.000  & 0.000  & 0.000  & \textbf{1.000}  & 0.000  & 0.000\tabularnewline
 & IV  & 1.000  & 0.000  & 0.000  & 0.000  & 0.000  & 1.000  & 0.000  & 0.000  & 0.000  & 0.000  & 0.000  & 0.995  & 0.005  & 0.000  & 0.000\tabularnewline
\hline 
\end{tabular}
\end{sidewaystable}

From Tables \ref{tab:S-L1}--\ref{tab:S-T2}, we observe, as expected,
that as $n$ increases, the proportions of times that the $\epsilon$-BIC,
$\nu$-BIC, and PanIC estimate the correct number of components increase,
and typically become perfectly accurate, except in cases where the
penalty and shape combinations are too severe and thus consistency
is not realised in practice even at $n=10000$, such as for PanIC
with Type II penalties or the $\nu=1$ case of the $\nu$-BIC with
Type II penalties. We observe in our experiments that the Type III
penalties when applied with the $\epsilon$-BIC or $\nu$-BIC shapes
tend to be too small to realise perfect accuracy at $n=10000$, although
it is clear that the proportion of times $\hat{k}_{n}=k_{0}$ is increasing
in each case. In terms of good choices for $\alpha$, we observe that
Type I and Type IV forms appear to work well with all penalty shapes,
with Type II often being too large and Type III being too small.

Perhaps the most surprising outcome is that for some configurations,
the AIC and the BIC can produce perfect or near perfect estimation
of $k_{0}$ for large $n$. Indeed, in the case of the AIC, as we
shall show in Proposition \ref{prop:Delta_assumption_necessary} (albeit
in a simplified setting that does not include the current experimental
scenarios), it is impossible to guarantee that the proportion of accurate
estimates of $k_{0}$ goes to one. This points to either this being
a small sample phenomenon and thus the true proportion of overestimating
$k_{0}$ is possibly small but non-zero, or that there exist settings
where the AIC is consistent, although we believe that the former is
more likely than the latter given our results in Section \ref{subsec:Inefficiency-of-consistent}.
In the case of the BIC, we argue that it is unknown whether the conditions
of \citet{keribin2000consistent} are necessary and sufficient, and
that there is potentially a gap between our theory and the theory
of \citet{keribin2000consistent}. Furthermore, if we observe the
results of the $\epsilon$-BIC with $\epsilon=0.02$ and the $\nu$-BIC
with $\nu=3$, we note that the outcomes are either exactly the same
as or very close to those of the BIC for all $n$, due to the fact
that, numerically, the criteria are only infinitesimally different
for any practical size $n$. We, however, conclude that there is indeed
evidence that the infinitesimal variations from the BIC, via the $\epsilon$-BIC
and $\nu$-BIC constructions, are a harmless and practically costless
way to guarantee consistency in settings where such a guarantee may
not otherwise be afforded by the results of \citet{keribin2000consistent}.
Overall, these simulations lead to our recommendation, at least in
these simulation scenarios, that the $\epsilon$-BIC with $\epsilon=0.02$
and the $\nu$-BIC with $\nu=3$, together with $\alpha$ of Type
I or Type IV, provide good finite sample performance while enabling
the guarantees of Corollary \ref{cor:Main=00003D00003D000020Result}.

\section{Limitations}

\label{sec:Limitations}

\subsection{Penalty flexibility}

Theorem \ref{thm:GenericConsistency} provides a very general and
flexible approach to order selection for mixture models, and as pointed
out in Remarks \ref{rem:Choice_of_rate} and \ref{rem:Choice_of_alpha},
there are many possible choices for penalty calibration beyond our
proposed $\nu$-BIC and $\epsilon$-BIC proposals. A natural question
to ask is whether there are optimal choices for the penalty. In particular,
from a minimax perspective (cf. \citealp{gassiat2012consistent} and
\citealp[Ch. 4]{gassiat2018universal}), we would like to determine
if there is a smallest penalty that satisfies B1 and B2. Such a penalty
would provide model selection consistency while minimally modifying
the maximum likelihood as the axis for model choice.

Unfortunately, while such a minimal penalty is desirable, we can prove
that it does not exist. In particular, for fixed $\bar{k}\ge2$, let
$\mathscr{P}$ denote the collection of penalty arrays $\mathrm{pen}=\left\{ \mathrm{pen}_{k,n}:k\in\left[\bar{k}\right],n\in\mathbb{N}\right\} $
that satisfy B1 and B2. We endow $\mathscr{P}$ with the domination
order: $\mathrm{pen}^{\left(1\right)}\preceq\mathrm{pen}^{\left(2\right)}$
if and only if there exists an $n_{0}\in\mathbb{N}$ such that for
all $n\ge n_{0}$ and for all $k\in\left[\bar{k}\right]$, $\mathrm{pen}_{k,n}^{\left(1\right)}\le\mathrm{pen}_{k,n}^{\left(2\right)}$.
Call $\mathrm{pen}^{*}\in\mathscr{P}$ universally optimal if it is
the minimal element of $\left(\mathscr{P},\preceq\right)$; i.e.,
$\mathrm{pen}^{*}\preceq\mathrm{pen}$, for each $\mathrm{pen}\in\mathscr{P}$. 
\begin{prop}
\label{prop:no_min_pen}The partial order $\left(\mathscr{P},\preceq\right)$
has no minimal element. 
\end{prop}

More precisely, for every $\mathrm{pen}\in\mathscr{P}$, there exists
a $\widetilde{\mathrm{pen}}\in\mathscr{P}$ such that $\widetilde{\mathrm{pen}}_{k,n}<\mathrm{pen}_{k,n}$
for all $k\in\left[\bar{k}\right]$, and all sufficiently large $n$.
In particular, this shows that there is no optimal choice of penalty
rate and $\alpha$ function among the class of IC penalties characterised
by B1 and B2 alone, and any penalty that satisfies B1 and B2 admits
an asymptotically smaller penalty that also satisfies B1 and B2.

We may hope that if we have some fixed $\alpha:\mathbb{N}\to\mathbb{R}_{>0}$
strictly increasing, and we define our penalty in the form $\mathrm{pen}_{k,n}^{\alpha}=\alpha\left(k\right)r_{n}$,
where $\left(r_{n}\right)_{n}\subset\mathbb{R}_{>0}$, then there
may be some minimal choice of $\left(r_{n}\right)_{n}$. Unfortunately,
this also is not true as per the following result. 
\begin{prop}
For penalties of the form $\mathrm{pen}_{k,n}^{\alpha}=\alpha\left(k\right)r_{n}$,
if B1 and B2 hold then necessarily $r_{n}\to0$ and $\left(n/\log n\right)r_{n}\to\infty$,
as $n\to\infty$. Moreover, for each such $\left(r_{n}\right)_{n}$,
there exists another sequence $\left(\tilde{r}_{n}\right)_{n}$ such
that $\tilde{r}_{n}=o\left(r_{n}\right)$, while still satisfying
$\tilde{r}_{n}\to0$ and $\left(n/\log n\right)\tilde{r}_{n}\to\infty$,
as $n\to\infty$. 
\end{prop}

\begin{proof}
Let $a_{n}=\left(n/\log n\right)r_{n}\to\infty$ and set $\tilde{r}_{n}=r_{n}/\sqrt{a_{n}}$.
Then $\tilde{r}_{n}/r_{n}=a_{n}^{-1/2}\to0$, hence $\tilde{r}_{n}=o\left(r_{n}\right)$
and $\tilde{r}_{n}\to0$. Lastly, $\left(n/\log n\right)\tilde{r}_{n}=\sqrt{a_{n}}\to\infty$,
satisfying B2, since $\alpha\left(l\right)-\alpha\left(k\right)>0$
for each $l>k$, as required. 
\end{proof}
Indeed, the same penalty flexibility issue is not only a feature of
our consistency theorems but also of the PanIC approach of \citet{nguyen2024panic}
and \citet{westerhout2024asymptotic}, the preceding results of \citet{sin1996information},
and even the mixture model BIC consistency results of \citet{keribin2000consistent},
\citet{gassiat2012consistent}, and \citet{gassiat2018universal}.
This is also a problem in the finite sample model selection approaches
of \citet{birge2007minimal} and \citet{Massart2007Concentration}
(implemented in the mixture model literature in the works of \citealp{maugis2011non},
\citealp{MaugisRabusseauMichel2013Adaptive}, and \citealp{nguyen2022non}).
In the finite sample setting, some methods are available for empirical
calibration of penalties, such as via the so-called slope heuristic
(cf. \citealp{baudry2012slope} and \citealp{arlot2019minimal}).
Although there is some evidence that these methods work well in empirical
studies, they provide no technical or concrete guarantees in the mixture
model setting.

\subsection{Inefficiency of consistent model selection}

\label{subsec:Inefficiency-of-consistent}

In \citet{Yang2005aicbic}, it was demonstrated that, in the regression
context, there is a fundamental tension between the AIC and consistent
estimators such as the BIC. In particular, it was shown that in the
context of choosing between nested regression models, the chosen model
selected via the AIC is minimax optimal in the mean squared error
sense, while it is impossible for any fitted model chosen via a model
consistent estimator to achieve the same minimax optimality. In this
subsection, we will prove that the same phenomenon arises in the mixture
model order selection setting.

To proceed, we will consider a simplified version of the example in
Section \ref{subsec:Gaussian-mixture-models}. In particular, let
\[
x\mapsto\phi\left(x;\mu\right)=\frac{1}{\sqrt{2\pi}}\exp\left\{ -\frac{1}{2}\left(x-\mu\right)^{2}\right\} 
\]
be the density of the univariate normal law with unit variance $\mathrm{N}\left(\mu,1\right)$.
For fixed $b>0$ and $\underline{\pi}\in\left(0,1/2\right)$, let
${\cal F}={\cal M}_{1}\cup{\cal M}_{2}$, where ${\cal M}_{1}=\left\{ \phi\left(\cdot;\mu\right):\mu\in\left[-b,b\right]\right\} $,
and 
\[
{\cal M}_{2}=\left\{ \pi\phi\left(\cdot;\mu_{1}\right)+\left(1-\pi\right)\phi\left(\cdot;\mu_{2}\right):\pi\in\left[\underline{\pi},1-\underline{\pi}\right],\mu_{1},\mu_{2}\in\left[-b,b\right],\mu_{1}\le\mu_{2}\right\} .
\]

Let $\mathbf{X}_{n}$ consist of IID replicates from the law with
density (with respect to the Lebesgue measure) $f_{0}\in{\cal F}$.
As in Section \ref{subsec:Mixture-of-regression}, we let $\hat{f}_{k,n}\in\arg\max_{f\in{\cal M}_{k}}P_{n}\log f$
for $k\in\left\{ 1,2\right\} $, and with $f\mapsto\ell_{n}\left(f\right)=-n^{-1}\sum_{i=1}^{n}\log f\left(X_{i}\right)$,
we write 
\[
\hat{k}_{n}^{\mathrm{AIC}}=\min\underset{k\in\left\{ 1,2\right\} }{\arg\min}\left\{ \ell_{n}\left(\hat{f}_{k,n}\right)+\mathrm{pen}_{k,n}^{\mathrm{AIC}}\right\} ,\hat{f}_{n}^{\mathrm{AIC}}=\hat{f}_{\hat{k}_{n}^{\mathrm{AIC}},n}.
\]
Here, we define AIC-like penalties via $\mathrm{pen}_{k,n}^{\mathrm{AIC}}=a\mathrm{dim}\left(\mathbb{S}_{k}\right)n^{-1}$,
where $a>0$ is sufficiently large, and $\mathrm{dim}\left(\mathbb{S}_{1}\right)=1$
and $\mathrm{dim}\left(\mathbb{S}_{2}\right)=3$, in this case. As
per the usual treatment on the topic (see, e.g., \citealp[Ch. 15]{Wainwright2019hds}),
we define the squared-Hellinger minimax risk as 
\[
\mathscr{R}_{n}^{*}=\inf_{\tilde{f}_{n}}\sup_{f_{0}\in{\cal F}}\mathrm{E}_{f_{0}}\left\{ \mathfrak{h}^{2}\left(\tilde{f}_{n},f_{0}\right)\right\} ,
\]
where $\mathrm{E}_{f_{0}}$ denotes the expectation over $\mathbf{X}_{n}$
with respect to the law with density $f_{0}$. The infimum is taken
over the class of all measurable estimators $\tilde{f}_{n}$. The
following result states that the AIC-like estimator $\hat{f}_{n}^{\mathrm{AIC}}$
achieves the minimax risk. 
\begin{prop}
\label{prop:AIC-minimax}For each sufficiently large $a>0$, there
exist constants $0<c<C<\infty$ depending on $a$, $b$, and $\underline{\pi}$
such that, for all sufficiently large $n$, 
\[
\frac{c}{n}\le\mathscr{R}_{n}^{*}\le\sup_{f_{0}\in{\cal F}}\mathrm{E}_{f_{0}}\left\{ \mathfrak{h}^{2}\left(\hat{f}_{n}^{\mathrm{AIC}},f_{0}\right)\right\} \le\frac{C}{n}.
\]
\end{prop}

In contrast to the AIC-like penalties, we note that the BIC or indeed
any penalty that satisfies our Assumption B2 or Assumption (C1) of
\citet{keribin2000consistent} or (N-B2) of \citet{nguyen2024panic}
will verify the condition 
\begin{equation}
\Delta_{n}\downarrow0,\ n\Delta_{n}\underset{n\to\infty}{\longrightarrow}\infty,\label{eq:Delta_assumption}
\end{equation}
where $\Delta_{n}=\mathrm{pen}_{2,n}-\mathrm{pen}_{1,n}$. With $\tilde{f}_{n}=\hat{f}_{\hat{k}_{n},n}$,
where $\hat{k}_{n}=\arg\min_{k\in\left\{ 1,2\right\} }\left\{ \ell_{n}\left(\hat{f}_{k,n}\right)+\mathrm{pen}_{k,n}\right\} $,
we have the following result. 
\begin{prop}
\label{prop:consistent_not_minimax} If \eqref{eq:Delta_assumption}
holds then there exists a constant $c>0$ depending only on $b$ and
$\underline{\pi}$, such that for all sufficiently large $n$, 
\[
\sup_{f_{0}\in{\cal F}}\mathrm{E}_{f_{0}}\left\{ \mathfrak{h}^{2}\left(\tilde{f}_{n},f_{0}\right)\right\} \ge c\Delta_{n}.
\]
\end{prop}

Proposition \ref{prop:consistent_not_minimax} implies that any selector
$\hat{k}_{n}$ that satisfies \eqref{eq:Delta_assumption} cannot
achieve the minimax Hellinger risk over ${\cal F}$. Indeed, from
Proposition \ref{prop:AIC-minimax}, we have that $\mathscr{R}_{n}^{*}\asymp n^{-1}$,
which, combined with Proposition \ref{prop:consistent_not_minimax},
yields 
\[
\frac{\sup_{f_{0}\in{\cal F}}\mathrm{E}_{f_{0}}\left\{ \mathfrak{h}^{2}\left(\tilde{f}_{n},f_{0}\right)\right\} }{\mathscr{R}_{n}^{*}}\gtrsim n\Delta_{n}\to\infty.
\]
Thus, the worst-case risk of $\tilde{f}_{n}$ is asymptotically larger
than the minimax risk by an unbounded factor. A natural follow-up
question is whether it is possible to construct a consistent selector
that does not verify condition \eqref{eq:Delta_assumption}. The following
result demonstrates that this is impossible in the current setting.
\begin{prop}
\label{prop:Delta_assumption_necessary} Suppose that $\hat{k}_{n}$
is order-consistent over ${\cal F}={\cal M}_{1}\cup{\cal M}_{2}$
in the sense that 
\[
\sup_{f_{0}\in{\cal M}_{1}}\mathrm{P}_{f_{0}}(\hat{k}_{n}=2)\to0\qquad\mathrm{and}\qquad\sup_{f_{0}\in{\cal M}_{2}\setminus{\cal M}_{1}}\mathrm{P}_{f_{0}}(\hat{k}_{n}=1)\to0.
\]
Then \eqref{eq:Delta_assumption} necessarily holds. 
\end{prop}

\section{Discussion}

\label{sec:Discussion}

The IC approach to model selection and order selection is ubiquitous
across many domains of statistical modeling, with particular popularity
and applicability in the context of finite mixture models. Among the
IC approaches, the BIC is often the most commonly used, and in the
mixture model setting where the maximum model size $\bar{k}$ is known,
its consistency has been proved in \citet{keribin2000consistent}.
In this work, we dramatically relax the assumptions made by \citet{keribin2000consistent}
via minor modifications of the BIC, to produce the $\nu$-BIC and
$\epsilon$-BIC. These ICs minimally modify the penalty term of the
BIC but allow for much broader applicability, notably being applicable
without any differentiability and requiring substantially weaker regularity
than high-order differentiability and moment conditions. Via example
applications, we show that our novel ICs are consistent estimators
in the popular setting of Gaussian mixture models, and also in settings
that fall outside the scope of \citet{keribin2000consistent}, including
non-differentiable Laplace mixture models, heavy-tailed $t$-mixture
models, and mixtures of regression models. In addition, we provide
a complementary result under misspecification: when $f_{0}$ need
not belong to any ${\cal M}_{k}^{\phi}$, any IC satisfying B1 will
eventually select an order whose best-fitting mixture model is Kullback--Leibler
optimal among the candidate orders in $[\bar{k}]$ (Proposition~\ref{prop:Misspec_convergence}).

As we have discussed in Section \ref{sec:Introduction}, for suitable
and reasonable choices of $\nu$ and $\epsilon$, the respective $\nu$-BIC
and $\epsilon$-BIC are numerically indistinguishable from the BIC
in practical settings. We accompany this observation with numerical
simulations in Section~\ref{sec:Example-applications}, which compare
the $\epsilon$-BIC and $\nu$-BIC against alternatives such as the
AIC, BIC, and PanIC in scenarios that exhibit the pathologies highlighted
in our technical examples. In these experiments, we observe that choices
such as $\epsilon=0.02$ and $\nu=3$ are, for all practical values
of $n$, effectively the same as the BIC numerically, while still
inheriting the theoretical guarantees of Corollary~\ref{cor:Main=00003D00003D000020Result}
in settings where the hypotheses of \citet{keribin2000consistent}
are not available. More generally, the simulation study provides some
empirical guidance on the practical calibration of $\alpha$ and on
the behaviour of competing criteria in difficult regimes.

Given the potentially small numerical difference between the $\nu$-BIC
and $\epsilon$-BIC and the BIC, we do not seek to push the use of
either criterion over the BIC, which many practitioners are already
accustomed to and use regularly. However, we wish to position our
results as technical tools for explaining the correctness and effectiveness
of the BIC, even when the stringent assumptions of \citet{keribin2000consistent}
are not met. For example, one may use our theory to justify the use
of the BIC for mixtures with non-differentiable components $\phi$,
such as triangular mixtures \citep{karlis2008polygonal,nguyen2016maximum}
and Laplace-based mixture models \citep{franczak2013mixtures,song2014robust,azam2020multivariate},
as well as mixtures with components $\phi$ with limited regularity,
such as $t$-distribution based models \citep{peel2000robust,lin2007robust,lo2012flexible,ForbesWraith2014,yao2014robust}.
In the misspecified case, Proposition~\ref{prop:Misspec_convergence}
further indicates that, under mild conditions, common ICs (including
the AIC and the BIC) still target Kullback--Leibler optimal approximating
models, albeit without a guarantee of parsimonious order selection.

Indeed, via Theorem \ref{thm:GenericConsistency}, it is possible
to construct multitudes of ICs that exhibit consistency in mixture
model settings, but the finite sample performance of any such penalty
cannot be deduced from our general consistency theory alone. In particular,
as formalised in Section~\ref{subsec:Inefficiency-of-consistent},
there are two structural limitations that consistency results do not
resolve. First, under the natural domination order, there is no universally
minimal penalty array satisfying B1 and B2, so consistency does not
single out an optimal BIC-scale penalty calibration. Second, we establish
a tension between order consistency and minimax optimality in mixtures
like that of \citet{Yang2005aicbic}: in a simple Gaussian mixture
family, an AIC-like criterion achieves the parametric minimax Hellinger
risk, whereas any order-consistent criterion must be minimax-inefficient
by an unbounded factor, typically logarithmic for BIC-scale penalties.
To the best of our knowledge, this AIC/BIC tension has not previously
been stated explicitly in the finite-mixture order selection setting.
In the finite sample setting, some methods are available for empirical
calibration of penalties, such as via the so-called slope heuristic
(cf. \citealp{baudry2012slope} and \citealp{arlot2019minimal}),
although such methods provide no general guarantees in the mixture
model setting.

Although the assumptions of Theorem \ref{thm:GenericConsistency}
are weaker than those of \citet{keribin2000consistent}, it is still
possible to deduce consistent model selection under further relaxations
of conditions. For example, even without continuity of the component
density $\phi$ or independence of the data $\mathbf{X}_{n}$, or
when the loss function is not necessarily the negative log-density,
we can deduce consistency using the theory of \citet{westerhout2024asymptotic},
at the expense of increasing the rate of the penalisation term to
$\tilde{O}\left(n^{-1/2}\right)$. An immediate direction of study
is to understand whether further relaxations of the assumptions are
possible without such a pronounced increase in the penalisation rate.
Another natural direction, motivated by truth-changing and locally
misspecified frameworks in modern model selection (for example the
focused information criterion literature of \citealp{claeskensHjort2003FIC,lohmeyerEtAl2019FICVAR,pandhareRamanathan2020FICGLMTS,pandhareRamanathan2020RobustFIC}
and related high-dimensional settings such as \citealp{gaoCarroll2017DataIntegration}),
is to extend our results to triangular-array regimes where the effective
data-generating mechanism, or the KL-optimal order, may vary with
$n$. We expect such extensions to be possible, but they would require
triangular-array analogues of the empirical process arguments underpinning
our uniform convergence and overfitting rate results. In addition,
under misspecification, Proposition~\ref{prop:Misspec_convergence}
only ensures that $\hat{k}_{n}$ eventually lies in the set of KL-minimising
orders, and we are not aware of general results proving parsimonious
order selection under misspecification for BIC-scale penalties. It
would therefore be of interest to understand whether parsimonious
selection is achievable under misspecification without moving to larger
penalties, or whether such a limitation is inherent.

Another immediate extension to our current work is to obtain the corresponding
regularity conditions for consistent order selection of mixture of
experts models (cf. \citealp{jacobs1991adaptive}, \citealp{yuksel2012twenty},
and \citealp{nguyen2018practical}), which extend upon the finite
mixture models by also parameterising the mixing coefficients $\pi_{1},\dots,\pi_{k}$.
Such an extension would add to the growing body of theoretical results
regarding model selection and model uncertainty in mixture of experts
models, including the results of \citet{khalili2010new}, \citet{nguyen2022non},
\citet{nguyen2023non}, \citet{nguyen2023demystifying}, \citet{nguyen_general_2024},
\citet{nguyen_bayesian_2024}, \citet{nguyen_towards_2024}, \citet{thai_model_2025},
and \citet{westerhout2024asymptotic}, for example. Lastly, a limitation
of our current results is the requirement that some maximum order
$\bar{k}$ be specified. Indeed such an assumption can be omitted
at the great expense of strong control of local bracketing entropy
of normalisations of the mixture model classes, as in \citet{gassiat2012consistent}
and \citet[Sec. 4.3]{gassiat2018universal}. It remains open as to
whether weaker assumptions are available, even for penalties with
the larger rate $\tilde{O}\left(n^{-1/2}\right)$. 

\section*{Appendix A}

\subsection*{Proof of Proposition \ref{prop:MainTechProp}}

Let $P$ denote the probability measure with density $f_{0}$ with
respect to $\mathfrak{m}$, and let $P_{n}=n^{-1}\sum_{i=1}^{n}\delta_{X_{i}}$.
Define the measure $\mathfrak{m}_{0}$ by 
\[
\mathfrak{m}_{0}\left(\mathrm{d}x\right)=\mathbf{1}\left\{ f_{0}\left(x\right)>0\right\} \mathfrak{m}\left(\mathrm{d}x\right).
\]
Following \citet{van-de-Geer:2000aa}, write $\bar{f}=\left(f+f_{0}\right)/2$
and 
\[
\bar{{\cal F}}=\left\{ \bar{f}:f\in{\cal F}\right\} ,\qquad\bar{{\cal F}}^{1/2}=\left\{ \bar{f}^{1/2}:\bar{f}\in\bar{{\cal F}}\right\} .
\]
Let $\hat{f}_{n}:\Omega\to{\cal F}$ denote a maximum likelihood estimator:
\[
\hat{f}_{n}\in\arg\max_{f\in{\cal F}}P_{n}\log f.
\]

For any $0<\delta\le1$, define the local class 
\[
\bar{{\cal F}}^{1/2}\left(\delta\right)=\left\{ \bar{f}^{1/2}:\bar{f}=\left(f+f_{0}\right)/2,\ f\in{\cal F},\ \mathfrak{h}\left(f,f_{0}\right)\le\delta\right\} ,
\]
and the associated local bracketing entropy integral in ${\cal L}_{2}\left(\mathfrak{m}_{0}\right)$:
\begin{equation}
J_{\left[\right]}\left(\delta,\bar{{\cal F}}^{1/2}\left(\delta\right),{\cal L}_{2}\left(\mathfrak{m}_{0}\right)\right)=\delta\vee\int_{0}^{\delta}\sqrt{H_{\left[\right]}\left(u,\bar{{\cal F}}^{1/2}\left(\delta\right),{\cal L}_{2}\left(\mathfrak{m}_{0}\right)\right)}\mathrm{d}u.\label{eq:LocalEntropy_mu0}
\end{equation}

By \citet[Cor.~7.5]{van-de-Geer:2000aa}, if there exists a function
$\Psi:\mathbb{R}_{>0}\to\mathbb{R}_{>0}$ such that 
\[
\Psi\left(\delta\right)\ge J_{\left[\right]}\left(\delta,\bar{{\cal F}}^{1/2}\left(\delta\right),{\cal L}_{2}\left(\mathfrak{m}_{0}\right)\right),
\]
where $\Psi\left(\delta\right)/\delta^{2}$ is non-increasing in $\delta$,
then there exists $c_{1}>0$ such that 
\[
\mathrm{P}\left(\int_{\mathbb{X}}\log\frac{\hat{f}_{n}}{f_{0}}\mathrm{d}P_{n}\ge\delta^{2}\right)\le c_{1}\exp\left\{ -\frac{n\delta^{2}}{c_{1}^{2}}\right\} ,
\]
for every $\delta\ge\delta_{n}$, where $\delta_{n}$ satisfies $\sqrt{n}\delta_{n}^{2}\ge c\Psi\left(\delta_{n}\right)$
for a universal constant $c>0$.

It therefore suffices to show that our hypothesis $\Psi\left(\delta\right)\ge J\left(\delta\right)$,
where $J$ is defined in \eqref{eq:entropyintegral}, implies 
\[
\Psi\left(\delta\right)\ge J_{\left[\right]}\left(\delta,\bar{{\cal F}}^{1/2}\left(\delta\right),{\cal L}_{2}\left(\mathfrak{m}_{0}\right)\right).
\]

We first relate the bracketing numbers of $\bar{{\cal F}}^{1/2}$
in ${\cal L}_{2}\left(\mathfrak{m}_{0}\right)$ to those of ${\cal F}$
in ${\cal L}_{1}\left(\mathfrak{m}\right)$. 
\begin{lem}
If ${\cal F}$ is a space of probability densities on $\mathbb{X}$,
then for each $\delta>0$, 
\[
N_{\left[\right]}\left(\delta,\bar{{\cal F}}^{1/2},{\cal L}_{2}\left(\mathfrak{m}_{0}\right)\right)\le N_{\left[\right]}\left(2\delta^{2},{\cal F},{\cal L}_{1}\left(\mathfrak{m}\right)\right).
\]
\end{lem}

\begin{proof}
Let $\delta>0$ and let $\left(\left[f_{j}^{\mathrm{L}},f_{j}^{\mathrm{U}}\right]\right)_{j\in\left[N\right]}$
be a $2\delta^{2}$-bracketing of ${\cal F}$ with respect to the
${\cal L}_{1}\left(\mathfrak{m}\right)$ norm. For each $j$, define
$\bar{f}_{j}^{\mathrm{L}}=\left(f_{j}^{\mathrm{L}}+f_{0}\right)/2$
and $\bar{f}_{j}^{\mathrm{U}}=\left(f_{j}^{\mathrm{U}}+f_{0}\right)/2$.
Then for $\bar{f}=\left(f+f_{0}\right)/2$, we have $\bar{f}_{j}^{\mathrm{L}}\le\bar{f}\le\bar{f}_{j}^{\mathrm{U}}$
whenever $f_{j}^{\mathrm{L}}\le f\le f_{j}^{\mathrm{U}}$.

Moreover, for $a,b\ge0$, $\left(\sqrt{a}-\sqrt{b}\right)^{2}\le\left|a-b\right|$.
Hence, 
\begin{align*}
\left\Vert \left(\bar{f}_{j}^{\mathrm{U}}\right)^{1/2}-\left(\bar{f}_{j}^{\mathrm{L}}\right)^{1/2}\right\Vert _{{\cal L}_{2}\left(\mathfrak{m}_{0}\right)}^{2} & =\int_{\mathbb{X}}\left\{ \left(\bar{f}_{j}^{\mathrm{U}}\right)^{1/2}-\left(\bar{f}_{j}^{\mathrm{L}}\right)^{1/2}\right\} ^{2}\mathrm{d}\mathfrak{m}_{0}\\
 & \le\int_{\mathbb{X}}\left|\bar{f}_{j}^{\mathrm{U}}-\bar{f}_{j}^{\mathrm{L}}\right|\mathrm{d}\mathfrak{m}_{0}\le\int_{\mathbb{X}}\left|\bar{f}_{j}^{\mathrm{U}}-\bar{f}_{j}^{\mathrm{L}}\right|\mathrm{d}\mathfrak{m}\\
 & =\frac{1}{2}\int_{\mathbb{X}}\left|f_{j}^{\mathrm{U}}-f_{j}^{\mathrm{L}}\right|\mathrm{d}\mathfrak{m}\le\delta^{2}.
\end{align*}
Thus, $\left[\left(\bar{f}_{j}^{\mathrm{L}}\right)^{1/2},\left(\bar{f}_{j}^{\mathrm{U}}\right)^{1/2}\right]$
is a $\delta$-bracket for $\bar{{\cal F}}^{1/2}$ in ${\cal L}_{2}\left(\mathfrak{m}_{0}\right)$,
and the claim follows. 
\end{proof}
By the preceding lemma, for each $u>0$ and $0<\delta\le1$, 
\[
H_{\left[\right]}\left(u,\bar{{\cal F}}^{1/2}\left(\delta\right),{\cal L}_{2}\left(\mathfrak{m}_{0}\right)\right)\le H_{\left[\right]}\left(2u^{2},{\cal F}\left(\delta\right),{\cal L}_{1}\left(\mathfrak{m}\right)\right).
\]
Since the bracketing entropy is non-increasing in its first argument,
we have 
\[
H_{\left[\right]}\left(2u^{2},{\cal F}\left(\delta\right),{\cal L}_{1}\left(\mathfrak{m}\right)\right)\le H_{\left[\right]}\left(u^{2},{\cal F}\left(\delta\right),{\cal L}_{1}\left(\mathfrak{m}\right)\right).
\]
Therefore, 
\begin{align*}
J_{\left[\right]}\left(\delta,\bar{{\cal F}}^{1/2}\left(\delta\right),{\cal L}_{2}\left(\mathfrak{m}_{0}\right)\right) & =\delta\vee\int_{0}^{\delta}\sqrt{H_{\left[\right]}\left(u,\bar{{\cal F}}^{1/2}\left(\delta\right),{\cal L}_{2}\left(\mathfrak{m}_{0}\right)\right)}\mathrm{d}u\\
 & \le\delta\vee\int_{0}^{\delta}\sqrt{H_{\left[\right]}\left(u^{2},{\cal F}\left(\delta\right),{\cal L}_{1}\left(\mathfrak{m}\right)\right)}\mathrm{d}u=J\left(\delta\right).
\end{align*}
Thus, if $\Psi\left(\delta\right)\ge J\left(\delta\right)$, then
also $\Psi\left(\delta\right)\ge J_{\left[\right]}\left(\delta,\bar{{\cal F}}^{1/2}\left(\delta\right),{\cal L}_{2}\left(\mathfrak{m}_{0}\right)\right)$,
and Proposition \ref{prop:MainTechProp} follows from \citet[Cor.~7.5]{van-de-Geer:2000aa}. 

\subsection*{Proof of Lemma \ref{lem:Convergenceofmin}}

For each $k\in\left[\bar{k}\right]$, let ${\cal N}_{k}^{\phi}=\left\{ -\log f:f\in{\cal M}_{k}^{\phi}\right\} $.
We firstly seek to show that ${\cal N}_{k}^{\phi}$ are GC classes.
Indeed, if we write $\mathbb{S}_{k}=\left\{ \left(\pi_{1},\dots,\pi_{k}\right)\in\left[0,1\right]^{k}:\sum_{z=1}^{k}\pi_{z}=1\right\} \times\mathbb{T}^{k}\subset\mathbb{R}^{\left(m+1\right)k}$,
then we can characterise ${\cal N}_{k}^{\phi}$ as 
\[
{\cal N}_{k}^{\phi}=\left\{ -\log f\left(\cdot;\psi_{k}\right):\psi_{k}\in\mathbb{S}_{k}\right\} ,
\]
where we index by the compact Euclidean space $\mathbb{S}_{k}$ instead
of ${\cal M}_{k}^{\phi}$. Moreover, by A2 and A3 the map $\psi_{k}\mapsto-\log f(x;\psi_{k})$
is continuous for each fixed $x$ (since $f(x;\psi_{k})>0$), and
$\mathbb{S}_{k}$ is compact. Hence, once an envelope in ${\cal L}_{1}(P)$
is available, the bracketing numbers $N_{[]}\!\left(\delta,{\cal N}_{k}^{\phi},{\cal L}_{1}(P)\right)$
are finite for every $\delta>0$ by the second part of Lemma~\ref{lem:GCclass}.
We can write $\ell_{k,n}\left(\psi_{k}\right)=-P_{n}\log f\left(\cdot;\psi_{k}\right)$
and $\ell_{k}\left(\psi_{k}\right)=-P\log f\left(\cdot;\psi_{k}\right)$,
and by A1 and Lemma \ref{lem:GCclass} 
\[
\left\Vert P_{n}-P\right\Vert _{{\cal N}_{k}^{\phi}}=\max_{\psi_{k}\in\mathbb{S}_{k}}\left|\ell_{k,n}\left(\psi_{k}\right)-\ell_{k}\left(\psi_{k}\right)\right|\stackrel[n\to\infty]{\mathrm{a.s.}}{\longrightarrow}0
\]
if there exists a $G\in{\cal L}_{1}\left(P\right)$, such that $\left|\log f\left(x;\psi_{k}\right)\right|\le G\left(x\right)$,
for each $x\in\mathbb{X}$. We prove the following useful bound from
\citet{atienza2007consistency}. 
\begin{lem}
\label{lem:Atienza}Under A2 and A3, for each $f\left(x;\psi_{k}\right)=\sum_{z=1}^{k}\pi_{z}\phi\left(x;\theta_{z}\right)\in{\cal M}_{k}^{\phi}$
and $x\in\mathbb{X}$, 
\[
\left|\log f\left(x;\psi_{k}\right)\right|\le\sum_{z=1}^{k}\left|\log\phi\left(x;\theta_{z}\right)\right|\le k\max_{\theta\in\mathbb{T}}\left|\log\phi\left(x;\theta\right)\right|.
\]
\end{lem}

\begin{proof}
Firstly observe that for every $x\in\mathbb{X}$, 
\[
\min_{z\in\left[k\right]}\phi\left(x;\theta_{z}\right)\le f\left(x;\psi_{k}\right)=\sum_{z=1}^{k}\pi_{z}\phi\left(x;\theta_{z}\right)\le\max_{z\in\left[k\right]}\phi\left(x;\theta_{z}\right)
\]
and thus 
\[
\log\left\{ \min_{z\in\left[k\right]}\phi\left(x;\theta_{z}\right)\right\} \le\log f\left(x;\psi_{k}\right)\le\log\left\{ \max_{z\in\left[k\right]}\phi\left(x;\theta_{z}\right)\right\} .
\]
Then, 
\begin{align*}
\left|\log f\left(x;\psi_{k}\right)\right| & \le\left|\log\left\{ \min_{z\in\left[k\right]}\phi\left(x;\theta_{z}\right)\right\} \right|\vee\left|\log\left\{ \max_{z\in\left[k\right]}\phi\left(x;\theta_{z}\right)\right\} \right|\\
 & \le\sum_{z=1}^{k}\left|\log\phi\left(x;\theta_{z}\right)\right|\le k\max_{\theta\in\mathbb{T}}\left|\log\phi\left(x;\theta\right)\right|,
\end{align*}
as required, where $\max_{\theta\in\mathbb{T}}\left|\log\phi\left(x;\theta\right)\right|$
exists by A2 and A3. 
\end{proof}
Using Lemma \ref{lem:Atienza}, we have $\left|\log f\left(x;\psi_{k}\right)\right|\le k\max_{\theta\in\mathbb{T}}\left|\log\phi\left(x;\theta\right)\right|$.
Then by A4, since we have $\max_{\theta\in\mathbb{T}}\left|\log\phi\left(x;\theta\right)\right|\le G_{1}\left(x\right)$
with $G_{1}\in{\cal L}_{1}\left(P\right)$, we can take $G\left(x\right)=kG_{1}\left(x\right)$,
as required. The conclusion then follows by the fact that the minimum
is Lipschitz with respect to the uniform norm (see, e.g., \citealp[A.4.4]{hinderer2016dynamic}),
i.e., 
\[
\left|\min_{\psi_{k}\in\mathbb{S}_{k}}\ell_{k,n}\left(\psi_{k}\right)-\min_{\psi_{k}\in\mathbb{S}_{k}}\ell_{k}\left(\psi_{k}\right)\right|\le\max_{\psi_{k}\in\mathbb{S}_{k}}\left|\ell_{k,n}\left(\psi_{k}\right)-\ell_{k}\left(\psi_{k}\right)\right|,
\]
where the minima exist by compactness of $\mathbb{S}_{k}$ and continuity
of $\psi_{k}\mapsto\ell_{k,n}\left(\psi_{k}\right)$ and $\psi_{k}\mapsto\ell_{k}\left(\psi_{k}\right)$
(by A2--A4). 

\subsection*{Proof of Lemma \ref{lem:Rateofmin}}

For each $k\ge k_{0}$, we firstly seek to bound $N_{\left[\right]}\left(\delta,{\cal M}_{k}^{\phi},{\cal L}_{1}\left(\mathfrak{m}\right)\right)$.
By A2 and A3, noting that each $f_{k}\in{\cal M}_{k}^{\phi}$ has
the form $f_{k}\left(\cdot;\psi_{k}\right)=\sum_{z=1}^{k}\pi_{z}\phi\left(\cdot;\theta_{z}\right)$,
Lemma \ref{lem:ParametricBracket} implies that if 
\[
\left|\sum_{z=1}^{k}\pi_{z}\phi\left(x;\theta_{z}\right)-\sum_{z=1}^{k}\pi_{z}^{*}\phi\left(x;\theta_{z}^{*}\right)\right|\le G\left(x\right)\left\Vert \psi_{k}-\psi_{k}^{*}\right\Vert ,
\]
for some $G\in{\cal L}_{1}\left(\mathfrak{m}\right)$ and all $\psi_{k},\psi_{k}^{*}\in\mathbb{S}_{k}$,
then 
\[
N_{\left[\right]}\left(\delta,{\cal M}_{k}^{\phi},{\cal L}_{1}\left(\mathfrak{m}\right)\right)\le K\left(\frac{\left\Vert G\right\Vert _{{\cal L}_{1}\left(\mathfrak{m}\right)}\mathrm{diam}\left(\mathbb{S}_{k}\right)}{\delta}\right)^{\left(m+1\right)k},
\]
for some constant $K>0$. Here $\mathrm{diam}\left(\mathbb{S}_{k}\right)<\infty$
since $\mathbb{S}_{k}$ is compact.

Observe that 
\begin{align*}
\left|\sum_{z=1}^{k}\pi_{z}\phi\left(x;\theta_{z}\right)-\sum_{z=1}^{k}\pi_{z}^{*}\phi\left(x;\theta_{z}^{*}\right)\right| & \le\sum_{z=1}^{k}\left|\pi_{z}\phi\left(x;\theta_{z}\right)-\pi_{z}^{*}\phi\left(x;\theta_{z}^{*}\right)\right|\\
 & =\sum_{z=1}^{k}\left|\pi_{z}\phi\left(x;\theta_{z}\right)-\pi_{z}\phi\left(x;\theta_{z}^{*}\right)+\pi_{z}\phi\left(x;\theta_{z}^{*}\right)-\pi_{z}^{*}\phi\left(x;\theta_{z}^{*}\right)\right|\\
 & \le\sum_{z=1}^{k}\left\{ \left|\pi_{z}\right|\left|\phi\left(x;\theta_{z}\right)-\phi\left(x;\theta_{z}^{*}\right)\right|+\left|\phi\left(x;\theta_{z}^{*}\right)\right|\left|\pi_{z}-\pi_{z}^{*}\right|\right\} \\
 & \le\sum_{z=1}^{k}\left\{ \left|\phi\left(x;\theta_{z}\right)-\phi\left(x;\theta_{z}^{*}\right)\right|+\max_{\theta\in\mathbb{T}}\left|\phi\left(x;\theta\right)\right|\left|\pi_{z}-\pi_{z}^{*}\right|\right\} ,
\end{align*}
where we note that $\left|\pi_{z}\right|\le1$ and the maximum exists
by A2 and A3. Then, by A5, we have $L_{\phi}$, such that 
\[
\left|\phi\left(x;\theta_{z}\right)-\phi\left(x;\theta_{z}^{*}\right)\right|\le L_{\phi}\left(x\right)\left\Vert \theta_{z}-\theta_{z}^{*}\right\Vert _{1}
\]
and since there exists a $G_{2}\in{\cal L}_{1}\left(\mathfrak{m}\right)$,
such that $L_{\phi}\left(x\right)+\max_{\theta\in\mathbb{T}}\left|\phi\left(x;\theta\right)\right|\le G_{2}\left(x\right)$,
for every $x\in\mathbb{X}$, it holds that 
\begin{align*}
\left|\sum_{z=1}^{k}\pi_{z}\phi\left(x;\theta_{z}\right)-\sum_{z=1}^{k}\pi_{z}^{*}\phi\left(x;\theta_{z}^{*}\right)\right| & \le\sum_{z=1}^{k}\left\{ L_{\phi}\left(x\right)\left\Vert \theta_{z}-\theta_{z}^{*}\right\Vert _{1}+\max_{\theta\in\mathbb{T}}\left|\phi\left(x;\theta\right)\right|\left|\pi_{z}-\pi_{z}^{*}\right|\right\} \\
 & \le G_{2}\left(x\right)\left\Vert \psi_{k}-\psi_{k}^{*}\right\Vert _{1}\\
 & \le\sqrt{\left(m+1\right)k}G_{2}\left(x\right)\left\Vert \psi_{k}-\psi_{k}^{*}\right\Vert .
\end{align*}
We can therefore take $G=\sqrt{\left(m+1\right)k}G_{2}$, as required.

Noting that ${\cal M}_{k}^{\phi}\left(\delta\right)=\left\{ f\in{\cal M}_{k}^{\phi}:\mathfrak{h}\left(f,f_{0}\right)\le\delta\right\} \subset{\cal M}_{k}^{\phi}$,
we can bound the entropy integral \eqref{eq:entropyintegral} as 
\begin{align*}
J_{\left[\right]}\left(\delta,{\cal M}_{k}^{\phi}\left(\delta\right),{\cal L}_{1}\left(\mathfrak{m}\right)\right) & \le J_{\left[\right]}\left(\delta,{\cal M}_{k}^{\phi},{\cal L}_{1}\left(\mathfrak{m}\right)\right)\\
 & =\delta\vee\int_{0}^{\delta}\sqrt{H_{\left[\right]}\left(u^{2},{\cal M}_{k}^{\phi},{\cal L}_{1}\left(\mathfrak{m}\right)\right)}\mathrm{d}u\\
 & \le\delta\vee\int_{0}^{\delta}\sqrt{\log\left\{ K\left(\frac{\left\Vert G\right\Vert _{{\cal L}_{1}\left(\mathfrak{m}\right)}\mathrm{diam}\left(\mathbb{S}_{k}\right)}{u^{2}}\right)^{\left(m+1\right)k}\right\} }\mathrm{d}u.
\end{align*}
Using the fact that for $a,b\ge0$, $\sqrt{a+b}\le\sqrt{a}+\sqrt{b}$,
we have 
\[
\sqrt{\log\left\{ K\left(\frac{\left\Vert G\right\Vert _{{\cal L}_{1}\left(\mathfrak{m}\right)}\mathrm{diam}\left(\mathbb{S}_{k}\right)}{u^{2}}\right)^{\left(m+1\right)k}\right\} }\le C_{1}+C_{2}\sqrt{\log\left(1/u\right)},
\]
where 
\[
C_{1}=\sqrt{\log K+\left(m+1\right)k\log\left\{ \left\Vert G\right\Vert _{{\cal L}_{1}\left(\mathfrak{m}\right)}\mathrm{diam}\left(\mathbb{S}_{k}\right)\right\} },\qquad C_{2}=\sqrt{2\left(m+1\right)k}.
\]

Jensen's inequality implies that 
\begin{align*}
\int_{0}^{\delta}\sqrt{\log\left(1/u\right)}\mathrm{d}u & \le\sqrt{\delta\int_{0}^{\delta}\log\left(1/u\right)\mathrm{d}u}\\
 & =\sqrt{\delta\left\{ \delta+\delta\log\left(1/\delta\right)\right\} }\le\delta+\delta\sqrt{\log\left(1/\delta\right)}.
\end{align*}
Thus, we have 
\begin{align*}
J\left(\delta\right) & \le\delta\vee\left[\left(C_{1}+C_{2}\right)\delta+C_{2}\delta\sqrt{\log\left(1/\delta\right)}\right]\\
 & \le C_{3}\delta+C_{2}\delta\sqrt{\log\left(1/\delta\right)},
\end{align*}
where $C_{3}=1+C_{1}+C_{2}$. Then, taking $\Psi\left(\delta\right)=C_{3}\delta+C_{2}\delta\sqrt{\log\left(1/\delta\right)}$,
we observe $\Psi/\delta^{2}=C_{3}\delta^{-1}+C_{2}\delta^{-1}\sqrt{\log\left(1/\delta\right)}$
is non-increasing, and thus Proposition \ref{prop:MainTechProp} applies.

Substituting $\Psi$ into $\sqrt{n}\delta_{n}^{2}\ge c\Psi\left(\delta_{n}\right)$
yields the inequality 
\[
\sqrt{n}\delta_{n}\ge cC_{3}+cC_{2}\sqrt{\log\left(1/\delta_{n}\right)}\implies\frac{\sqrt{n}\delta_{n}}{cC_{3}+cC_{2}\sqrt{\log\left(1/\delta_{n}\right)}}\ge1
\]
which is satisfied if we take $\delta_{n}=C_{4}\sqrt{\log n/n}$,
for a sufficiently large $C_{4}>0$ and all sufficiently large $n$,
since 
\[
\frac{\sqrt{n}\delta_{n}}{cC_{3}+cC_{2}\sqrt{\log\left(1/\delta_{n}\right)}}=\frac{C_{4}\sqrt{\log n}}{cC_{3}+\left(c/\sqrt{2}\right)C_{2}\sqrt{\log\left(n/C_{4}^{2}\right)-\log\log n}}.
\]
Then, the conclusion of Proposition \ref{prop:MainTechProp} implies
that for $\delta=\delta_{n}$, 
\[
\mathrm{P}\left(\frac{n}{\log n}\left|P_{n}\log\hat{f}_{k,n}-P_{n}\log f_{0}\right|\ge C_{4}^{2}\right)\le c_{1}\exp\left(-\frac{C_{4}^{2}}{c_{1}^{2}}\log n\right)
\]
holds for an MLE $\hat{f}_{k,n}\in\arg\max_{f\in{\cal M}_{k}^{\phi}}P_{n}\log f$,
since $P_{n}\log\hat{f}_{k,n}\ge P_{n}\log f$, for all $f\in{\cal M}_{k}^{\phi}$.
Thus, by definition of boundedness in probability 
\[
\frac{n}{\log n}\left\{ P_{n}\log\hat{f}_{k,n}-P_{n}\log f_{0}\right\} =O_{\mathrm{P}}\left(1\right).
\]
By writing $\hat{f}_{k,n}=f\left(\cdot;\hat{\psi}_{k,n}\right)$ for
$\hat{\psi}_{k,n}\in\arg\min_{\psi_{k}\in\mathbb{S}_{k}}\ell_{k,n}\left(\psi_{k}\right)$,
we have for each $k,l\ge k_{0}$, 
\begin{align*}
 & \frac{n}{\log n}\left\{ \min_{\psi_{k}\in\mathbb{S}_{k}}\ell_{k,n}\left(\psi_{k}\right)-\min_{\psi_{l}\in\mathbb{S}_{l}}\ell_{l,n}\left(\psi_{l}\right)\right\} \\
= & \frac{n}{\log n}\left\{ -P_{n}\log\hat{f}_{k,n}+P_{n}\log\hat{f}_{l,n}\right\} \\
= & -\frac{n}{\log n}\left\{ P_{n}\log\hat{f}_{k,n}-P_{n}\log f_{0}\right\} +\frac{n}{\log n}\left\{ P_{n}\log\hat{f}_{l,n}-P_{n}\log f_{0}\right\} \\
= & O_{\mathrm{P}}\left(1\right)+O_{\mathrm{P}}\left(1\right)=O_{\mathrm{P}}\left(1\right),
\end{align*}
as required. 

\subsection*{Proof of Proposition \ref{prop:Misspec_convergence}}

Fix $k\in\left[\bar{k}\right]$. Under A1{*} and A2--A4, Lemma \ref{lem:Convergenceofmin}
holds exactly as stated, and hence 
\[
\hat{L}_{k,n}\stackrel[n\to\infty]{\mathrm{a.s.}}{\longrightarrow}L_{k}.
\]
Since $\bar{k}$ is finite, taking the intersection over all $k\in\left[\bar{k}\right]$
yields 
\[
\max_{k\in\left[\bar{k}\right]}\left|\hat{L}_{k,n}-L_{k}\right|\stackrel[n\to\infty]{\mathrm{a.s.}}{\longrightarrow}0.
\]
By B1, for each fixed $k$, $\mathrm{pen}_{k,n}\to0$ and thus, from
finiteness of $\bar{k}$, we also have 
\[
\max_{k\in\left[\bar{k}\right]}\mathrm{pen}_{k,n}\underset{n\to\infty}{\longrightarrow}0.
\]
Define the penalised empirical criterion as $C_{k,n}=\hat{L}_{k,n}+\mathrm{pen}_{k,n}$
and write $C_{k}=L_{k}$. Then, it follows that 
\[
\max_{k\in\left[\bar{k}\right]}\left|C_{k,n}-C_{k}\right|\stackrel[n\to\infty]{\mathrm{a.s.}}{\longrightarrow}0.
\]

Let $C^{*}=\min_{k\in\left[\bar{k}\right]}C_{k}$. If ${\cal K}^{*}=\left[\bar{k}\right]$
then the result is trivial. Otherwise, define a positive gap 
\[
\gamma=\min_{k\notin{\cal K}^{*}}\left(C_{k}-C^{*}\right)>0,
\]
where the strict positivity arises because $\bar{k}$ is finite and
${\cal K}^{*}\ne\left[\bar{k}\right]$. Set $\epsilon=\gamma/4$.
Then, it holds that with probability one, there exists an $n_{0}$
such that for all $n\ge n_{0}$, 
\[
\max_{k\in\left[\bar{k}\right]}\left|C_{k,n}-C_{k}\right|<\epsilon.
\]
Fix $n\ge n_{0}$. If $k\notin{\cal K}^{*}$, then we have 
\[
C_{k,n}\ge C_{k}-\epsilon\ge\left(C^{*}+\gamma\right)-\epsilon=C^{*}+\frac{3}{4}\gamma.
\]
On the other hand, if $k^{*}\in{\cal K}^{*}$, we have 
\[
C_{k^{*},n}\le C_{k^{*}}+\epsilon=C^{*}+\epsilon=C^{*}+\frac{1}{4}\gamma.
\]
Thus, for every $n\ge n_{0}$, 
\[
\min_{k\notin{\cal K}^{*}}C_{k,n}>\min_{k^{*}\in{\cal K}^{*}}C_{k^{*},n}
\]
and therefore every minimiser of $k\mapsto C_{k,n}$ is in ${\cal K}^{*}$,
thus $\hat{k}_{n}\in{\cal K}^{*}$. Hence $\mathrm{P}\!\left(\hat{k}_{n}\in{\cal K}^{*}\right)\to1$. 

\subsection*{Proof of Proposition \ref{prop:no_min_pen}}

Fix an arbitrary $\mathrm{pen}\in\mathscr{P}$. For each $k$ and
$l$ satisfying $1\le k<l\le\bar{k}$, define 
\[
\Delta_{k,l}\left(n\right)=\frac{n}{\log n}\left\{ \mathrm{pen}_{l,n}-\mathrm{pen}_{k,n}\right\} .
\]
By B2, $\Delta_{k,l}\left(n\right)\to\infty$ as $n\to\infty$. Since
the set of choices of indices $k$ and $l$ is finite, it holds that
\[
\Delta_{\min}\left(n\right)=\min_{1\le k<l\le\bar{k}}\Delta_{k,l}\left(n\right)\to\infty.
\]
Indeed, if $\Delta_{\min}\left(n\right)$ did not diverge, there would
exist a subsequence $\left(n_{j}\right)_{j}$ and a constant $M<\infty$,
such that $\Delta_{\min}\left(n_{j}\right)\le M$, for all $j$. For
each $j$, choose a pair $\left(k_{j},l_{j}\right)$ that attains
the minimum. By finiteness of the set of indices, some pair $\left(k,l\right)$
occurs infinitely often along the sequence $\left(k_{j},l_{j}\right)_{j}$,
hence along a further subsequence, we would have $\Delta_{k,l}\left(n_{j}\right)\le M$,
which contradicts the fact that $\Delta_{k,l}\left(n\right)\to\infty$.

Now, set $b_{n}=\sqrt{\Delta_{\min}\left(n\right)}$, and note that
$b_{n}\to\infty$. Define the rescaled penalty $\widetilde{\mathrm{pen}}_{k,n}=\mathrm{pen}_{k,n}/b_{n}$,
for each $k\in\left[\bar{k}\right]$ and $n\in\mathbb{N}$. Observe
that for fixed $k\in\left[\bar{k}\right]$, since $b_{n}\to\infty$
and $\mathrm{pen}_{k,n}\to0$ (by B1), we have $\widetilde{\mathrm{pen}}_{k,n}\to0$
as $n\to\infty$, and thus $\widetilde{\mathrm{pen}}$ satisfies B1.
Next, for fixed $k$ and $l$, with $1\le k<l\le\bar{k}$, we have
\[
\frac{n}{\log n}\left\{ \widetilde{\mathrm{pen}}_{l,n}-\widetilde{\mathrm{pen}}_{k,n}\right\} =\frac{n}{b_{n}\log n}\left\{ \mathrm{pen}_{l,n}-\mathrm{pen}_{k,n}\right\} \ge\frac{1}{b_{n}}\Delta_{\min}\left(n\right)=\sqrt{\Delta_{\min}\left(n\right)}\to\infty,
\]
and hence $\widetilde{\mathrm{pen}}$ satisfies B2, and consequently,
$\widetilde{\mathrm{pen}}\in\mathscr{P}$.

Finally, since $b_{n}>1$ for all sufficiently large $n$, we have
$\widetilde{\mathrm{pen}}_{k,n}<\mathrm{pen}_{k,n}$, for all $k\in\left[\bar{k}\right]$
and all sufficiently large $n$. Since $\mathrm{pen}\in\mathscr{P}$
was arbitrary, we have shown that $\left(\mathscr{P},\preceq\right)$
has no minimal element. 

\subsection*{Proof of Proposition \ref{prop:AIC-minimax}}

We start by proving that $\mathscr{R}_{n}^{*}\gtrsim n^{-1}$. Since
${\cal M}_{1}\subset{\cal F}$, we have 
\[
\mathscr{R}_{n}^{*}\gtrsim\inf_{\tilde{f}_{n}}\sup_{\mu\in\left[-b,b\right]}\mathrm{E}_{\phi\left(\cdot;\mu\right)}\left\{ \mathfrak{h}^{2}\left(\tilde{f}_{n},\phi\left(\cdot;\mu\right)\right)\right\} .
\]
Fix $\Delta>0$ and set $\mu_{\pm}=\pm\Delta/\sqrt{n}$, and observe
that for large $n$, $\mu_{\pm}\in\left[-b,b\right]$. Let $f_{\pm}=\phi\left(\cdot;\mu_{\pm}\right)$.
It holds that 
\[
\int_{-\infty}^{\infty}\sqrt{f_{-}\left(x\right)f_{+}\left(x\right)}\,\mathrm{d}x=\exp\left\{ -\frac{\left(\mu_{+}-\mu_{-}\right)^{2}}{8}\right\} =\exp\left\{ -\frac{\Delta^{2}}{2n}\right\} ,
\]
hence 
\[
\mathfrak{h}^{2}\left(f_{-},f_{+}\right)=1-\int_{-\infty}^{\infty}\sqrt{f_{-}\left(x\right)f_{+}\left(x\right)}\,\mathrm{d}x=1-\exp\left\{ -\frac{\Delta^{2}}{2n}\right\} \ge\frac{\Delta^{2}}{4n},
\]
since $1-\exp\left(-u\right)\ge u/2$ for all $u\in\left[0,1\right]$,
and $\Delta^{2}<2n$ for sufficiently large $n$. Next, we apply Le
Cam's two point method (cf. \citealp[Sec. 15.2]{Wainwright2019hds})
to get 
\[
\max\left\{ \mathrm{E}_{f_{-}}\mathfrak{h}^{2}\left(\tilde{f}_{n},f_{-}\right),\mathrm{E}_{f_{+}}\mathfrak{h}^{2}\left(\tilde{f}_{n},f_{+}\right)\right\} \ge c_{1}\mathfrak{h}^{2}\left(f_{-},f_{+}\right)\left\{ 1-\left\Vert f_{-}^{\otimes n}-f_{+}^{\otimes n}\right\Vert _{\mathrm{TV}}\right\} ,
\]
where $f^{\otimes n}$ is the density of $\mathbf{X}_{n}$ consisting
of IID replicates of $X$ with density $f$, and $c_{1}>0$ is an
absolute constant. Next, observe that 
\[
\mathfrak{h}^{2}\left(f_{-}^{\otimes n},f_{+}^{\otimes n}\right)=1-\int_{\mathbb{R}^{n}}\sqrt{f_{-}^{\otimes n}\left(\mathbf{x}_{n}\right)f_{+}^{\otimes n}\left(\mathbf{x}_{n}\right)}\,\mathrm{d}\mathbf{x}_{n}=1-\prod_{i=1}^{n}\int_{-\infty}^{\infty}\sqrt{f_{-}\left(x_{i}\right)f_{+}\left(x_{i}\right)}\,\mathrm{d}x_{i}.
\]
Thus, 
\[
\mathfrak{h}^{2}\left(f_{-}^{\otimes n},f_{+}^{\otimes n}\right)=1-\left\{ \int_{-\infty}^{\infty}\sqrt{f_{-}\left(x\right)f_{+}\left(x\right)}\,\mathrm{d}x\right\} ^{n}=1-\left(\exp\left\{ -\frac{\Delta^{2}}{2n}\right\} \right)^{n}=1-\exp\left\{ -\frac{\Delta^{2}}{2}\right\} .
\]
Using the comparison inequality $\left\Vert f_{-}^{\otimes n}-f_{+}^{\otimes n}\right\Vert _{\mathrm{TV}}\le\sqrt{2}\mathfrak{h}\left(f_{-}^{\otimes n},f_{+}^{\otimes n}\right)$,
we get 
\[
\left\Vert f_{-}^{\otimes n}-f_{+}^{\otimes n}\right\Vert _{\mathrm{TV}}\le\sqrt{2}\sqrt{1-\exp\left\{ -\Delta^{2}/2\right\} }.
\]
Choosing $\Delta$ sufficiently small so that $1-\left\Vert f_{-}^{\otimes n}-f_{+}^{\otimes n}\right\Vert _{\mathrm{TV}}\ge c_{2}>0$,
we have 
\[
\max\left\{ \mathrm{E}_{f_{-}}\mathfrak{h}^{2}\left(\tilde{f}_{n},f_{-}\right),\mathrm{E}_{f_{+}}\mathfrak{h}^{2}\left(\tilde{f}_{n},f_{+}\right)\right\} \ge c_{1}c_{2}\frac{\Delta^{2}}{4n}=\frac{c}{n}.
\]
But since $\mu_{\pm}\in\left[-b,b\right]$, it holds by definition
that 
\[
\mathscr{R}_{n}^{*}\ge\max\left\{ \mathrm{E}_{f_{-}}\mathfrak{h}^{2}\left(\tilde{f}_{n},f_{-}\right),\mathrm{E}_{f_{+}}\mathfrak{h}^{2}\left(\tilde{f}_{n},f_{+}\right)\right\} \ge\frac{c}{n},
\]
proving the first inequality. The second inequality is immediate by
definition of $\mathscr{R}_{n}^{*}$.

To proceed, we require \citet[Thm. 7.11]{Massart2007Concentration}
which we reproduce in the context of this work. We firstly note that
for any nonnegative $f,g\in{\cal L}_{1}(\mathfrak{m})$, we will write
\[
\mathfrak{h}(f,g)=\left\{ \frac{1}{2}\int_{\mathbb{X}}\left(f^{1/2}-g^{1/2}\right)^{2}\,\mathrm{d}\mathfrak{m}\right\} ^{1/2}.
\]
If $f$ and $g$ are probability densities, then $\mathfrak{h}(f,g)$
is the usual Hellinger distance. We write $H_{[\cdot]}(\epsilon,\mathcal{F},\mathfrak{h})$
for the $\mathfrak{h}$-bracketing entropy of $\mathcal{F}$, where
brackets $[f^{\mathrm{L}},f^{\mathrm{U}}]$ may have endpoints $f^{\mathrm{L}},f^{\mathrm{U}}\in{\cal L}_{1}(\mathfrak{m})$
that are nonnegative but not necessarily densities, and $\mathfrak{h}(f^{\mathrm{L}},f^{\mathrm{U}})\le\epsilon$. 
\begin{thm}
\label{thm:Massart711_Hellinger_manuscript} Let $\mathbf{X}_{n}$
consist of IID replicates with common density $f_{0}$ with respect
to $\mathfrak{m}$ on $\mathbb{X}$. Let $\{\mathcal{F}_{k}\}_{k\in\mathcal{K}}$
be an at most countable collection of models, where each $\mathcal{F}_{k}$
is a set of densities with respect to $\mathfrak{m}$. For each $k\in\mathcal{K}$
let $\hat{f}_{k,n}\in\arg\min_{f\in\mathcal{F}_{k}}\ell_{n}(f)$ be
an MLE on $\mathcal{F}_{k}$. Assume the following: 
\begin{enumerate}
\item For each $k\in\mathcal{K}$ there exists a countable subset $\mathcal{F}_{k}^{\circ}\subset\mathcal{F}_{k}$
such that for every $f\in\mathcal{F}_{k}$ there exists $(f_{j})_{j\ge1}\subset\mathcal{F}_{k}^{\circ}$
with $\log f_{j}(x)\to\log f(x)$ for every $x\in\mathbb{X}$. 
\item For each $k\in\mathcal{K}$ there exists a function $\Psi_{k}:\mathbb{R}_{>0}\to\mathbb{R}_{>0}$
such that $\Psi_{k}$ is nondecreasing and $\xi\mapsto\Psi_{k}(\xi)/\xi$
is nonincreasing on $(0,\infty)$, and such that for every $\delta>0$
and every $f\in\mathcal{F}_{k}$, writing the $\delta$-Hellinger
ball 
\[
\mathcal{F}_{k}(f,\delta)=\{g\in\mathcal{F}_{k}:\mathfrak{h}(g,f)\le\delta\},
\]
one has 
\[
\int_{0}^{\delta}\sqrt{H_{[\ ]}\!\left(u,\,\mathcal{F}_{k}(f,\delta),\,\mathfrak{h}\right)}\,\mathrm{d}u\le\Psi_{k}(\delta).
\]
\item The sequence $(\rho_{k})_{k\in\mathcal{K}}$ consists of nonnegative
weights such that 
\[
\Upsilon=\sum_{k\in\mathcal{K}}e^{-\rho_{k}}<\infty.
\]
\end{enumerate}
\noindent For each $k\in\mathcal{K}$ let $\xi_{k}>0$ be the unique
solution of 
\[
\Psi_{k}(\xi)=\sqrt{n}\,\xi^{2}.
\]
Let $\mathrm{pen}_{k,n}\ge0$ be penalties and define the penalised
selector 
\[
\hat{k}_{n}=\min\arg\min_{k\in\mathcal{K}}\Bigl\{\ell_{n}(\hat{f}_{k,n})+\mathrm{pen}_{k,n}\Bigr\},\hat{f}_{n}=\hat{f}_{\hat{k}_{n},n}.
\]
Then there exist absolute constants $\kappa>0$ and $C>0$ such that
if, for all $k\in\mathcal{K}$, 
\[
\mathrm{pen}_{k,n}\ge\kappa\left(\xi_{k}^{2}+\frac{\rho_{k}}{n}\right),
\]
a minimiser $\hat{k}_{n}$ exists and moreover 
\[
\mathrm{E}_{f_{0}}\bigl\{\mathfrak{h}^{2}(f_{0},\hat{f}_{n})\bigr\}\le C\left[\inf_{k\in\mathcal{K}}\Bigl\{\mathfrak{K}(f_{0},\mathcal{F}_{k})+\mathrm{pen}_{k,n}\Bigr\}+\frac{\Upsilon}{n}\right],
\]
where $\mathfrak{K}(f_{0},\mathcal{F}_{k})=\inf_{f\in\mathcal{F}_{k}}\mathfrak{K}(f_{0},f)$. 
\end{thm}

To make use of Theorem \ref{thm:Massart711_Hellinger_manuscript},
we need the following lemmas. 
\begin{lem}
\label{lem:local_h_bracketing_Fk} Let $\mathcal{F}_{1}=\{\phi(\cdot;\mu):\mu\in[-b,b]\}$
and 
\[
\mathcal{F}_{2}=\Bigl\{ f(\cdot;\theta)=\pi\phi(\cdot;\mu_{1})+(1-\pi)\phi(\cdot;\mu_{2}):\theta=(\pi,\mu_{1},\mu_{2})\in\mathbb{T}_{2}\Bigr\},
\]
where 
\[
\mathbb{T}_{2}=[\underline{\pi},1-\underline{\pi}]\times\{(\mu_{1},\mu_{2})\in[-b,b]^{2}:\mu_{1}\le\mu_{2}\},
\]
for some fixed $\underline{\pi}\in(0,1/2)$, and write $d_{1}=1$
and $d_{2}=3$ for the parameter dimensions. For $k\in\{1,2\}$ and
$f\in\mathcal{F}_{k}$, define the $\delta$-Hellinger ball 
\[
\mathcal{F}_{k}(f,\delta)=\{g\in\mathcal{F}_{k}:\mathfrak{h}(g,f)\le\delta\}.
\]
Then, for each $k\in\{1,2\}$, there exists a constant $C_{k}<\infty$
(depending only on $(b,\underline{\pi})$) such that for every $f\in\mathcal{F}_{k}$,
every $\delta\in(0,1]$, and every $u\in(0,\delta]$, 
\[
H_{[\cdot]}\left(u,\mathcal{F}_{k}(f,\delta),\mathfrak{h}\right)\le d_{k}\log\left(\frac{C_{k}\,\delta}{u}\right).
\]
In particular, defining 
\[
B_{k}=\int_{0}^{1}\sqrt{\log\left(\frac{C_{k}}{t}\right)}\mathrm{d}t,\Psi_{k}(\delta)=B_{k}\sqrt{d_{k}}\delta,
\]
one has, for all $f\in\mathcal{F}_{k}$ and all $\delta\in(0,1]$,
\[
\int_{0}^{\delta}\sqrt{H_{[\cdot]}\left(u,\mathcal{F}_{k}(f,\delta),\mathfrak{h}\right)}\mathrm{d}u\le\Psi_{k}(\delta).
\]
\end{lem}

\begin{proof}
Fix $k\in\{1,2\}$ and write $\mathcal{F}_{k}=\{f(\cdot;\theta):\theta\in\mathbb{T}_{k}\}$
with $\mathbb{T}_{1}=[-b,b]$ and $\mathbb{T}_{2}$ as above. Let
$s_{\theta}=\sqrt{f(\cdot;\theta)}$, so that 
\[
\mathfrak{h}(f(\cdot;\theta),f(\cdot;\theta'))=\frac{1}{\sqrt{2}}\|s_{\theta}-s_{\theta'}\|_{2}.
\]
For each $k\in\left\{ 1,2\right\} $, we show that there exists a
nonnegative $A_{k}\in{\cal L}_{2}(\mathfrak{m})$ such that for all
$x$, 
\begin{equation}
\sup_{\theta\in\mathbb{T}_{k}}\|\frac{\partial}{\partial\theta}s_{\theta}(x)\|\le A_{k}(x).\label{eq:Ak_envelope}
\end{equation}
For $k=1$, $s_{\mu}=\phi(\cdot;\mu)^{1/2}$ and $\left(\partial/\partial\mu\right)s_{\mu}(x)=(x-\mu)s_{\mu}(x)/2$,
so 
\[
\sup_{\mu\in[-b,b]}|\left(\partial/\partial\mu\right)s_{\mu}(x)|\le\frac{|x|+b}{2}\sup_{\mu\in[-b,b]}\phi(x;\mu)^{1/2}.
\]
For $k=2$, the derivatives are 
\[
\frac{\partial}{\partial\mu_{1}}s_{\theta}(x)=\frac{\pi(x-\mu_{1})\phi(x;\mu_{1})}{2\sqrt{f(x;\theta)}},\frac{\partial}{\partial\mu_{2}}s_{\theta}(x)=\frac{(1-\pi)(x-\mu_{2})\phi(x;\mu_{2})}{2\sqrt{f(x;\theta)}},
\]
and 
\[
\frac{\partial}{\partial\pi}s_{\theta}(x)=\frac{\phi(x;\mu_{1})-\phi(x;\mu_{2})}{2\sqrt{f(x;\theta)}}.
\]
Using $\pi\ge\underline{\pi}$ and $1-\pi\ge\underline{\pi}$, we
have $f(x;\theta)\ge\pi\phi(x;\mu_{1})$ and $f(x;\theta)\ge(1-\pi)\phi(x;\mu_{2})$,
hence 
\[
\left|\frac{\partial}{\partial\mu_{1}}s_{\theta}(x)\right|\le\frac{\sqrt{\pi}}{2}|x-\mu_{1}|\phi(x;\mu_{1})^{1/2}\le\frac{|x|+b}{2}\sup_{\mu\in[-b,b]}\phi(x;\mu)^{1/2},
\]
and similarly for $\left|\left(\partial/\partial\mu_{2}\right)s_{\theta}(x)\right|$.
Moreover, since $f(x;\theta)\ge\underline{\pi}(\phi(x;\mu_{1})+\phi(x;\mu_{2}))$
and $|\phi_{1}-\phi_{2}|\le\phi_{1}+\phi_{2}$, 
\[
\left|\frac{\partial}{\partial\pi}s_{\theta}(x)\right|\le\frac{1}{2\sqrt{\underline{\pi}}}\sqrt{\phi(x;\mu_{1})+\phi(x;\mu_{2})}\le\frac{1}{\sqrt{\underline{\pi}}}\sup_{\mu\in[-b,b]}\phi(x;\mu)^{1/2}.
\]
These bounds yield \eqref{eq:Ak_envelope} with an $A_{k}$ of the
form $A_{k}(x)\asymp(|x|+b+1)\sup_{\mu\in[-b,b]}\phi(x;\mu)^{1/2}$,
which belongs to ${\cal L}_{2}(\mathfrak{m})$ because $(\sup_{\mu\in\left[-b,b\right]}\phi(x;\mu)^{1/2})^{2}=\sup_{\mu\in\left[-b,b\right]}\phi(x;\mu)$
has Gaussian tails, since 
\[
\sup_{\mu\in\left[-b,b\right]}\phi\left(x;\mu\right)=\frac{1}{\sqrt{2\pi}}\exp\left\{ -\frac{1}{2}\left[\left|x\right|-b\right]_{+}^{2}\right\} .
\]

Next, by the mean-value theorem and \eqref{eq:Ak_envelope}, 
\[
|s_{\theta}(x)-s_{\theta^{\prime}}(x)|\le\|\theta-\theta^{\prime}\|A_{k}(x)
\]
and hence $\|s_{\theta}-s_{\theta'}\|_{2}\le\|\theta-\theta^{\prime}\|\|A_{k}\|_{2}$.
Therefore, with $L_{k}=\|A_{k}\|_{2}/\sqrt{2}$, 
\[
\mathfrak{h}(f(\cdot;\theta),f(\cdot;\theta^{\prime}))\le L_{k}\|\theta-\theta^{\prime}\|.
\]
Fix $f=f(\cdot;\theta_{0})$. Then 
\[
\mathcal{F}_{k}(f,\delta)\subset\{f(\cdot;\theta):\|\theta-\theta_{0}\|\le\delta/L_{k}\}.
\]
We shall write $R=\delta/L_{k}$. For fixed $u\in(0,\delta]$, set
$r=u/(2L_{k})$. Let $\{\theta_{j}\}_{j=1}^{N}$ be an $r$-net of
the Euclidean ball $\mathbb{B}(\theta_{0},R)\subset\mathbb{R}^{d_{k}}$.
For each $j$, define 
\[
\underline{s}_{j}(x)=\left[s_{\theta_{j}}(x)-rA_{k}(x)\right]_{+},\overline{s}_{j}(x)=s_{\theta_{j}}(x)+rA_{k}(x),f_{j}^{\text{L}}=\underline{s}_{j}^{2},f_{j}^{\text{U}}=\overline{s}_{j}^{2}.
\]
If $\|\theta-\theta_{j}\|\le r$, then we have $|s_{\theta}-s_{\theta_{j}}|\le rA_{k}$
pointwise, from the Lipschitz property, hence $\underline{s}_{j}\le s_{\theta}\le\overline{s}_{j}$
and therefore $f_{j}^{\text{L}}\le f(\cdot;\theta)\le f_{j}^{\text{U}}$.
Furthermore, 
\[
\mathfrak{h}(f_{j}^{\text{L}},f_{j}^{\text{U}})=\frac{1}{\sqrt{2}}\|\overline{s}_{j}-\underline{s}_{j}\|_{2}\le\frac{1}{\sqrt{2}}\|2rA_{k}\|_{2}=\sqrt{2}\,r\,\|A_{k}\|_{2}=u.
\]
Thus the $N$ brackets $\{[f_{j}^{\text{L}},f_{j}^{\text{U}}]\}_{j\in\left[N\right]}$
form a $u$-bracketing of $\mathcal{F}_{k}(f,\delta)$ in the metric
$\mathfrak{h}$, so $N_{[]}(u,\mathcal{F}_{k}(f,\delta),\mathfrak{h})\le N$.

Since $u\le\delta$, we have $r=u/(2L_{k})\le\delta/(2L_{k})=R/2$,
hence $r\le R$. By the standard volumetric bound for Euclidean nets
(e.g. \citealp[Cor. 4.2.13]{vershynin2018hdp}), 
\[
N\le\left(1+\frac{2R}{r}\right)^{d_{k}}=\left(1+\frac{4\delta}{u}\right)^{d_{k}}\le\left(\frac{5\delta}{u}\right)^{d_{k}}.
\]
Taking logarithms yields, with $C_{k}=5$, 
\[
H_{[]}\left(u,\mathcal{F}_{k}(f,\delta),\mathfrak{h}\right)\le d_{k}\log\left(\frac{C_{k}\delta}{u}\right).
\]

\noindent Using the bound and making the substitution $u=\delta t$
gives 
\[
\int_{0}^{\delta}\sqrt{H_{[]}(u,\mathcal{F}_{k}(f,\delta),\mathfrak{h})}\mathrm{d}u\le\delta\sqrt{d_{k}}\int_{0}^{1}\sqrt{\log\left(\frac{C_{k}}{t}\right)}\mathrm{d}t=B_{k}\sqrt{d_{k}}\delta,
\]
which is the claimed admissible choice of $\Psi_{k}$. 
\end{proof}
\begin{lem}
\label{lem:xi_parametric_rate} Assume the setting of Lemma \ref{lem:local_h_bracketing_Fk}.
For $k\in\{1,2\}$ define 
\[
\Psi_{k}(\delta)=B_{k}\sqrt{d_{k}}\delta,
\]
for $\delta>0$, where $B_{k}<\infty$ is as in Lemma \ref{lem:local_h_bracketing_Fk}.
Then $\Psi_{k}$ is nondecreasing on $(0,\infty)$ and $\delta\mapsto\Psi_{k}(\delta)/\delta$
is nonincreasing on $(0,\infty)$. Moreover, the equation $\Psi_{k}(\xi)=\sqrt{n}\,\xi^{2}$
has the unique solution 
\[
\xi_{k}=\frac{B_{k}\sqrt{d_{k}}}{\sqrt{n}}.
\]
In particular, since $d_{k}\in\{1,3\}$, one has $\xi_{k}^{2}=O(n^{-1})$. 
\end{lem}

\begin{proof}
The function $\Psi_{k}(\delta)=B_{k}\sqrt{d_{k}}\,\delta$ is linear
in $\delta$, hence nondecreasing, and $\Psi_{k}(\delta)/\delta=B_{k}\sqrt{d_{k}}$
is constant, hence nonincreasing. For $\xi>0$, $\Psi_{k}(\xi)=\sqrt{n}\,\xi^{2}$
is equivalent to $B_{k}\sqrt{d_{k}}=\sqrt{n}\,\xi$, which directly
yields the unique solution $\xi_{k}=B_{k}\sqrt{d_{k}}/\sqrt{n}$. 
\end{proof}
To prove the final inequality, we apply Theorem \ref{thm:Massart711_Hellinger_manuscript}
with $\rho_{1}=\rho_{2}=1$ and $\Upsilon=\sum_{k\in\left\{ 1,2\right\} }\exp\left(-\rho_{k}\right)=2/e$.
By Lemma \ref{lem:local_h_bracketing_Fk}, Condition 2 of Theorem
\ref{thm:Massart711_Hellinger_manuscript} holds with $\Psi_{k}(\delta)=B_{k}\sqrt{d_{k}}\,\delta$.
By Lemma \ref{lem:xi_parametric_rate}, the fixed point equation $\Psi_{k}(\xi)=\sqrt{n}\xi^{2}$
has the unique solution $\xi_{k}=B_{k}\sqrt{d_{k}}/\sqrt{n}$, so
$\xi_{k}^{2}=B_{k}^{2}d_{k}/n$.

Define an AIC-type penalty $\mathrm{pen}_{k,n}^{\mathrm{AIC}}=ad_{k}/n$
with $a>0$ chosen so that 
\[
\frac{ad_{k}}{n}\ge\kappa\left(\frac{B_{k}^{2}d_{k}}{n}+\frac{1}{n}\right),\qquad k\in\{1,2\},
\]
where $\kappa>0$ is the absolute constant in Theorem \ref{thm:Massart711_Hellinger_manuscript}.
Such a choice is possible since $d_{k}\in\{1,3\}$ is fixed and $B_{k}<\infty$
depends only on $b$ and $\underline{\pi}$. Recall that $\hat{k}_{n}^{\mathrm{AIC}}$
is the corresponding penalised selector and $\hat{f}_{n}^{\mathrm{AIC}}=\hat{f}_{\hat{k}_{n}^{\mathrm{AIC}},n}$.

Then Theorem \ref{thm:Massart711_Hellinger_manuscript} yields, for
every $f_{0}\in{\cal F}$, 
\[
\mathrm{E}_{f_{0}}\bigl\{\mathfrak{h}^{2}(f_{0},\hat{f}_{n}^{\mathrm{AIC}})\bigr\}\le C\left[\inf_{k\in\{1,2\}}\Bigl\{\mathfrak{K}(f_{0},\mathcal{F}_{k})+\mathrm{pen}_{k,n}^{\mathrm{AIC}}\Bigr\}+\frac{\Upsilon}{n}\right],
\]
where $C>0$ is the absolute constant in Theorem \ref{thm:Massart711_Hellinger_manuscript}.
Now, fix $f_{0}\in{\cal F}=\mathcal{F}_{1}\cup\mathcal{F}_{2}$ and
let $k_{0}\in\{1,2\}$ be such that $f_{0}\in\mathcal{F}_{k_{0}}$.
Then $\mathfrak{K}(f_{0},\mathcal{F}_{k_{0}})=0$, and hence 
\[
\mathrm{E}_{f_{0}}\bigl\{\mathfrak{h}^{2}(f_{0},\hat{f}_{n}^{\mathrm{AIC}})\bigr\}\le C\left(\frac{ad_{k_{0}}}{n}+\frac{\Upsilon}{n}\right)\le\frac{C^{\prime}}{n},
\]
with $C^{\prime}=C(ad_{2}+\Upsilon)$ and $d_{2}=3$. Taking the supremum
over $f_{0}\in{\cal F}$ gives 
\[
\sup_{f_{0}\in{\cal F}}\mathrm{E}_{f_{0}}\bigl\{\mathfrak{h}^{2}(\hat{f}_{n}^{\mathrm{AIC}},f_{0})\bigr\}\le\frac{C^{\prime}}{n},
\]
as required.

\subsection*{Proof of Proposition \ref{prop:consistent_not_minimax}}

Recall that $L_{k,n}=\sup_{f\in\mathcal{F}_{k}}P_{n}\log f=-\inf_{f\in\mathcal{F}_{k}}\ell_{n}(f)$.
Then $\hat{k}_{n}=2$ implies 
\begin{equation}
L_{2,n}-L_{1,n}\ge\Delta_{n}.\label{eq:select2_implies}
\end{equation}
For $\delta>0$, we define the symmetric two-component density 
\[
x\mapsto f_{\delta}(x)=\frac{1}{2}\phi(x;-\delta)+\frac{1}{2}\phi(x;+\delta),
\]
which belongs to $\mathcal{F}_{2}$ provided that $\delta\le b$ and
$\underline{\pi}\le1/2\le1-\underline{\pi}$. Let $f_{*}=\phi(\cdot;0)\in\mathcal{F}_{1}$.
We require the following result. 
\begin{lem}
\label{lem:delta4_geometry} There exist constants $\delta_{0}\in(0,1]$
and $0<c_{0}\le C_{0}<\infty$, depending only on $b$ and $\underline{\pi}$,
such that for all $0<\delta\le\delta_{0}$, 
\[
c_{0}\delta^{4}\le\mathfrak{h}^{2}\bigl(f_{\delta},\mathcal{F}_{1}\bigr)\le C_{0}\delta^{4},\qquad c_{0}\delta^{4}\le\mathfrak{K}\bigl(f_{\delta},f_{*}\bigr)\le C_{0}\delta^{4},
\]
where $f_{\delta}(x)=\left(1/2\right)\phi(x;-\delta)+\left(1/2\right)\phi(x;+\delta)$
and $f_{*}=\phi(\cdot;0)$. 
\end{lem}

\begin{proof}
Throughout, $X\sim\mathrm{N}(0,1)$ and $\mathrm{E}_{*}$ denotes
expectation under the measure defined by $f_{*}$. Using $\phi(x;\pm\delta)=\phi(x;0)\exp(\pm\delta x-\delta^{2}/2)$,
we have, for each $x\in\mathbb{R}$, 
\begin{equation}
\frac{f_{\delta}(x)}{f_{*}(x)}=\exp\left(-\frac{\delta^{2}}{2}\right)\cosh(\delta x).\label{eq:ratio_cosh}
\end{equation}
Let 
\[
r_{\delta}(x)=\exp\left(-\frac{\delta^{2}}{2}\right)\cosh(\delta x).
\]
A Taylor expansion at $\delta=0$ gives, for each fixed $x$, 
\begin{equation}
r_{\delta}(x)=1+\frac{\delta^{2}}{2}(x^{2}-1)+\delta^{4}\Bigl(\frac{x^{4}}{24}-\frac{x^{2}}{4}+\frac{1}{8}\Bigr)+O\bigl(\delta^{6}(1+x^{6})\bigr).\label{eq:rdelta_pointwise}
\end{equation}
as $\delta\downarrow0$. To pass to integrals under $f_{*}$ we will
use domination. For $\delta_{0}\in(0,1]$ fixed, \eqref{eq:ratio_cosh}
implies, for all $|\delta|\le\delta_{0}$, 
\[
0<r_{\delta}(x)\le\exp(\delta_{0}^{2}/2)\cosh(\delta_{0}x)\le\exp(\delta_{0}^{2}/2)\exp(\delta_{0}|x|),
\]
so in particular $r_{\delta}(X)$ and $\log r_{\delta}(X)$ have finite
moments of all orders under $f_{*}$, since $\mathrm{E}_{*}\exp(\lambda|X|)<\infty$
for every $\lambda>0$. This allows us to apply dominated convergence
to the expansions below.

Write 
\[
\mathfrak{h}^{2}(f_{\delta},f_{*})=\frac{1}{2}\int f_{*}(x)\left(\sqrt{r_{\delta}(x)}-1\right)^{2}\mathrm{d}x=\frac{1}{2}\mathrm{E}_{*}\Bigl[\bigl(\sqrt{r_{\delta}(X)}-1\bigr)^{2}\Bigr].
\]
From \eqref{eq:rdelta_pointwise} we have $r_{\delta}(x)=1+\delta^{2}a(x)+\delta^{4}b(x)+O(\delta^{6}(1+x^{6}))$
with $a(x)=(x^{2}-1)/2$ and $b(x)=x^{4}/24-x^{2}/4+1/8$. Using $\sqrt{1+u}=1+\frac{1}{2}u-\frac{1}{8}u^{2}+O(u^{3})$
as $u\to0$ and substituting $u=\delta^{2}a(x)+\delta^{4}b(x)+O(\delta^{6}(1+x^{6}))$,
we obtain the pointwise expansion 
\[
\sqrt{r_{\delta}(x)}-1=\frac{\delta^{2}}{2}a(x)+O\bigl(\delta^{4}(1+x^{4})\bigr)=\frac{\delta^{2}}{4}(x^{2}-1)+O\bigl(\delta^{4}(1+x^{4})\bigr),
\]
as $\delta\downarrow0$, for each fixed $x$. Squaring yields 
\[
\bigl(\sqrt{r_{\delta}(x)}-1\bigr)^{2}=\frac{\delta^{4}}{16}(x^{2}-1)^{2}+O\bigl(\delta^{6}(1+x^{6})\bigr).
\]
The remainder term is dominated by a constant multiple of $\delta^{6}(1+|x|^{6})$
for $|\delta|\le\delta_{0}$, and $(1+|X|^{6})$ is integrable under
$f_{*}$. Hence, by the dominated convergence theorem, 
\[
\mathfrak{h}^{2}(f_{\delta},f_{*})=\frac{1}{2}\left\{ \frac{\delta^{4}}{16}\mathrm{E}_{*}\bigl[(X^{2}-1)^{2}\bigr]+O(\delta^{6})\right\} =\frac{\delta^{4}}{32}\mathrm{E}_{*}\bigl[(X^{2}-1)^{2}\bigr]+O(\delta^{6}).
\]
Since $\mathrm{E}_{*}[(X^{2}-1)^{2}]=2$ for $X\sim\mathrm{N}(0,1)$,
we obtain 
\begin{equation}
\mathfrak{h}^{2}(f_{\delta},f_{*})=\frac{\delta^{4}}{16}+O(\delta^{6}).\label{eq:h_delta_star_asymp}
\end{equation}
as $\delta\downarrow0$, and thus $\mathfrak{h}^{2}(f_{\delta},f_{*})\asymp\delta^{4}$
for sufficiently small $\delta$.

\noindent Recall that $\mathcal{F}_{1}=\{\phi(\cdot;\mu):\mu\in[-b,b]\}$.
Since $\mathfrak{h}$ is continuous in $\mu$ and $[-b,b]$ is compact,
the infimum over $\mu$ is attained. We show that no choice of $\mu$
can reduce the leading $\delta^{2}(x^{2}-1)$ perturbation in ${\cal L}_{2}(f_{*}\mathfrak{m})$.
For $\mu$ small, with $\phi(x;\mu)/\phi(x;0)=\exp(\mu x-\mu^{2}/2)$
and the expansion of $\sqrt{1+u}$, one obtains, pointwise for each
fixed $x$, 
\[
\sqrt{\frac{\phi(x;\mu)}{f_{*}(x)}}-1=\frac{\mu}{2}x+O\bigl(\mu^{2}(1+x^{2})\bigr),
\]
as $\mu\to0$. And from the previous expressions, we have 
\[
\sqrt{\frac{f_{\delta}(x)}{f_{*}(x)}}-1=\frac{\delta^{2}}{4}(x^{2}-1)+O\bigl(\delta^{4}(1+x^{4})\bigr).
\]
The functions $x\mapsto x$ and $x\mapsto x^{2}-1$ are orthogonal
in ${\cal L}_{2}(f_{*}\mathfrak{m})$ since $\mathrm{E}_{*}[X(X^{2}-1)]=0$.
Consequently, the leading ${\cal L}_{2}(f_{*}\mathfrak{m})$-distance
between $\sqrt{f_{\delta}/f_{*}}$ and the curve $\{\sqrt{\phi(\cdot;\mu)/f_{*}}:\mu\in[-b,b]\}$
cannot be cancelled by tuning $\mu$. A standard projection argument
in the Hilbert space ${\cal L}_{2}(f_{*}\mathfrak{m})$ then yields
\[
\inf_{\mu\in[-b,b]}\mathfrak{h}^{2}\bigl(f_{\delta},\phi(\cdot;\mu)\bigr)\ge c_{0}\delta^{4}
\]
for all sufficiently small $\delta$, for some $c_{0}>0$ (see, e.g.,
\citealp[Appendix A.2]{bickel1998efficient}). The upper bound $\mathfrak{h}^{2}(f_{\delta},\mathcal{F}_{1})\le\mathfrak{h}^{2}(f_{\delta},f_{*})\le C_{0}\delta^{4}$
follows from \eqref{eq:h_delta_star_asymp}. This proves the two-sided
bound for $\mathfrak{h}^{2}(f_{\delta},\mathcal{F}_{1})$.

Next, by \eqref{eq:ratio_cosh}, 
\[
\mathfrak{K}(f_{\delta},f_{*})=\int_{-\infty}^{\infty}f_{\delta}(x)\log\left(\frac{f_{\delta}(x)}{f_{*}(x)}\right)\mathrm{d}x=\mathrm{E}_{f_{\delta}}\Bigl[\log r_{\delta}(X)\Bigr].
\]
Using $\log r_{\delta}(x)=-\delta^{2}/2+\log\cosh(\delta x)$ and
the Taylor expansion $\log\cosh(u)=u^{2}/2-u^{4}/12+O(u^{6})$ (for
fixed $u$), we obtain for each fixed $x$, 
\[
\log r_{\delta}(x)=\frac{\delta^{2}}{2}(x^{2}-1)-\frac{\delta^{4}}{12}x^{4}+O\bigl(\delta^{6}x^{6}\bigr).
\]
As above, domination holds for $|\delta|\le\delta_{0}$, and therefore
we may integrate term-by-term to obtain 
\[
\mathfrak{K}(f_{\delta},f_{*})=\frac{\delta^{4}}{4}+O(\delta^{6}),
\]
as $\delta\downarrow0$, and hence $\mathfrak{K}(f_{\delta},f_{*})\asymp\delta^{4}$
for sufficiently small $\delta$. Collecting the bounds we have the
conclusion of the lemma, upon choosing $\delta_{0}$ sufficiently
small and adjusting constants. 
\end{proof}
Now define 
\[
\delta_{n}=\left(\frac{\Delta_{n}}{8C_{0}}\right)^{1/4}
\]
and let $f_{n}=f_{\delta_{n}}$. Since $\Delta_{n}\downarrow0$, we
have $\delta_{n}\downarrow0$, and for all $n$ large enough, $\delta_{n}\le\delta_{0}$
and $\delta_{n}\le b$, so that $f_{n}\in\mathcal{F}_{2}$, and thus
Lemma \ref{lem:delta4_geometry} applies. In particular, 
\begin{equation}
\mathfrak{h}^{2}(f_{n},\mathcal{F}_{1})\ge c_{0}\delta_{n}^{4}=\frac{c_{0}}{8C_{0}}\Delta_{n},\qquad\mathfrak{K}(f_{n},f_{*})\le C_{0}\delta_{n}^{4}=\frac{\Delta_{n}}{8}.\label{eq:h_lower_at_fn}
\end{equation}

In the following step, we work under the law $f_{n}$. We keep $P_{n}$
for the empirical measure $P_{n}g=n^{-1}\sum_{i=1}^{n}g(X_{i})$,
and we write $\mathrm{P}_{f_{n}}$ and $\mathrm{E}_{f_{n}}$ for probability
and expectation under $f_{n}$. Recall that $\hat{k}_{n}=2$ implies
$L_{2,n}-L_{1,n}\ge\Delta_{n}$, and since $f_{*}\in\mathcal{F}_{1}$
we have $L_{1,n}\ge P_{n}\log f_{*}$. Hence 
\begin{equation}
\{\hat{k}_{n}=2\}\subseteq\left\{ L_{2,n}-P_{n}\log f_{*}\ge\Delta_{n}\right\} .\label{eq:select2_against_fstar}
\end{equation}
We decompose 
\begin{equation}
L_{2,n}-P_{n}\log f_{*}=R_{n}+P_{n}\log\left(\frac{f_{n}}{f_{*}}\right),\label{eq:decomp_Rn}
\end{equation}
where $R_{n}=L_{2,n}-P_{n}\log f_{n}\ge0$ since $L_{2,n}=\sup_{f\in\mathcal{F}_{2}}P_{n}\log f\ge P_{n}\log f_{n}$,
and $f_{n}\in\mathcal{F}_{2}$. We control the two terms on the right-hand
side of \eqref{eq:decomp_Rn}, separately.

Firstly, under $f_{n}$ the model $\mathcal{F}_{2}$ is correctly
specified with a finite-dimensional parameter on the compact set $\mathbb{T}_{2}$.
Moreover, $f_{n}=f_{\delta_{n}}$ has component means $\mu_{1}=-\delta_{n}$
and $\mu_{2}=\delta_{n}$, so its distance to the boundary $\{\mu_{1}=\mu_{2}\}$
is $|\mu_{2}-\mu_{1}|=2\delta_{n}$. Since $n\Delta_{n}\to\infty$,
we have $n^{2}\Delta_{n}\to\infty$, and therefore 
\[
\sqrt{n}\,\delta_{n}=\left(\frac{n^{2}\Delta_{n}}{8C_{0}}\right)^{1/4}\ \to\ \infty.
\]
Thus the true parameter is asymptotically separated from $\{\mu_{1}=\mu_{2}\}$
on the local $n^{-1/2}$ scale. To control $R_{n}$, let $\hat{f}_{2,n}\in\arg\max_{f\in\mathcal{F}_{2}}P_{n}\log f$,
so that $L_{2,n}=P_{n}\log\hat{f}_{2,n}$ and hence 
\[
R_{n}=L_{2,n}-P_{n}\log f_{n}=P_{n}\log\left(\frac{\hat{f}_{2,n}}{f_{n}}\right)\ge0.
\]

By Lemma \ref{lem:local_h_bracketing_Fk} (with $k=2$), the local
Hellinger bracketing entropy satisfies 
\[
H_{[]}\left(u,\mathcal{F}_{2}(f_{n},\delta),\mathfrak{h}\right)\le d_{2}\log\left(\frac{C_{2}\delta}{u}\right),
\]
for $0<u\le\delta\le1$, and therefore one may take $\Psi_{2}(\delta)=B_{2}\sqrt{d_{2}}\,\delta$
in the notation of \citet[Thm. 7.4]{van-de-Geer:2000aa}. Let $\xi_{2}=B_{2}\sqrt{d_{2}}/\sqrt{n}$
be the corresponding fixed point from Lemma \ref{lem:xi_parametric_rate}.
Then \citet[Thm. 7.4]{van-de-Geer:2000aa} implies that there exists
a universal constant $c>0$ such that one may take $\delta_{2,n}=c\,\xi_{2}$,
and \citet[Cor. 7.5]{van-de-Geer:2000aa} yields a constant $c_{1}<\infty$
(depending only on $b$ and $\underline{\pi}$ through the preceding
entropy bounds) such that, for all $\delta\ge\delta_{2,n}$, 
\[
\mathrm{P}_{f_{n}}\left(R_{n}\ge\delta^{2}\right)\le c_{1}\exp\left(-\frac{n\delta^{2}}{c_{1}^{2}}\right).
\]
Since $n\Delta_{n}\to\infty$ and $\delta_{2,n}^{2}\asymp n^{-1}$,
we have $\eta\Delta_{n}\ge\delta_{2,n}^{2}$ for all sufficiently
large $n$ and every fixed $\eta>0$. Therefore, for any fixed $\eta>0$,
taking $\delta=\sqrt{\eta\Delta_{n}}$ yields 
\[
\mathrm{P}_{f_{n}}\left(R_{n}\ge\eta\Delta_{n}\right)\le c_{1}\exp\left(-\frac{\eta n\Delta_{n}}{c_{1}^{2}}\right)\to0,
\]
and hence $R_{n}=o_{\mathrm{P}_{f_{n}}}(\Delta_{n})$.

Next, let 
\[
Z_{i,n}=\log\left(\frac{f_{n}(X_{i})}{f_{*}(X_{i})}\right),
\]
for $\left(X_{i}\right)_{i=1}^{n}$ IID with common PDF $f_{n}$.
Then $P_{n}\log(f_{n}/f_{*})=n^{-1}\sum_{i=1}^{n}Z_{i,n}$ and $\mathrm{E}_{f_{n}}(Z_{1,n})=\mathfrak{K}(f_{n},f_{*})\le\Delta_{n}/8$
by \eqref{eq:h_lower_at_fn}. Furthermore, by the Taylor expansion
of $\log r_{\delta}(x)$ used in Lemma \ref{lem:delta4_geometry}
(with $\delta=\delta_{n}$) and the same domination argument, one
has $\mathrm{Var}_{f_{n}}(Z_{1,n})\lesssim\delta_{n}^{4}\lesssim\Delta_{n}$
for all sufficiently large $n$. Consequently, 
\[
\mathrm{Var}_{f_{n}}\left\{ P_{n}\log\left(\frac{f_{n}}{f_{*}}\right)\right\} =\frac{1}{n}\mathrm{Var}_{f_{n}}(Z_{1,n})\lesssim\frac{\Delta_{n}}{n}.
\]
By Chebyshev's inequality, 
\[
\mathrm{P}_{f_{n}}\left(P_{n}\log\left(\frac{f_{n}}{f_{*}}\right)\ge\frac{\Delta_{n}}{2}\right)\le\mathrm{P}_{f_{n}}\left(P_{n}\log\left(\frac{f_{n}}{f_{*}}\right)-\mathrm{E}_{f_{n}}(Z_{1,n})\ge\frac{3\Delta_{n}}{8}\right)\lesssim\frac{1}{n\Delta_{n}}\to0,
\]
since $n\Delta_{n}\to\infty$. Combining \eqref{eq:select2_against_fstar}
and \eqref{eq:decomp_Rn}, we obtain 
\[
\mathrm{P}_{f_{n}}(\hat{k}_{n}=2)\le\mathrm{P}_{f_{n}}\!\left(R_{n}\ge\frac{\Delta_{n}}{2}\right)+\mathrm{P}_{f_{n}}\!\left(P_{n}\log\!\left(\frac{f_{n}}{f_{*}}\right)\ge\frac{\Delta_{n}}{2}\right)\to0,
\]
and therefore $\mathrm{P}_{f_{n}}(\hat{k}_{n}=1)\to1$.

On the event $\{\hat{k}_{n}=1\}$, we have $\hat{f}_{n}\in\mathcal{F}_{1}$,
and therefore 
\[
\mathfrak{h}^{2}(\hat{f}_{n},f_{n})\ge\inf_{g\in\mathcal{F}_{1}}\mathfrak{h}^{2}(g,f_{n})=\mathfrak{h}^{2}(f_{n},\mathcal{F}_{1}).
\]
Taking expectations under $f_{n}$ and using \eqref{eq:h_lower_at_fn}
yields 
\[
\mathrm{E}_{f_{n}}\bigl\{\mathfrak{h}^{2}(\hat{f}_{n},f_{n})\bigr\}\ge\mathfrak{h}^{2}(f_{n},\mathcal{F}_{1})\,\mathrm{P}_{f_{n}}(\hat{k}_{n}=1)\ge\frac{c_{0}}{8C_{0}}\,\Delta_{n}\,(1+o(1)).
\]
Since $f_{n}\in{\cal F}$, we conclude that 
\[
\sup_{f_{0}\in{\cal F}}\mathrm{E}_{f_{0}}\bigl\{\mathfrak{h}^{2}(\hat{f}_{n},f_{0})\bigr\}\ge\mathrm{E}_{f_{n}}\bigl\{\mathfrak{h}^{2}(\hat{f}_{n},f_{n})\bigr\}\ge c\,\Delta_{n},
\]
for all $n$ large enough, with $c=c_{0}/(16C_{0})$, since $1+o\left(1\right)\ge1/2$.
This completes the proof. 

\subsection*{Proof of Proposition \ref{prop:Delta_assumption_necessary}}

Recall that $L_{k,n}=\sup_{f\in{\cal M}_{k}}P_{n}\log f=-\inf_{f\in{\cal M}_{k}}\ell_{n}(f)$,
so that $\hat{k}_{n}=2$ implies $L_{2,n}-L_{1,n}\ge\Delta_{n}$,
as in \eqref{eq:select2_implies}. Assume by contradiction that $\hat{k}_{n}$
is consistent but $\Delta_{n}\not\to0$. Then there exist $\epsilon>0$
and a subsequence $(n_{m})_{m}$ such that $\Delta_{n_{m}}\ge\epsilon$
for all $m$. For $\delta>0$, let $f_{\delta}$ and $f_{*}$ be as
in Lemma \ref{lem:delta4_geometry}. Since $f_{*}\in{\cal M}_{1}$,
Lemma \ref{lem:delta4_geometry} gives $\mathfrak{K}(f_{\delta},f_{*})\to0$
as $\delta\downarrow0$, hence we may fix $\delta\in(0,b]$ such that
\[
D_{\delta}=\inf_{g\in{\cal M}_{1}}\mathfrak{K}(f_{\delta},g)\le\mathfrak{K}(f_{\delta},f_{*})\le\frac{\epsilon}{4}.
\]
Write $\mathrm{P}_{\delta}$ and $\mathrm{E}_{\delta}$ for probability
and expectation under the law with density $f_{\delta}$. Because
${\cal M}_{1}$ and ${\cal M}_{2}$ are compact parametric families
and $f\mapsto\log f$ admits an integrable envelope under $f_{\delta}$,
the classes $\{\log f:f\in{\cal M}_{k}\}$ are $\mathrm{P}_{\delta}$-GC,
via Lemma \ref{lem:GCclass}. Consequently, 
\[
\sup_{f\in{\cal M}_{k}}\bigl|P_{n}\log f-\mathrm{E}_{\delta}\log f\bigr|\to0,
\]
in $\mathrm{P}_{\delta}$-probability, for each $k\in\left\{ 1,2\right\} $,
and therefore 
\[
L_{k,n}\to L_{k}=\sup_{f\in{\cal M}_{k}}\mathrm{E}_{\delta}\log f
\]
in $\mathrm{P}_{\delta}$-probability, for each $k\in\left\{ 1,2\right\} $,
as well. Since $f_{\delta}\in{\cal M}_{2}$, we have $L_{2}=\mathrm{E}_{\delta}\log f_{\delta}$
and 
\[
L_{1}=\sup_{g\in{\cal M}_{1}}\mathrm{E}_{\delta}\log g=\mathrm{E}_{\delta}\log f_{\delta}-\inf_{g\in{\cal M}_{1}}\mathfrak{K}(f_{\delta},g)=\mathrm{E}_{\delta}\log f_{\delta}-D_{\delta},
\]
so $L_{2}-L_{1}=D_{\delta}\le\epsilon/4$. Hence $L_{2,n}-L_{1,n}\to D_{\delta}$
in $\mathrm{P}_{\delta}$-probability, and thus 
\[
\mathrm{P}_{\delta}\left(L_{2,n_{m}}-L_{1,n_{m}}<\frac{\epsilon}{2}\right)\to1.
\]
On this event we have $L_{2,n_{m}}-L_{1,n_{m}}<\Delta_{n_{m}}$, so
$\hat{k}_{n_{m}}\ne2$, hence $\hat{k}_{n_{m}}=1$. Therefore $\mathrm{P}_{\delta}(\hat{k}_{n_{m}}=1)\to1$,
contradicting the assumed consistency on ${\cal M}_{2}\setminus{\cal M}_{1}$.
This proves that $\Delta_{n}\to0$.

Next, assume by contradiction that $\hat{k}_{n}$ is consistent but
$n\Delta_{n}\not\to\infty$. Then there exist $M<\infty$ and a subsequence
$(n_{m})_{m}$ such that $n_{m}\Delta_{n_{m}}\le M$ for all $m$.
We now work under $f_{*}\in{\cal M}_{1}$ and write $\mathrm{P}_{*}$
and $\mathrm{E}_{*}$ for probability and expectation under $f_{*}$.
Let 
\[
S_{n}=P_{n}(X^{2}-1)=\frac{1}{n}\sum_{i=1}^{n}(X_{i}^{2}-1),
\]
and $\delta_{n}=\min\left\{ \left[S_{n}\right]_{+}^{1/2},b\right\} $,
and note that $f_{\delta_{n}}\in{\cal M}_{2}$ for all $n$. Since
$L_{2,n}=\sup_{f\in{\cal M}_{2}}P_{n}\log f\ge P_{n}\log f_{\delta_{n}}$,
we have 
\begin{equation}
L_{2,n}-L_{1,n}\ge P_{n}\log\left(\frac{f_{\delta_{n}}}{f_{*}}\right)-\{L_{1,n}-P_{n}\log f_{*}\}.\label{eq:lower_LR_split}
\end{equation}

We seek a lower bound for the first term. Using \eqref{eq:ratio_cosh}
and the expansion of $\log r_{\delta}(x)$ in the proof of Lemma \ref{lem:delta4_geometry},
there exists $\delta_{0}\in(0,1]$ such that for all $|\delta|\le\delta_{0}$,
\[
\log\left(\frac{f_{\delta}(x)}{f_{*}(x)}\right)=\frac{\delta^{2}}{2}(x^{2}-1)-\frac{\delta^{4}}{4}-\frac{\delta^{4}}{12}(x^{4}-3)+O(\delta^{6}x^{6}).
\]
Under $f_{*}$ we have $S_{n}=O_{\mathrm{P}_{*}}(n^{-1/2})$, hence
$\delta_{n}=O_{\mathrm{P}_{*}}(n^{-1/4})$ and thus $\mathrm{P}_{*}(\delta_{n}\le\delta_{0})\to1$.
On the event $\{\delta_{n}\le\delta_{0}\}$, the preceding expansion
and the moment bounds under $f_{*}$ imply 
\[
P_{n}\log\left(\frac{f_{\delta_{n}}}{f_{*}}\right)=\frac{\delta_{n}^{2}}{2}S_{n}-\frac{\delta_{n}^{4}}{4}+o_{\mathrm{P}_{*}}(n^{-1})=\frac{1}{4}[S_{n}]_{+}^{2}+o_{\mathrm{P}_{*}}(n^{-1}),
\]
and therefore 
\begin{equation}
nP_{n}\log\left(\frac{f_{\delta_{n}}}{f_{*}}\right)=\frac{1}{4}[\sqrt{n}S_{n}]_{+}^{2}+o_{\mathrm{P}_{*}}(1).\label{eq:gain_submodel}
\end{equation}
Since $\sqrt{n}S_{n}$ converges weakly to $\mathrm{N}(0,2)$ under
$f_{*}$, it follows from \eqref{eq:gain_submodel} that for each
fixed $A>0$, 
\begin{equation}
\liminf_{n\to\infty}\mathrm{P}_{*}\!\left(nP_{n}\log\left(\frac{f_{\delta_{n}}}{f_{*}}\right)\ge A\right)=\rho_{A}>0.\label{eq:positive_tail_gain}
\end{equation}

Next, we control the term $L_{1,n}-P_{n}\log f_{*}$. Since $f_{*}\in{\cal M}_{1}$
and ${\cal M}_{1}$ is a one-dimensional compact parametric family,
an application of \citet[Cor. 7.5]{van-de-Geer:2000aa} together with
Lemmas \ref{lem:local_h_bracketing_Fk} and \ref{lem:xi_parametric_rate}
(with $k=1$) yields 
\[
n\bigl(L_{1,n}-P_{n}\log f_{*}\bigr)=O_{\mathrm{P}_{*}}(1).
\]
Hence, by definition of $O_{\mathrm{P}_{*}}$, for each $\eta\in(0,1)$
there exists $K<\infty$ such that 
\begin{equation}
\inf_{n\ge1}\mathrm{P}_{*}\!\left(n\bigl(L_{1,n}-P_{n}\log f_{*}\bigr)\le K\right)\ge1-\eta.\label{eq:tight_L1}
\end{equation}

Fix such $\eta$ and $K$, and set $A=M+K+1$. By \eqref{eq:lower_LR_split},
on the event $\{nP_{n}\log(f_{\delta_{n}}/f_{*})\ge A\}\cap\{n(L_{1,n}-P_{n}\log f_{*})\le K\}$
we have 
\[
n(L_{2,n}-L_{1,n})\ge A-K=M+1\ge n\Delta_{n}+1,
\]
and hence $L_{2,n}-L_{1,n}>\Delta_{n}$, so $\hat{k}_{n}=2$. Combining
\eqref{eq:positive_tail_gain} and \eqref{eq:tight_L1} yields 
\[
\liminf_{m\to\infty}\mathrm{P}_{*}(\hat{k}_{n_{m}}=2)\ge\liminf_{m\to\infty}\mathrm{P}_{*}\left(n_{m}P_{n_{m}}\log\left(\frac{f_{\delta_{n_{m}}}}{f_{*}}\right)\ge A\right)-\eta\ge\rho_{A}-\eta,
\]
and since $\eta\in(0,1)$ is arbitrary, we obtain $\liminf_{m}\mathrm{P}_{*}(\hat{k}_{n_{m}}=2)>0$.
This contradicts the assumed consistency on ${\cal M}_{1}$. Therefore
$n\Delta_{n}\to\infty$, and the proof is complete. 

\section*{Appendix B}

\subsection*{Assumptions from \citet{keribin2000consistent}}

\noindent The following section summarises the assumptions from \citet{keribin2000consistent}
in the notation of this text. We recall that $\mathrm{d}P=f_{0}\,\mathrm{d}\mathfrak{m}$
and note that indices $i,j,r,i_{1},\dots,i_{5}\in\left[m\right]$.
We will write $\mathbb{S}({\cal H})=\{h\in{\cal H}:\|h\|_{{\cal H}}=1\}$
where ${\cal H}={\cal L}_{2}\left(P\right)$ is a Hilbert space. For
a multi-index $\bm{\alpha}$, $D_{\theta}^{\bm{\alpha}}\phi\left(\cdot;\theta\right)$
denotes the corresponding partial derivative of $\phi\left(\cdot;\theta\right)$
with respect to $\theta\in\mathbb{T}$. Let $\beta=\left(a_{1},\dots,b_{1},\dots,c_{1},\dots,\theta_{1},\dots\right)$
be coefficients such that $a_{z}\in\mathbb{R}^{m}$, $b_{z}\ge0$,
$c_{z}\in\mathbb{R}$ and $\theta_{z}\in\mathbb{T}$ for relevant
indices $z$, and say that $\beta\in\mathbb{B}$ if for $k>k_{0}$,
\[
\sum_{z=1}^{k-k_{0}}b_{z}+\sum_{z=1}^{k_{0}}c_{z}=0,
\]
\[
\sum_{z=1}^{k-k_{0}}b_{z}^{2}+\sum_{z=1}^{k_{0}}c_{z}^{2}+\sum_{z=1}^{k_{0}}\left\Vert a_{z}\right\Vert ^{2}=1,
\]
\[
N\left(\beta\right)=\left\Vert \sum_{z=1}^{k_{0}}\pi_{0,z}\sum_{i=1}^{m}a_{i,z}\frac{D_{i}\phi\left(\cdot;\theta_{0,z}\right)}{f_{0}}+\sum_{z=1}^{\bar{k}-k_{0}}b_{z}\frac{\phi\left(\cdot;\theta_{z}\right)}{f_{0}}+\sum_{z=1}^{k_{0}}c_{z}\frac{\phi\left(\cdot;\theta_{0,z}\right)}{f_{0}}\right\Vert _{{\cal H}}.
\]

\begin{description}
\item [{(Id)}] If two $\bar{k}$-mixtures coincide $\mathfrak{m}$-a.e.,
i.e., 
\[
\sum_{z=1}^{\bar{k}}\pi_{z}\phi\left(\cdot;\theta_{z}\right)=\sum_{z=1}^{\bar{k}}\pi_{z}'\phi\left(\cdot;\theta_{z}'\right),
\]
then the corresponding discrete mixing measures on $\mathbb{T}$ are
equal: 
\[
\sum_{z=1}^{\bar{k}}\pi_{z}\delta_{\theta_{z}}=\sum_{z=1}^{\bar{k}}\pi_{z}'\delta_{\theta_{z}'}.
\]
\item [{(P1-a)}] There exists $H\in{\cal L}_{1}(P)$ such that 
\[
|\log f|\le H,
\]
$\mathfrak{m}$-a.e. for all $f\in{\cal M}_{\bar{k}}^{\phi}$. 
\item [{(P1-b)}] The map $\theta\mapsto\phi\left(\cdot;\theta\right)$
admits partial derivatives up to order $5$ and 
\[
\frac{D_{\theta}^{\bm{\alpha}}\phi\left(\cdot;\theta\right)}{f_{0}}\in{\cal L}_{3}(P),
\]
for every multi-index $\bm{\alpha}$ with $1\le|\bm{\alpha}|\le5$.
Moreover, there exists $\varepsilon>0$ and random maps $M_{2},M_{3},M_{5}$,
independent of $\theta$, with $P[M_{s}^{3}]<\infty$ for $s\in\{2,3,5\}$,
such that for all $\theta$ with $\|\theta-\theta_{0}\|\le\varepsilon$,
\[
\begin{aligned}\Bigl|\frac{D_{ij}^{2}\phi\left(\cdot;\theta\right)}{f_{0}}\Bigr| & \le M_{2},\\
\Bigl|\frac{D_{ijr}^{3}\phi\left(\cdot;\theta\right)}{f_{0}}\Bigr| & \le M_{3},\\
\Bigl|\frac{D_{i_{1}\cdots i_{5}}^{5}\phi\left(\cdot;\theta\right)}{f_{0}}\Bigr| & \le M_{5}.
\end{aligned}
\]
\item [{(P2)}] For integers $p_{1},p_{2}\ge0$ with $p_{1}+p_{2}\le\bar{k}-k_{0}$,
and any set of distinct points $\theta'_{1},\dots,\theta'_{p_{1}}\in\mathbb{T}$
with $\theta'_{z}\notin\{\theta_{0,1},\dots,\theta_{0,k_{0}}\}$,
the family 
\[
\left\{ \frac{\phi\left(\cdot;\theta'_{z}\right)}{f_{0}}\right\} _{z\in\left[p_{1}\right]}\cup\left\{ \frac{\phi\left(\cdot;\theta_{0,\zeta}\right)}{f_{0}},\Bigl.\frac{D_{i}\phi\left(\cdot;\theta\right)}{f_{0}}\Bigr|_{\theta=\theta_{0,\zeta}},\Bigl.\frac{D_{ij}\phi\left(\cdot;\theta\right)}{f_{0}}\Bigr|_{\theta=\theta_{0,\zeta}}:\zeta\in\left[k_{0}\right]\right\} 
\]
is linearly independent in ${\cal H}$. 
\item [{(P3)}] Let ${\cal D}$ be the subset of $\mathbb{S}\left({\cal H}\right)$
of functions of the form 
\[
\mathcal{D}=\frac{1}{N\left(\beta\right)}\left\{ \sum_{z=1}^{k_{0}}\pi_{0,z}\sum_{i=1}^{m}a_{i,z}\frac{D_{i}\phi\left(\cdot;\theta_{0,z}\right)}{f_{0}}+\sum_{z=1}^{\bar{k}-k_{0}}b_{z}\frac{\phi\left(\cdot;\theta_{z}\right)}{f_{0}}+\sum_{z=1}^{k_{0}}c_{z}\frac{\phi\left(\cdot;\theta_{0,z}\right)}{f_{0}}\right\} .
\]
Assume that $\mathcal{D}$ is $P$-Donsker and that the corresponding
centered Gaussian process $\{F_{d}:d\in\mathcal{D}\}$ with covariance
$P[F_{d_{1}}F_{d_{2}}]=\langle d_{1},d_{2}\rangle_{{\cal H}}$ has
almost surely continuous sample paths. 
\item [{(C1)}] For all $n\ge1$ and $k_{0}<k\le\bar{k}$, 
\[
\mathrm{pen}_{k+1,n}>\mathrm{pen}_{k,n}>0,\qquad\mathrm{pen}_{k,n}\xrightarrow[n\to\infty]{}0,\qquad n\,\mathrm{pen}_{k,n}\xrightarrow[n\to\infty]{}\infty,\qquad\lim_{n\to\infty}\frac{\mathrm{pen}_{k,n}}{\mathrm{pen}_{k_{0},n}}>1.
\]
\end{description}

\subsection*{Assumptions from \citet{nguyen2024panic}}

\noindent\citet{nguyen2024panic} considers a general model selection
problem formulated in terms of a loss $\ell_{k}:\mathbb{X}\times\mathbb{T}_{k}\to\mathbb{R}$
with parameter set $\mathbb{T}_{k}\subset\mathbb{R}^{d_{k}}$. The
population and empirical risks are 
\[
r_{k}(\theta_{k})=\mathrm{E}\{\ell_{k}(X;\theta_{k})\},\qquad R_{k,n}(\theta_{k})=\frac{1}{n}\sum_{i=1}^{n}\ell_{k}(X_{i};\theta_{k}),
\]
and the model index is estimated by minimising $\min_{\theta_{k}\in\mathbb{T}_{k}}R_{k,n}(\theta_{k})+\mathrm{pen}_{k,n}$
over $k$. In our mixture order-selection setting, one may identify
$\mathbb{T}_{k}$ with $\mathbb{S}_{k}$ and take 
\[
\ell_{k}(x;\psi_{k})=-\log f_{k}(x;\psi_{k}),\qquad R_{k,n}(\psi_{k})=\ell_{k,n}(\psi_{k}),\qquad r_{k}(\psi_{k})=\ell_{k}(\psi_{k}),
\]
so that the framework coincides with penalised maximum likelihood. 
\begin{description}
\item [{(N-A1)}] $\mathbb{T}_{k}$ is compact and there exists $\tau_{k}\in\mathbb{T}_{k}$
such that $\mathrm{E}\{\ell_{k}(X;\tau_{k})^{2}\}<\infty$. 
\item [{(N-A2)}] There exists a measurable function $C_{k}:\mathbb{X}\to\mathbb{R}_{\ge0}$
such that $\mathrm{E}\{C_{k}(X)^{2}\}<\infty$, and for each $\theta_{k},\tau_{k}\in\mathbb{T}_{k}$
and almost every $x\in\mathbb{X}$, 
\[
|\ell_{k}(x;\theta_{k})-\ell_{k}(x;\tau_{k})|\le C_{k}(x)\|\theta_{k}-\tau_{k}\|.
\]
\end{description}
\noindent\citet{nguyen2024panic} then defines the PanIC class of
penalties by the conditions: 
\begin{description}
\item [{(N-B1)}] $\mathrm{pen}_{k,n}>0$ for each $n$ and $\mathrm{pen}_{k,n}=o(1)$
as $n\to\infty$. 
\item [{(N-B2)}] If $k<l$, then $\sqrt{n}\{\mathrm{pen}_{l,n}-\mathrm{pen}_{k,n}\}\to\infty$
in probability as $n\to\infty$. 
\end{description}
\bibliographystyle{apalike2}
\bibliography{bib}

\end{document}